\documentclass[10pt,a4paper]{amsart}%{tran-l}
\usepackage[margin=15mm]{geometry}
\usepackage{amsmath,amssymb,amsthm,amscd,ytableau}
\usepackage{epsfig,url,graphics}
\usepackage{hyperref}
\usepackage[numbers]{natbib}
\usepackage[all]{xypic}

\usepackage{xcolor}

\newtheorem{theorem}{Theorem}%[section]
\newtheorem{corollary}{Corollary}%[section]
\newtheorem{definition}{Definition}%[section]
\newtheorem{lemma}{Lemma}%[section]
\newtheorem{proposition}{Proposition}%[section]
\newtheorem{remark}{Remark}%[section]
\newcommand{\Z}{\mathbb{Z}}

\newcommand{\C}{\mathbb{C}}

\newcommand{\p}{\mathbb{P}}
\newcommand{\g}{\mathfrak g}

\newcommand{\gp}{\mathfrak p}

\newcommand{\gl}{\mathfrak{gl}}

\newcommand{\so}{\mathfrak{so}}

\newcommand{\sll}{\mathfrak{sl}}

\newcommand{\spp}{\mathfrak{sp}}

\newcommand{\Span}[1]{\left\langle#1\right\rangle}% alternative span
\DeclareMathOperator{\LL}{LGr}
\DeclareMathOperator{\rk}{\mathrm{rk}}

\DeclareMathOperator{\tr}{tr}
\DeclareMathOperator{\ad}{ad}
\DeclareMathOperator{\id}{id}
\DeclareMathOperator{\rank}{rank}

\DeclareMathOperator{\codim}{codim}

\DeclareMathOperator{\ev}{ev}%%%%%%%%
\DeclareMathOperator{\gr}{gr}%%%%%%%%
\DeclareMathOperator{\Lie}{Lie}
\DeclareMathOperator{\SO}{\mathsf{SO}}

\DeclareMathOperator{\GL}{\mathsf{GL}}

\DeclareMathOperator{\SL}{\mathsf{SL}}
\DeclareMathOperator{\PGL}{\mathsf{PGL}}
\DeclareMathOperator{\Sp}{\mathsf{Sp}}
\DeclareMathOperator{\CSp}{\mathsf{CSp}}

\DeclareMathOperator{\Gr}{Gr}
\DeclareMathOperator{\End}{End}
\DeclareMathOperator{\Fl}{Fl}%
\DeclareMathOperator{\Aut}{Aut}

\DeclareMathOperator{\Hom}{Hom}

\DeclareMathOperator{\Stab}{\mathrm{Stab}}
\newcommand{\OO}{\mathcal{O}}
\newcommand{\E}{\mathcal{E}}

\newcommand{\CC}{\mathcal{C}}

\newcommand{\Th}{^\textrm{th}}
\newcommand{\St}{^\textrm{st}}

\newcommand{\Nd}{^\textrm{nd}}

\DeclareMathOperator{\ch}{\mathrm{ch}}
\DeclareMathOperator{\Der}{\mathrm{Der}}
\def\csp{\mathfrak{csp}}
\def\fh{\mathfrak{h}}
\def\sp{\mathfrak{sp}}
\def\fp{\mathfrak{p}}
\def\fq{\mathfrak{q}}
\def\ZZ{\mathbb{Z}}
\def\pt{\mathrm{pt}}
\def\iso{\mathrm{iso}}
\def\fs{\mathfrak{s}}
\let\fg\g
\let\PP\p
\let\CC\CC
\let\LG\LL

\renewcommand{\ss}{\textrm{ss}}
\DeclareMathOperator{\sgn}{\mathrm{sgn}}
\newcommand{\Inv}{\mathsf{Inv}}

\newcommand{\Asf}{\mathsf{A}}
\newcommand{\Bsf}{\mathsf{B}}
\newcommand{\Csf}{\mathsf{C}}
\newcommand{\Dsf}{\mathsf{D}}
\newcommand{\Esf}{\mathsf{E}}
\newcommand{\Fsf}{\mathsf{F}}
\newcommand{\Gsf}{\mathsf{G}}

 \def\polhk#1{\setbox0=\hbox{#1}{\ooalign{\hidewidth
  \lower1.5ex\hbox{`}\hidewidth\crcr\unhbox0}}}

\begin{document}\sloppy

\title{Lowest degree invariant 2nd order PDEs over rational homogeneous contact manifolds}
   \author{Dmitri V. Alekseevsky}%Jan Gutt, Gianni Manno, Giovanni Moreno
 \email{dalekseevsky@iitp.ru}
  \address{Institute for Information Transmission Problems, B. Karetny
per. 19, 127051, Moscow (Russia) and University of Hradec Kralove, Rokitanskeho 62,
Hradec Kralove 50003 (Czech Republic).}
   \author{Jan Gutt}%Jan Gutt, Gianni Manno, Giovanni Moreno
 \email{jgutt@cft.edu.pl}
  \address{INdAM -- Dipartimento di Matematica ``G. L. Lagrange'', Politecnico di Torino, Corso Duca degli Abruzzi, 24, 10129 Torino, ITALY.}
 \author{Gianni Manno}
 \email{giovanni.manno@polito.it}
   \address{Dipartimento di Matematica ``G. L. Lagrange'', Politecnico di Torino, Corso Duca degli Abruzzi, 24, 10129 Torino, ITALY.}
\author{Giovanni Moreno}
 \email{gmoreno@impan.pl}
   \address{Institute of Mathematics, Polish Academy of Sciences, ul. Sniadeckich 8, 00--656 Warsaw, POLAND.}

\begin{abstract}
For each simple Lie algebra $\g$ (excluding, for trivial reasons, type $\Csf$)
we find the lowest possible degree
of an invariant $2\Nd$ order PDE over the adjoint variety in $\p\g$, a homogeneous contact manifold.
Here a PDE $F(x^i,u,u_i,u_{ij})=0$ has degree $\le d$ if $F$ is a polynomial
of degree $\le d$ in the minors of $(u_{ij})$, with coefficients functions
of the contact coordinates $x^i$, $u$, $u_i$ (e.g., Monge--Amp\`ere equations have degree 1).
For $\g$ of type $\Asf$ or $\Gsf_2$ we show that this gives all
invariant $2\Nd$ order PDEs. For $\g$ of type $\Bsf$ and $\Dsf$
we provide an explicit formula for the lowest-degree invariant $2\Nd$ order PDEs.
For $\g$ of type $\Esf$ and $\Fsf_4$
we prove uniqueness of the lowest-degree invariant $2\Nd$ order PDE; we also conjecture
that uniqueness holds in type $\Dsf$.
%
%
%There is a natural method to obtain $2\Nd$ order PDEs with prescribed simple group of symmetries $G$.  However, one observes that the so--obtained PDEs are, in many cases, of high degree.  In this paper we consider the problem of existence of lower degree PDEs with the same group of symmetries.
\end{abstract}
%
%\subjclass[2008]{30, 80}

 \date{\today}

 \maketitle
 \setcounter{tocdepth}{4}

 %\tableofcontents

\noindent\textbf{Keywords:} simple Lie algebras, adjoint variety, Lagrangian Grassmannian, second order PDEs, symmetries of PDEs, invariant theory

\medskip\noindent\textbf{MSC 2010:} 35A30, 58J70, 53C30, 53D10, 17B08, 17B10

\section{Introduction}

\subsection{Starting point}
The problem of classifying scalar second order PDEs admitting a large group of symmetries
is the basis of a very extensive research programme, originating in the work of Lie, Darboux,
Cartan and others. It is naturally broken into the sub-problems of classifying $G$-invariant
PDEs with a prescribed Lie group $G$ of symmetries. This is still a vast project in general. Using
the interpretation of $2\Nd$ order PDEs in terms of contact structures, we restrict our attention
to the case where $G$ acts transitively on the underlying contact manifold $M$. Even there we are
facing a hopelessly complex task, in particular implying the classification of homogeneous
contact $G$-manifolds for the given Lie group. The latter problem is considered by one of
the authors in \cite{Ale90}. In the present paper we shall still significantly narrow the focus,
by working in the complex holomorphic setting, assuming $G$ to be simple, and requiring
that $M$ be compact. There we can use the structure theory of simple complex Lie
algebras, and finally reduce the question to an algebraic problem in invariant theory.
Since results about the real case can be then recovered
(in a straightforward, if laborious, manner) considering suitable real forms of $G$,
from now on we shall apply the term `partial differential equation' to what is more properly its complexification (see Subsection \ref{subFirstLook} for further remarks about the real setting).

\subsection{Context}
Remarkably, it turns out that in this context---with the exception of groups $G$ of type $\Asf$ and, for trivial reasons, $\Csf$---a $G$-invariant second order PDE has precisely the Lie algebra $\fg$ of $G$ as
its \emph{local} infinitesimal symmetries at any point of $M$. This
has been observed and used by D. The in \cite{2016arXiv160308251T} to realise the simple Lie
algebras not of type $\Csf$ as infinitesimal symmetries of $2\Nd$ order PDEs, a problem with a long tradition, inaugurated by the  1893   works by Cartan and Engel,\footnote{See \cite{2016arXiv160308251T} for a more detailed historical account, as well as a broader list of references.}
 and recently recast in full generality by  P. Nurowski in the context of the  so--called \emph{parabolic contact geometries} (see \cite{MR2532439}, sec. 4.2). In this spirit, $2\Nd$ order PDEs can be thought of as additional structures on contact manifolds.

\subsection{The degree of a $2\Nd$ order PDE on a contact manifold}\label{ss:from-pdes}
A $2\Nd$ order (scalar, in $n$ independent variables and one unknown function) PDE on a $(2n+1)$--dimensional contact manifold $(M,\CC)$, $\CC$ being the contact distribution on $M$, is, roughly speaking, a first--order condition imposed on the Lagrangian submanifolds of $M$. Note that the latter are the integral $n$--dimensional submanifolds of the exterior differential system $(M,\mathcal{I}_{\CC})$, where $\mathcal{I}_{\CC}$ is the ideal of differential forms vanishing on $\CC$ (see \cite{MR1083148}). Recall that the Levi (twisted) two--form $\CC\wedge\CC\longrightarrow TM/\CC$ is nondegenerate, and as such it defines a conformal symplectic structure on each contact plane. Imposing the aforementioned first--order condition is the same as restricting the prolonged exterior differential system $(M^{(1)},\mathcal{I}_{\CC}^{(1)})$ to a hypersurface $\E$ of the manifold $M^{(1)}$ of the $n$--dimensional integral elements of $(M,\mathcal{I}_{\CC})$.
The manifold $M^{(1)}$ has a natural smooth bundle structure
\begin{equation}\label{eq:m1}
\pi:M^{(1)}=\bigcup_{m\in M}\mathrm{LGr}(\mathcal{C}_{m})\rightarrow M\, ,
\end{equation}
such that
 the fibre of $\pi $ at $m\in M$ naturally identifies with the \emph{Grassmannian of
Lagrangian planes} $\mathrm{LGr}(\mathcal{C}_{m})$ of $\mathcal{C}_{m}$.\par
In contact (or Darboux) coordinates $(x^i,u,u_i)$, a generic Lagrangian submanifold of $M$ is the graph $\Gamma_f^{(1)}:=\{x^i,u=f(x^1,\ldots,x^n), u_i=f_{x^i}(x^1,\ldots,x^n)\}$ of the $1\St$ jet of a  function in the $n$ variables $x^i$. Hence, a first--order condition on Lagrangian submanifolds is a relation between the first derivatives of both $f$ and all the $f_{x^i}$'s, that is, a second--order PDE on $f$. Globally, this corresponds to a hypersurface $\E\subset M^{(1)}$.\par
One can locally extend the Darboux coordinates to $M^{(1)}$ as follows. The Lagrangian space $L=T_m\Gamma_f^{(1)}\in M^{(1)}$ has coordinates $(x^i,u,u_i,u_{ij}=u_{ji})$ where $(x^i,u,u_i)$ are the coordinates of $m$, and $T_m\Gamma_f^{(1)}=\Span{ \left.D_{x^i}\right|_m + \left.u_{ij}\partial_{u_j}\right|_m \mid i=1,2,\ldots,n}$,  with
\begin{equation}\label{eq.total}
D_{x^i}=\partial_{x^i}+u_i\partial_u\,.
\end{equation}
Observe that the $u_{ij}$--coordinates of $T_m\Gamma_f^{(1)}$ are precisely the second--order derivatives of $f$ at $(x^1,\ldots,x^n)$ and, as  such, they are symmetric in the indices $(i,j)$: this corresponds to the canonical identification
$
T_L\mathrm{LGr}(\mathcal{C}_{m})\simeq S^2L^*$, valid for all $L\in\LL(\CC_m)$.
Accordingly, in the above coordinates, a $2\Nd$ order PDE  reads as
$%\begin{equation*}%\label{eqProtoEquation}
\E=\{F(x^{i},u,u_{i},u_{ij})=0\}
$. \par%\end{equation*}
%The proposed  ``global'' perspective on $2\Nd$ order PDEs is necessary to deal with the main problem of this paper.
The intrinsic geometry of $\LG(\CC_m)$ allows us
to introduce a point-wise numerical invariant characterising
some of the hypersurfaces $\E$: the \emph{degree}.
Namely, we say that $\E = \{ F = 0 \}$ is of degree $d$
at $m = (x^i,u,u_i)$ if $F(x^i,u,u_i, u_{ij})$
is a polynomial of degree $d$ as a function of the minors
of the symmetric matrix $U=\big(u_{ij}\big)$. If the number $d$ is well-defined and does not depend on $m$, we say that the   PDE $\E$ has \emph{degree} $d$, though its \emph{order} is always 2 (this
is the case e.g. when a group of contactomorphisms acts transitively on $M$ preserving $\E$).
%The chief  example when the notion of degree is well--defined, is when there is a group of contactomorphisms acting transitively on   $\E$.

Recall that we have decided to work in the complex analytic setting. Furthermore, we shall
only be concerned with closed hypersurfaces $\E \subset M^{(1)}$. In that case the
Lagrangian Grassmannian $\LG(\CC_m)$ is a projective manifold
(in the Pl\"ucker embedding: the homogeneous coordinates on the ambient projective space
may be identified with the minors of $U$), and the fibre
$\E_m = \E \cap \LG(\CC_m)$ is a closed analytic subvariety. It then follows that
$\E_m$ is actually algebraic (Chow's theorem),
and the defining equation $F$ is a polynomial in
the minors. In particular, the degree of $\E$ at $m$ is a well-defined integer.

\subsection{A first  look at the results}\label{subFirstLook}
Our main results concern the case where $(M,\CC)$ is a homogeneous
contact manifold for a \emph{complex} simple Lie group $G$. In fact,
requiring $M$ to be compact, it is uniquely determined by $G$. We
consider  hypersurfaces $\mathcal{E}\subset M^{(1)}$ that are invariant under
the natural lift of the $G$-action. For some $G$ we give a complete and explicit
description of the set of all such hypersurfaces. In general, for each $G$
we characterise the minimal possible \emph{degree} of such a hypersurface.\par% (note
%the degree of a $G$-invariant hypersurface is the same at each point of $M^{(1)}$).
%
Since $G$ acts transitively on $M$ by contact transformations,
and $\mathcal{E} \subset M^{(1)}$ is supposed to be $G$-invariant,
the problem of describing such hypersurfaces is easily reduced to
the study of hypersurfaces in a single fibre $\LG(\CC_m)$
invariant under the subgroup of $G$ stabilising $m$. Thus our results
are in fact about hypersurfaces in a Lagrangian Grassmannian of
a (conformal) symplectic vector space, invariant under certain subgroups
of the (conformal) symplectic group.\par
A careful inspection shows that in all cases where we did find an explicit
construction of such an invariant hypersurface, it can be in fact defined
over the reals -- at least if we use the split real form of $G$. This is
manifest in Section \ref{sec:pdes}, where explicit PDEs are written down
in Darboux coordinates: even though formally we discuss the complex case,
the resulting formulas make sense in the real context as well. This leads
to new constructions of (real) PDEs with prescribed symmetry algebras, that
may serve as interesting test cases for the more traditional methods of symmetry analysis of PDEs.

\subsection{Plan of the paper}
Section \ref{SecDescMainRes} gives a more in-depth description of our main
results, introducing for the first time all the constructions necessary
for a precise statement of Theorem \ref{thm:invs}. Section \ref{sec:prerequisites} provides a careful technical
exposition of the underlying material, where
we reintroduce and prove many of the standard results mentioned
in Section \ref{SecDescMainRes}. This way we prepare the ground for an algebraic reformulation of
the main result, opening Section \ref{sec:proof}. We then outline the strategy of the proof,
based on the classification of complex simple Lie algebras. Parts of the
main result corresponding to the different Cartan types $\Asf,\dots,\Gsf$
occupy the subsequent subsections.
Section \ref{sec:further-discussion} contains a discussion of further consequences of our results
and their relation to existing research. Finally, in Section \ref{sec:pdes} we write down
some our invariant PDEs in explicit form.

\section{Description of main results and methods}\label{SecDescMainRes}

\subsection{Basic constructions and results}\label{secBasics}
We will now quickly introduce a number of notions necessary for a precise statement
of our results. Complete definitions and proofs will be given in Section \ref{sec:prerequisites}. We
state our main Theorem at the end of this subsection.

\begin{definition}\label{def:adjoint}
Let $\fg$ be a complex simple Lie algebra, and $G$ the identity component
of $\Aut\fg$. Then the unique closed $G$-orbit $X$ in $\PP\fg$ is called
the \emph{adjoint variety} of $\fg$.
\end{definition}

Let $\g$ be a complex simple Lie algebra other than $\sll(2,\C)$.
We let $G$ and $X \subset \PP\fg$ be as in Definition \ref{def:adjoint}.
We shall fix an origin $o\in X$ and let $P \subset G$ be its stabiliser.
The latter is a parabolic subgroup with unipotent radical $P_+ \subset P$ and
connected, reductive quotient $G_0 \simeq P/P_+$. We split the projection $P \to G_0$ by fixing a Levi
decomposition
\begin{equation}\label{eqLevi}
P = G_0 \ltimes P_+\, ,
\end{equation}
thus inducing the so--called \emph{contact grading} (see \cite{MR2532439}, sec. 3.2.4)
\begin{equation}\label{eqContGrad}
 \g = \g_{-2} \oplus \g_{-1} \oplus \g_0 \oplus \g_1 \oplus \g_2\,
\end{equation}
of the Lie algebra $\g$, where   $\g_0$ is the Lie algebra of $G_0$,
and $\mathfrak{p}_+ = \g_1 \oplus \g_2$ the Lie algebra of $P_+$. Considering
the graded nilpotent subalgebra $\g_- = \g_{-2} \oplus \g_{-1}\simeq T_oX$ one
finds that $\dim \g_{-2} = 1$ and
\begin{equation}\label{eqDefEnne}
\dim \g_{-1} = 2n\, ,\quad n \in \mathbb{N}\, ,
\end{equation}
with the Lie bracket inducing a non--degenerate twisted (i.e., $ \g_{-2}$--valued)  symplectic
form  on $\g_{-1}$.
\par
The hyperplane $\g_{-1}\subset T_oX$  equips  the homogeneous space $X \simeq G/P$
with an invariant contact structure.
We thus have a contact manifold
as in Subsection \ref{ss:from-pdes} and we can
form the bundle
\begin{equation}\label{eqSTARSTAR}
X^{(1)}\longrightarrow X
\end{equation}
of Lagrangian Grassmannians as in \eqref{eq:m1}.
We are interested in complex analytic
hypersurfaces in $X^{(1)}$,
locally interpreted as $2\Nd$ order scalar PDEs in $n$ independent
variables (see again Subsection \ref{ss:from-pdes}).
In particular, considering the natural lift of the $G$-action to $X^{(1)}$,
it is natural to ask about the existence and classification of \emph{$G$-invariant}
hypersurfaces, as these correspond to PDEs with a large symmetry group.

We will only work with \emph{closed} hypersurfaces. Since
$X^{(1)}$ is a complex \emph{projective} manifold, every
such closed analytic hypersurface is in fact algebraic
Thus, from now on, we will use the term `hypersurface'
to refer to a closed algebraic hypersurface.
We consider the family
\begin{equation*}
 \Inv(X,G)=\left\{\E\mid\E\textrm{ hypersurface in }X^{(1)}\textrm{ such that }G\cdot\E=\E\right\}\, .
 \end{equation*}
In fact,
we shall further abuse the terminology by allowing the components of our hypersurfaces
to have non-negative multiplicities, whence the proper term would be an \emph{effective divisor}
(this is indeed very natural if one thinks of a PDE as an \emph{equation} in an algebraic sense,
rather than a subset of a geometric space).
Doing so, we stick  to the usual language of geometric PDE theory without introducing  unnecessary restrictions on the algebraic side. This abuse does not affect our
results. \par
By virtue of transitivity, the elements of $\Inv(X,G)$
can be put in a natural  one--to--one correspondence with the $P$--invariant hypersurfaces in
the fibre $X^{(1)}_o$. Observe also that $X^{(1)}_o$ is   the \emph{Lagrangian Grassmannian} of  the (conformal) symplectic space $\g_{-1}$, naturally embedded in the \emph{projectivised Pl\"ucker space}
  $\PP \Lambda^n_0\g_{-1}$, where $\Lambda^n_0\g_{-1}$ is the kernel
of the map
\begin{equation}\label{eqINS}
\Lambda^n \g_{-1} \to \Lambda^{n-2}\g_{-1} \otimes \g_{-2}
\end{equation}
induced
by the (twisted) symplectic form.
This embedding corresponds to an ample line bundle $\OO(1)$ on $X^{(1)}_o$, generating
its Picard group. That is: (i) the space of global sections of $\OO(1)$
is identified with $\Lambda^n_0\g_{-1}^*$ and the Pl\"ucker embedding is
simply the evaluation map $X^{(1)}_o \to \PP \Gamma(X^{(1)}_o, \OO(1))^*$, and
(ii) every line bundle on $X^{(1)}_o$ is a tensor power of $\OO(1)$ or its inverse.
We then meet again the same  notion of   \emph{degree}, introduced earlier in Subsection \ref{ss:from-pdes}:    a hypersurface
in $X^{(1)}_o$  has degree $d$ if it is cut out by a global section of
$\OO(d)$ or, equivalently, is an intersection of Pl\"ucker-embedded $X^{(1)}_o$ with a degree $d$ hypersurface
in $\PP \Lambda^n_0\g_{-1}$ (see Definition \ref{defDegree} later on).
Given a hypersurface $\E \in \Inv(X,G)$,
we refer to the degree of $\E_o \subset X_o^{(1)}$ as the degree of $\E$ itself.
This is compatible with the definition given in Subsection \ref{ss:from-pdes}
in terms of minors of the Hessian.
\par
There is a natural geometric construction, appearing for the first time in \cite{2016arXiv160308251T}, producing an invariant hypersurface for every $\g$ of type not $\Csf$. We review it in Section \ref{sec:further-discussion}
 under the name of
the \emph{Lagrangian Chow transform} of the \emph{subadjoint variety}.
Knowing thus that $\Inv(X,G)$
is nonempty, and that its elements are grouped by degree, we consider the following natural questions:
first, \emph{find the minimum degree} of an element; then, whenever it is possible, \emph{establish whether or not there is a unique element of this degree}.
We may now state our main result.

%
% \begin{theorem}[Main]\label{thm:invs}
%The lowest--degree nontrivial components of $\Inv(X,G)$ are listed in the following table:
%\begin{center}\fbox{\begin{tabular}{r|cccccccc}
% {\normalfont{type}} & $\Asf_\bullet$ & $\Bsf_\bullet$ & $\Dsf_\bullet$ & $\Esf_6$ &$\Esf_7$ &$\Esf_8$ & $\Fsf_4$ & $\Gsf_2$ \\
%\hline
% {\normalfont{lowest--degree invariants}}  &  $\p R_{(1,0)}\cup \p R_{(0,1)}$ & $\p R_4$ & $\p R_2$ $(^*)$ & $\p R_2$ & $\p R_2$ & $\p R_2$ & $\p R_4$ & $\p R_3$ \\
%  \hline
% {\normalfont{degree of the sub--adjoint variety}}   & 2 & \multicolumn{2}{c}{$2(n-1)$}   &  42 & 286 & 13188 & 16 & 3
%\end{tabular}}\end{center}
%\end{theorem}
%
\begin{theorem}\label{thm:invs}
The minimal degree of an element of $\Inv(X,G)$ is given
 by the first row of the following table, while the second row gives the number of elements of that degree (entries marked with an asterisk are conjectural).
\begin{center}\fbox{\begin{tabular}{r|cccccccc}
\normalfont{type} & $\Asf_.$ & $\Bsf_.$ & $\Dsf_.$ & $\Esf_6$ &$\Esf_7$ &$\Esf_8$ & $\Fsf_4$ & $\Gsf_2$ \\
\hline
  & $1$ & $4$ & $2$ &$2$ &$2$ & $2$ & $4$ & $3$\\
  & $2$ & unknown    & $1^*$ &$1$ &$1$ & $1$ & $1$ & $1$\\
  \hline
  & $2$ & \multicolumn{2}{c}{$2(n-1)$}   &  $42$ & $286$ & $13188$ & $16$ & $3$
\end{tabular}}\end{center}
(For comparison, the last row gives the degree of the Lagrangian Chow transform of the subadjoint variety, where  $n$ is defined by  \eqref{eqDefEnne}, see Section  \ref{sec:further-discussion}).
\end{theorem}
\begin{remark}\label{remStupido}
Let us explain the absence of type $\Csf$ from the classification. Since in that case
$G = \Sp_{n+2}$
acts transitively on $X^{(1)}$, there are no $G$-invariant hypersurfaces
whatsoever, i.e., $\Inv(X,G)=\emptyset$.
\end{remark}
\subsection{Further constructions and additional results}
Our approach relies on translating the geometric problems associated with $\Inv(X,G)$
to algebraic problems of the theory of invariants of the semisimple part of $G_0$.
We will now give an informal description of this transition. The subtler parts of it,
namely equations \eqref{eqHomCoordRingX10} and \eqref{eqInvPRxi}, will be stated
and proved as standalone  results    in   Section \ref{sec:prerequisites} below (see Lemma \ref{Lem10} (3) and Proposition \ref{pro:invxg}). As they become rather technical,
here we mostly wish to motivate their use.

We have already reformulated a problem involving $G$--invariance in terms of  $P$--invariance. But since $P_+$ acts trivially
on $\g_{-1}$, we may  further replace the $P$--invariance condition with $G_0$--invariance or, more accurately,   $G_0$--\emph{equivariance}.
Note that the linear action of $G_0$ on $\Lambda^n_0\g_{-1}$ induces an action
on $\OO(1)$ over $X^{(1)}_o$.
Elements of $\Inv(X,G)$ of degree $d$ are thus in one--to--one correspondence with the nonzero $G_0$--equivariant global sections of the line
bundles $\OO(d)$, $d > 0$, modulo the action of the multiplicative group
$\C^\times$.
An element   $f \in \Gamma(X^{(1)}_o, \OO(d))$ is called $G_0$--\emph{equivariant} if
there exists a group homomorphism $\xi : G_0 \to \mathbb{C}^\times$ satisfying
\begin{equation}\label{eqEqEquivariance}
g^* f = \xi(g) f\, ,\quad \forall  g \in G_0\, .
\end{equation}
Thanks to this   re-interpretation, our problem can be formulated in purely   algebraic terms. First, we identify the spaces of global sections
of $\OO(d)$ with the homogeneous summands of
the homogeneous coordinate ring of $X^{(1)}_o$, viz.
\begin{equation}\label{eqHomCoordRingX10}
\bigoplus_{d\ge 0} \Gamma(X^{(1)}_o, \OO(d)) \simeq (S^\bullet \Lambda^n_0 \g_{-1}^*) / I\, ,
\end{equation}
where $S^\bullet$ denotes the symmetric algebra,
and $I$ is the homogeneous ideal of $X^{(1)}_o $ in   its Pl\"ucker embedding (this holds by virtue of projective normality, see Sec. \ref{sec:prerequisites}).  In principle,  it remains  to take the submodule of $G_0$--equivariant elements in the  right--hand side of \eqref{eqHomCoordRingX10}, decompose it into irreducible submodules, and then  detect the one--dimensional ones.
It is convenient to consider an almost direct product decomposition
\begin{equation}\label{eqDecGOGOss}
 G_0 = G_0^\ss \cdot \tilde T,\quad G_0^\ss = [G_0,G_0]
\end{equation}
where $G_0^\ss$ is semi-simple, $\tilde T$ is a central torus, and
the two intersect in a finite subgroup.
We also have the quotient torus
$$ T = G_0 / G_0^\ss \simeq \tilde T / (G_0^\ss \cap \tilde T). $$
The  homomorphism
$\xi$ involved in the definition  \eqref{eqEqEquivariance}  of $G_0$-equivariance acts trivially on
$G_0^\ss$ and can thus be regarded as  an element of the
lattice $$\widehat T:=\Hom(T,\C^\times)\simeq\Z^r,\quad r=\rk T$$
of \emph{characters} of $T$ (it is a sublattice in the lattice of characters of $\tilde T$).
%, cf. \eqref{eqDecGOGOss}.
Since   the subring
\begin{equation}\label{eqDefR}
R := (S^\bullet \Lambda_0^n \g_{-1}^* / I)^{G_0^\ss}
\end{equation}
of $G_0^\ss$--invariants (i.e., $\g_0^\ss$-invariants, since $G_0^\ss$ is connected) is also $\tilde T$--invariant in $S^\bullet \Lambda_0^n \g_{-1}^* / I$,
we can decompose it into irreducible $T$--submodules:
\begin{equation}\label{eqGradR}
R = \bigoplus_{\xi \in \widehat T} R^\xi \, .
\end{equation}
\begin{remark}\label{remTorus}
 If $r=1$ (resp., $r=2$), then we can identify $\widehat T$ with $\Z$ (resp., with $\Z\times\Z$), in such a way that
$R^\xi\subset\Gamma(X_o^{(1)},\OO(\frac{1}{m} \xi))$
(resp.,   $R^{(\xi_1,\xi_2)}\subset\Gamma(X_o^{(1)},\OO( \frac{1}{m}\xi_1))$)
for a suitable positive integer $m$.
As we shall see below, the case  $r=2$ occurs only in type $\Asf$, otherwise $r=1$.
That is, we can regard the multi--grading \eqref{eqGradR}  as the usual  grading,  except    in type $\Asf$, when it is  a bi--grading refining the usual grading.
\end{remark}
 We are finally in position to recast the geometric problem of studying $G$--invariant hypersurfaces of $X_o^{(1)}$ in the algebraic terms of
the (bi)graded ring $R$ of $\g_0^\ss$-invariants:
\begin{equation}\label{eqInvPRxi}
\Inv(X,G) \simeq \coprod_{ \xi \in \widehat T\smallsetminus\{0\}} \PP R^\xi\, .
\end{equation}
A careful exposition and proof of the   identification \eqref{eqInvPRxi} is the
purpose of the following Section \ref{sec:prerequisites} (see Proposition \ref{pro:invxg}).
It is worth recalling that the elements of $\Inv(X,G)$ are really effective divisors;
to keep them reduced, we would need to remove classes of non-reduced elements of $R$ from
the right hand side of \eqref{eqInvPRxi}. Since we are interested in lowest degree elements,
thus automatically reduced, the distinction is irrelevant for our purposes.\par

Our main result, Theorem \ref{thm:invs}, is recast in this language as Theorem \ref{thm:invs2}
in Section \ref{sec:proof}.
Somewhat more can be said in special cases (see Proposition \ref{pro:AG}).
For $\fg$ of type $\Asf$, the ring $R$ is freely
generated by the pair of elements of degree $1$.
These are interchanged by the outer automorphism of $G$
corresponding to the reflection symmetry of the Dynkin diagram (i.e. transposition
of matrices); their product (corresponding to a union of
hypersurfaces) is the only reduced element invariant under the full
automorphism group of $\fg$.
For $\fg$ of type $\Gsf_2$, the ring $R$ is generated by a single element of degree $3$.
Finally, let us remark that
for $\fg$ of type $\Bsf$ and $\Dsf$, as well as $\Asf$ and $\Gsf_2$, we construct explicit
forms of the lowest degree non-constant elements of $R$. This allows us to
write down the corresponding PDEs in coordinates (see Section \ref{sec:pdes}).

\section{Prerequisites}\label{sec:prerequisites}

\subsection{The adjoint variety as a contact manifold}\label{subAdjContMan}
We will now return to the notions and constructions
introduced informally in   Section \ref{SecDescMainRes} above, providing
necessary definitions and proving some of their properties.
The results are standard, but we give detailed proofs
wherever they contain some valuable insights for the non-expert reader.
Some of the material of the previous sections is repeated in order to
make the present one self-contained.
Our ultimate goal is Proposition \ref{pro:invxg}. The reader is invited
to skip directly to the latter should the exposition become too pedagogical.

We fix $\fg$, $G \subset \Aut\fg$ and $X \subset \PP\fg$ as in
Definition \ref{def:adjoint}. Note that $G$ is a connected simple linear algebraic group
and $X$ a projective homogeneous variety for $G$, whence it follows that
the stabiliser in $G$ of any point of $X$ is a \emph{parabolic} subgroup.
Recall that the structure of such subgroups is easily described in terms
of root system data; we will return to this point later (see the proof of
Lemma \ref{lem:properties-grading}).\par
Let us refresh the notations  introduced earlier in Subsection \ref{secBasics}. The point $o \in X$ is the \emph{origin},
and  $P \subset G$ is its stabiliser. The normal subgroup $P_+ \subset P$, consisting
of the unipotent elements in the radical, is
the \emph{unipotent radical} of $P$. The quotient group $G_0 = P/P_+$
is reductive, namely a product of a connected semisimple group $G_0^\ss$
and a central torus $T$ (note that in our setting a torus means a
direct product of several copies of the multiplicative group $\C^\times$).
Later we will determine the type of $\fg_0^\ss$ and the rank of $T$
in terms of root system data (see the table in Lemma \ref{lem:properties-grading}). Symbols   $\fp_+ \subset \fp \subset \fg$ denotes the Lie algebras of $P_+ \subset P \subset G$.\par
Now we define a natural $P$-invariant \emph{decreasing filtration} on $\fg$, by setting
$\fg^{i+1} = [\fp_+,\fg^i]$ unless $\fg^{i+1}=\fg$ and arranging the indices so that $\fg^0 = \fp$.
Lemma \ref{lem:properties-filtration} below captures the key properties of this filtration. To this end, denote by   $\gr_\bullet(\fg/\fp)$   the associated graded space
of the induced filtration on $\fg/\fp$, with the natural action of $P$.
\begin{lemma}\label{lem:properties-filtration}
Let $\fg^\bullet$ be the filtration on $\fg$ satisfying $\fg^0 = \fp$ and $\fg^{i+1} = [\fp_+,\fg^i]$ unless $\fg^{i+1}=\fg$.
Then:
\begin{enumerate}
\item $[\fg^i,\fg^j]\subset\fg^{i+j}$,
\item $\fg^{-i} = \fg$ and $\fg^{i+1}=0$ for all $i \ge 2$,
\item there is an induced $P$-invariant filtration on $\fg/\fp$.
\item $P_+$ acts trivially on $\gr(\fg/\fp)$,
\item $\gr(\fg/\fp) = \gr_{-2}(\fg/\fp) \oplus \gr_{-1}(\fg/\fp)$,
\item $\dim\gr_{-2}(\fg/\fp)=1$, $\dim\gr_{-1}(\fg/\fp)$ is even, and the map
$$ \omega : \Lambda^2 \gr_{-1}(\fg/\fp) \to \gr_{-2}(\fg/\fp) $$
induced by the Lie bracket is a $P$-equivariant \emph{twisted symplectic form}.\footnote{The condition of $\omega$ being a twisted symplectic form means that
the induced map $\gr_{-1}(\fg/\fp) \to \gr_{-1}(\fg/\fp)^* \otimes \gr_{-2}(\fg/\fp)$
is an isomorphism.}
\end{enumerate}
\end{lemma}

We will prove Lemma \ref{lem:properties-filtration}  once we introduce a \emph{grading} on $\fg$
splitting the filtration, and describe it in terms of the root system data,
i.e., after the proof of Lemma \ref{lem:properties-grading}.
The reason for delaying this step is that the filtration is natural, i.e.,
it is well-defined as soon as we have chosen the origin $o \in X$ (equivalently
the parabolic $P \subset G$), while the grading will require us to
make an additional choice (corresponding to fixing a homomorphism $G_0 \hookrightarrow P$
splitting the natural projection). Of course, passing to the root system description requires
even further choices.

Let us now identify $X$ with $G/P$, so that the coset $gP$ corresponds to $g\cdot o$. Viewing
$G \to X$ as a $P$-principal bundle, we may then identify the tangent bundle of $X$ with
an associated bundle:
$$ TX \simeq G\times^P (\fg/\fp). $$
The $P$-equivariant filtration on $\fg/\fp$ induces then a $G$-invariant filtration
on $TX$. The only non-trivial sub-bundle we obtain this way is
$$
 TX \supset \CC \simeq G\times^P (\fg^{-1}/\fp) \subset G\times^P (\fg/\fp).
$$
\begin{lemma}
The sub-bundle $\CC \subset TX$ is a $G$-invariant contact distribution. Furthermore,
the Levi bracket $\Lambda^2\CC \to TX/\CC$ evaluated at $o \in X$ coincides
with the map $\omega$ under the identifications $\CC_o \simeq \gr_{-1}(\fg/\fp)$
and $T_oX/\CC_o \simeq \gr_{-2}(\fg/\fp)$.
\end{lemma}
\begin{proof}
We use Lemma \ref{lem:properties-filtration}, in particular
(5) and (6).
$G$-invariance of $\CC$ follows from
$P$-invariance of $\fg^{-1}$. Furthermore, we have a natural identification
$$ TX/\CC \simeq G\times^P \gr_{-2}(\fg/\fp) $$
so that in particular $\CC$ has corank $1$ in $TX$.
To check that $\CC$ is a contact distribution it will be enough to show
commutativity of the following diagram of
vector bundle homomorphisms:
$$\begin{CD}
\Lambda^2 \CC @>>> TX/\CC \\
@| @| \\
G\times^P \Lambda^2 \gr_{-1}(\fg/\fp) @>{G\times^P\omega}>> G\times^P \gr_{-2}(\fg/\fp)
\end{CD}$$
where the top horizontal arrow is the Levi bracket,
and the vertical arrows are the natural identifications
with associated bundles. By $G$-equivariance, it is enough to restrict
to fibres over $o \in X$, where we want to show that
the Levi bracket $\Lambda^2\CC_o \to T_oX/\CC_o$
corresponds to $\omega$ under the identification
$T_oX \simeq \fg/\fp$ and $\CC_o \simeq \fg^{-1}/\fp$.
This is clear once one considers the commutative diagram
$$
\begin{CD}
\fg @>>> \Gamma(X,TX) \\
@VVV @VVV \\
\fg/\fp @= T_oX
\end{CD}
\quad\textrm{restricting to}\quad
\begin{CD}
\fg^{-1} @>>> \Gamma(X,\CC) \\
@VVV @VVV \\
\fg^{-1}/\fp @= \CC_o
\end{CD}
$$
where the top horizontal arrow is the infinitesimal action
of $\fg$ on $X$, the left vertical arrow is the natural projection,
and the right vertical one is the evaluation at the origin.
\end{proof}

The construction of the contact structure we have given here
emphasises the homogeneous space aspect of $X$. One may approach
it from a different angle as well: as the contact projectivisation of
a symplectic orbit of $G$ in $\fg$. We will not pursue this interpretation.

\subsection{The contact grading and the  root system data}\label{subContGrad}
It is standard---and notationally convenient---to work with a grading on $\fg$ rather than a filtration. We will follow this
custom in the remainder of this article; the previous subsection served as the
last reminder that the filtration is geometrically more fundamental.
Recall that the parabolic $P \subset G$ may be identified (non-canonically)
with the semi-direct product \eqref{eqLevi}, where $G_0 = P/P_+$ is the Levi factor.
Let us now \emph{fix} one such identification, amounting to choosing a homomorphism
$G_0 \to P$ splitting the natural projection, and from now on view $G_0$ as a subgroup
of $P$ (and thus also of $G$). Now, since $G_0$ is \emph{reductive}, it follows
that each $\fg^{i+1} \subset \fg^i$ has a $G_0$-invariant complement $\fg_i$ in $\fg^i$,
and we obtain a $G_0$-equivariant vector space decomposition $\fg = \bigoplus_i \fg_i$.
\begin{lemma}\label{lem:properties-grading}
The $G_0$-equivariant decomposition $\fg = \bigoplus_i \fg_i$ splitting the filtration $\fg^\bullet$
is unique, and satisfies the following properties:
\begin{enumerate}
\item $[\fg_i,\fg_j] \subset \fg_{i+j}$,
\item $\fg_{-i} = \fg_i = 0$ for $i>2$,
\item $\fg_{-i}^* \simeq \fg_i$ as representations of $G_0$,
\item $\dim\fg_{-2} = 1$, $\dim\fg_{-1}$ is even, and the map
\begin{equation}\label{eqBOX}
\omega:\Lambda^2 \fg_{-1} \to \fg_{-2}
\end{equation}
induced by the Lie bracket is a $G_0$-equivariant twisted symplectic form.
\end{enumerate}
Furthermore, recalling the decomposition $G_0 = G_0^\ss \times T$,
the following table lists, for
the different Cartan types of $\fg$, the isomorphism type
of $\g_0^\ss$, the rank of $T$, and
$\fg_{-1}$ as a representation of $\g_0^\ss$.
\begin{center}
\fbox{\begin{tabular}{l|l|l|c|l}
restriction & type of $\fg$   &  $\g_0^\ss$ & $\rk T$ & $\fg_{-1}$ \\ \hline
$n\ge 1$  & $\Asf_{n+1}$& $\sll_n $ & 2 & $\C^n \oplus \C^{n*}$  \\
$n\ge 3$ & $\Bsf_{(n+3)/2}$
or $\Dsf_{(n+4)/2}$
&  $\sll_2 \oplus \so_n$ & 1 & $\C^2\otimes\C^n$   \\
$n \ge 2$ & $\Csf_{n+1}$ & $\sp_n$ &1 & $\C^{2n}$   \\
&$\Esf_6$          & $\sll_6$ &1 & $\Lambda^3\C^6$   \\
&$\Esf_7$          & $\mathfrak{spin}_{12}$ &1 & spinor  \\
&$\Esf_8$          & $\Esf_7$ &1 & fundamental 56-dim \\
&$\Fsf_4$          & $\sp_3$ &1 & $\Lambda^3_0\C^6$   \\
&$\Gsf_2$          & $\sll_2$ &1 & $S^3\C^2$
\end{tabular}}
\end{center}
\end{lemma}
We refer to the above grading as \emph{the contact grading} of $\fg$ (it is
precisely the grading \eqref{eqContGrad}, see   Subsection \ref{secBasics}).
The proof uses the structure theory of $\fg$. We will only introduce it locally,
as it will not be needed throughout most of the paper---until it reappears in Subsection \ref{SubRepTheoSetUp},
where we supplement it with further representation-theoretic entities.
\begin{proof}
We fix a Cartan subalgebra $\fh \subset \fg$, and let $\Phi \subset \fh^*$ be the
root system of $\fg$. It is sometimes useful to consider also the maximal torus
$H \subset G$ corresponding to $\fh$, and its lattice of characters; since $G$ is
of adjoint type, that   lattice may be identified with the root lattice
generated in $\fh^*$
by $\Phi$.
Write
$$
\g = \fh \oplus \bigoplus_{\alpha\in\Phi}\fg_\alpha
$$
for the root space decomposition.
Let us further choose a subset $\Phi^+\subset\Phi$
of positive roots, and a system $\Delta\subset\Phi^+$ of simple roots. Given an
element $\alpha\in\mathbb{Z}\Phi$ of the root lattice, write $\alpha > 0$ if $\alpha
\in\mathbb{Z}_+\Delta$.  The direct sum of $\fh$ and root subspaces $\fg_\alpha$, $\alpha>0$
is a Borel subalgebra $\mathfrak{b}\subset\fp$.
Let $\gamma \in \Phi^+$
be the \emph{highest root}, i.e. the highest weight of the adjoint representation $\fg$. The
one-dimensional subspace $\fg_\gamma \subset \fg$ defines a point in $\PP\fg$; furthermore,
since the $G$-orbit of this point is closed, we have that in fact $\fg_\gamma$ belongs to
the adjoint variety $X$.
Possibly conjugating the choices we've made thus far, we shall assume that $\fg_\gamma$
is the origin $o \in X$. We may now compute $\fp \subset \fg$ as the subalgebra
stabilising $\fg_\gamma \subset \fg$: being $H$-invariant,
it is a direct sum of $\fh$ and some root subspaces. Clearly,
$\fp$ contains $\mathfrak{b}$; on the other hand, $\g_{-\alpha}$, $\alpha>0$ stabilises $\fg_\gamma$
if and only if $\alpha$ is orthogonal to $\gamma$. We thus find $\fp$ and its nilpotent radical
$\fp_+$ to be:
$$ \fp = \fh \oplus \bigoplus_{\alpha>0} \fg_\alpha \oplus \bigoplus_{\substack{\alpha>0 \\
(\alpha,\gamma)=0}}\fg_{-\alpha},\quad
\fp_+ = \bigoplus_{\substack{\alpha>0 \\ (\alpha,\gamma)\neq 0}} \fg_\alpha $$
where $(\cdot,\cdot)$ denotes the inner product on $\fh^*$ induced by the Killing form.
Again by the freedom to conjugate our choices by an element of $P$,
we may assume $H$ is contained in $G_0 \subset P$; then, being $H$-invariant, $\g_0$
is necessarily a direct sum of root subspaces and we find:
$$
\fg_0 = \fh \oplus \bigoplus_{(\alpha,\gamma)=0} \fg_\alpha.
$$
In particular, $\fh \subset \fg_0$ is a Cartan subalgebra, and $\fg_0$
is a reductive Lie algebra with root system $\Phi_0 \subset \Phi$ consisting
of roots orthogonal to $\gamma$. It is at this point clear that the $G_0$-invariant
splitting of the filtration $\g^\bullet$ is unique, since:
i) the resulting $\g_i$ are necessarily direct sums of root subspaces and subspaces of $\fh$,
ii) $\fh$
must be entirely contained in $\g_0$, for it is contained in $\fp$, and has
trivial intersection with $\fp_+ = [\fp,\fp_+]$.
 It is enough to exhibit such
a splitting. Let $\gamma^\vee \in \fh$ denote the coroot corresponding to $\gamma$,
and let $\langle \alpha,\gamma^\vee\rangle$ denote the natural pairing for a root $\alpha \in\fh^*$.
We then set\footnote{In other words, $\gamma^\vee \in \fh$ serves as
 the so-called \emph{grading element}.}
$$ \fg_i = \bigoplus_{\langle \alpha,\gamma^\vee\rangle = i} \fg_\alpha \oplus
\begin{cases}
\fh & i=0 \\
0 & i\neq0
\end{cases}$$
and observe that
$$
\fg_{-2} = \fg_{-\gamma},\quad
\fg_{-1} = \bigoplus_{\substack{\alpha>0,\ \alpha\neq\gamma \\ (\alpha,\gamma)\neq0}} \fg_{-\alpha},\quad
\fg_0 = \bigoplus_{(\alpha,\gamma)=0}\fg_\alpha\oplus\fh,\quad
\fg_{1} = \bigoplus_{\substack{\alpha>0,\ \alpha\neq\gamma \\ (\alpha,\gamma)\neq0}} \fg_{\alpha},\quad
\fg_{2} = \fg_{\gamma}.
$$
In particular, $\fg_0$ above \emph{is} the Lie algebra of $G_0 \subset P$, justifying the notation.
By construction, $\fg = \bigoplus_i \fg_i$ is a grading compatible with the Lie bracket,
and setting $\fg^i = \bigoplus_{j\ge i}\fg_j$ we have $\fg^0 = \fp$, $\fg^1 = \fp_+ = [\fp_+,\fp]$, $\fg^2 = [\fp_+,\fp_+]$.
Using the fact that the roots of $\fg_1$ are non-orthogonal to $\gamma$, it is also not difficult to check that
the bracket map $\fg_1 \otimes \fg_i \to \fg_{i+1}$ is surjective
for all $2 \le i \le 2$, thus proving that $[\fg^i,\fp_+] = \fg^{i+1}$. Thus the grading defined using $\gamma^\vee$
does indeed split the filtration induced by $\fp_+$.

We have thus established uniqueness, and found an explicit description of the grading.
Claims (1) and (2) are immediate. Claim (3) is straightforward, since $\g_0$ is clearly self-dual
via the Killing form, while the root subspaces in $\fg_i$ and $\fg_{-i}$ for $i\neq0$ correspond
to opposite subsets of $\Phi$. To prove (4), observe that
for every positive root $\alpha$ with $\langle \alpha,\gamma^\vee\rangle=1$, the vector
$\gamma - \alpha$ is also a positive root, and $\langle \gamma-\alpha,\gamma^\vee\rangle=2-1=1$. Hence
the map $\omega : \Lambda^2 \fg_{-1} \to \fg_{-2}$ decomposes into a sum of
isomorphisms $\fg_{-\alpha} \otimes \fg_{-\gamma+\alpha} \to \fg_{-\gamma}$.

It remains to verify the entries in the table. This is done on a case-by-case basis, and we shall only
sketch the argument. Since $\fg_0^\ss$ is a semi-simple Lie algebra with root system
$\Phi_0$, its Cartan type may be encoded by its Dynkin diagram. Then, since $\Delta_0 = \Delta\cap\Phi_0$
provides a system of simple roots for $\Phi_0$, we find that the Dynkin diagram of $\fg_0^\ss$
is a sub-diagram of the Dynkin diagram of $\fg$, obtained by removing the nodes corresponding to fundamental
weights entering with non-zero coefficients into the highest weight of the adjoint representation. At the same time,
it follows that the Cartan subalgebra $\fh_0$ of $\fg_0^\ss$ has rank equal to the cardinality of $\Delta_0$; since
the Lie algebra of the torus $T$ provides a complement to $\fh_0$ in $\fh$, it follows that $\rk T$ is equal to the
number of removed nodes. Finally, in order to find the decomposition of $\g_{-1}$ into
irreducible $U(\g_0^\ss)$-modules,\footnote{$U$ denotes the universal enveloping algebra.
We'll sometimes use the language of $U(\g_0^\ss)$-modules instead of representations
of $\g_0^\ss$.} it is enough to find roots $\alpha > 0$
with $\langle \alpha,\gamma^\vee\rangle = 1$ and such that $\alpha - \beta \notin \Phi$
for all $\beta\in\Phi_0 \cap \Phi^+$. Then, to each such root $\alpha$ corresponds an irreducible summand of $\g_{-1}$
whose highest weight with respect to $\fh_0$ is the image of $-\alpha$ under the natural projection $\fh^* \to \fh_0^*$.
Now, such roots $\alpha$ are precisely the simple roots in $\Delta \setminus \Delta_0$, and to each $\alpha \in \Delta\setminus\Delta_0$
we have the highest weight $U(\fg_0^\ss)$-module with highest weight $\lambda(\alpha)=-\alpha|_{\fh_0}$. Evaluating the latter on a coroot
associated with a simple root $\beta\in\Delta_0$ of $\g_0^\ss$, we have
$$ \langle \lambda(\alpha), \beta^\vee\rangle  = -\langle \alpha,\beta^\vee\rangle \ge 0 $$
so that the coefficients may be read off the Cartan matrix of $\fg$. Applying this recipe for each type one finds that
\begin{itemize}
\item in type $\Asf_{n+1}$ the set $\Delta\setminus\Delta_0$ consists of the two extreme nodes of the Dynkin diagram $\Delta$,
and the highest weights of the two summands of $\g_{-1}$ are fundamental, corresponding to the two extreme nodes of the Dynkin subdiagram $\Delta_0$;
\item in all remaining types the set $\Delta\setminus\Delta_0$ consists of a single node $\alpha$, and the highest weight of $\g_{-1}$ is:
the fundamental weight corresponding to the node of the Dynkin subdiagram $\Delta_0$ adjacent to $\alpha$, \emph{times} the
number of edges connecting the two if the arrow points away from $\alpha$.
\end{itemize}
These translate into the data we have included in the table.
\end{proof}

It is now easy to prove the properties of the filtration $\fg^\bullet$
used in the previous subsection.
\begin{proof}[Proof of Lemma \ref{lem:properties-filtration}]
Recall that in the proof of Lemma \ref{lem:properties-grading} we have
identified
$$ \g^{i} = \bigoplus_{j\ge i}\g_j. $$
Then claim (1) of Lemma \ref{lem:properties-filtration}
follows from claim (1) of Lemma \ref{lem:properties-grading},
as well as claim (2) of the former from claim (2) of the latter.
Claim (3) is then obvious, since $\fp = \fg^0$, and so is
claim (4), since $P_+ = \exp\fp_+$.
Finally, claims (5) and (6) follow immediately from Lemma \ref{lem:properties-grading}
once one identifies $\gr (\fg/\fp)$ with $\g_{-2} \oplus \g_{-1}$.
\end{proof}
{\bigskip\noindent\bf Notation.} From now on we shall use the graded subspaces $\g_i$, in particular
identifying $\g_i$ with $\gr_i(\fg/\fp)$ for $i<0$. We thus view $\g_{-1}$
as a representation of $P$, with the trivial action of $P_+$.
Furthermore, we identify the fibre $X^{(1)}_o$ of the Lagrangian
Grassmannian bundle at the origin with the Lagrangian Grassmannian $\LG(\g_{-1})$
of the conformal symplectic space $\g_{-1}$. We will alternate between
the two notations depending on context.

\subsection{The Lagrangian Grassmannian as a homogeneous space}
\label{ss:x1-ass-bun}
The contact space $\CC_o \simeq \g_{-1}$ at the origin
plays the role of a model for the intrinsic geometry
of the contact distribution. In particular (see Section \ref{secBasics}), the bundle
of Lagrangian Grassmanians $X^{(1)}$ introduced in \eqref{eqSTARSTAR}  is an associated bundle
with fibre modeled on the Lagrangian Grassmannian of the twisted
symplectic space $\g_{-1}$:
\begin{equation}\label{eqSTAR}
X^{(1)} \simeq G\times^P \LG(\g_{-1})\, .
\end{equation}
In fact, we shall typically use the notation $X^{(1)}_o$
instead of $\LG (\g_{-1})$. As we have already observed, the action of $P$ on $X^{(1)}_o$
factors through $G_0 \simeq P/P_+$, whence all problems we consider
in this paper may be reduced to the study of the $G_0$-action on $X^{(1)}_o$.
Before we reach this point, we need to understand $X^{(1)}_o$ as a homogeneous
space for the full symplectic group $\Sp (\g_{-1})$. The latter denotes the stabiliser in
$\GL (\g_{-1})$ of \emph{any} symplectic form obtained from $\omega$
by trivialising $\g_{-2}\simeq\C$. Let us first explain the place of $G_0$ in this picture.
\begin{lemma}
Every trivialisation of $\g_{-2} \simeq \C$ gives rise to a symplectic
form on $\g_{-1}$, and all such forms differ by scaling. The stabiliser
of any such form is the same subgroup of $\GL( \g_{-1})$,
denoted $\Sp(\g_{-1})$, while the stabiliser of
the one-dimensional
space of all such forms is denoted $\CSp (\g_{-1}) \simeq \Sp (\g_{-1}) \times \C^\times$.
The natural
action of $G_0$ on $\g_{-1}$ embeds $G_0$ as a subgroup of $\CSp(\g_{-1})$
and $G_0^\ss$ as a subgroup of $\Sp(\g_{-1})$.
\end{lemma}
\begin{proof}
The only non-trivial statement is that $G_0^\ss$ is contained in $\Sp(\g_{-1})$.
But since the action of $G_0^\ss$ clearly preserves the
one-dimensional space of symplectic forms on $\g_{-1}$ identified with $\g_{-2}^*$,
it follows that $G_0^\ss$ acts on this space by a character $\chi:G_0^\ss \to \C^\times$.
Of course, by semi-simplicity, $\chi$ is trivial, so that in fact $G_0^\ss$ preserves
every non-zero symplectic form parameterised by $\g_{-2}^*$.
\end{proof}
%
%Let us denote the dimension of $\g_{-1}$ by $2n$.
We shall reserve the
symbol $n$, defined earlier in \eqref{eqDefEnne}, for the half-dimension of the symplectic space, or contact distribution,
throughout the remainder of this paper. We use $\sp(\g_{-1})$ to denote
the Lie algebra of $\Sp(\g_{-1})$, a subalgebra of $\End(\g_{-1})$.
\begin{definition}
The Lagrangian Grassmannian $\LG(\g_{-1})$ is
the submanifold of $\Gr(n, \g_{-1})$ parameterising
maximal isotropic subspaces, i.e., $n$-dimensional linear subspaces
$L \subset \g_{-1}$ such that $\omega|_L=0$.
\end{definition}

It is well-known that the group $\Sp(\g_{-1})$ acts transitively on $\LG(\g_{-1})$.
Given a point in $\LG(\g_{-1})$ corresponding to $L \subset \g_{-1}$,
its stabiliser in $\Sp (\g_{-1})$ is a maximal parabolic subgroup, as we will
soon see in the context of a minimal projective embedding.
A more direct description of the stabiliser
may be obtained by fixing a Lagrangian complement to $L$. Such a complement
is identified with $L^* \otimes \g_{-2}$, and the choice
of a nonzero element in $\g_{-2}$ further identifies
it with $L^*$. As this involves arbitrary choices,
we will avoid its use in what follows; it is however
convenient for concrete computations. The proof
is completely standard, and thus omitted.
\begin{lemma}\label{lem:bilag-stuff}
The choice of a bi-Lagrangian
decomposition $$\g_{-1} = L \oplus L^*$$
induces:
\begin{enumerate}
\item a graded decomposition
$$ \sp (\g_{-1}) = S^2 {L^*} \oplus \End L \oplus S^2 L $$
with $\End L$ in degree $0$,
acting naturally on the remaining two summands in degrees $\pm 1$,
and with $[\varphi,\psi] = \psi\circ\varphi$ for
$\varphi \in S^2L^*$, $\psi \in S^2L$ viewed as maps
$\varphi:L\to L^*$, $\psi:L^* \to L$;
\item identifications
$$ \Stab L = \GL(L) \ltimes S^2 L,\quad \Stab L^* = S^2 L^* \rtimes \GL(L),\quad
\Stab L \cap \Stab L^* = \GL(L) $$
where $\Stab L$, $\Stab L^*$ are the stabilisers of $L$, $L^*$ in $\Sp (\g_{-1})$,
and we view $S^2 L$, $S^2 L^*$ as vector groups;
\item a $\GL(L)$-equivariant identification
 $$ S^2 L^* \simeq \{ \Lambda \in \LG( \g_{-1}) \ |\ \Lambda \cap L^* = 0 \} $$
sending $\varphi \in S^2 L^*$ to the graph of $\varphi : L \to L^*$.
\end{enumerate}
\end{lemma}

We remark that, from the point of view of a $G_0$-action,
a natural bi--Lagrangian decomposition of the symplectic vector space
$\g_{-1}$ exists only in type $\Asf$, where $G_0^\ss$ is precisely
the semisimple part of the Levi factor $\GL_n$
of the parabolic arising as a stabiliser
of a point of $\LG (\g_{-1})$ in $\Sp (\g_{-1})$.

\subsection{The Pl\"ucker embedding}\label{subPluck}
The Lagrangian Grassmannian comes equipped with a distinguished
$\Sp (\g_{-1})$-equivariant embedding into the projectivisation
of an irreducible representation (more precisely, the kernel of the map \eqref{eqINS}, see Subsection \ref{secBasics}). We will describe its properties
in this subsection, along with some further data on the representation
theory of the symplectic group. As before, we confine our use
of structure theory to the proofs.
\begin{definition}
Let $\dim \g_{-1} = 2n$, as in \eqref{eqDefEnne}.
\begin{enumerate}
\item $\Lambda^i_0 \g_{-1}$ denotes the kernel of the map
$\Lambda^i \g_{-1} \to \Lambda^{i-2}\otimes\g_{-2}$
given by contraction with $\omega \in \Lambda^2 \g_{-1}^* \otimes \g_{-2}$,
\item $\Lambda^n_0 \g_{-1}$ is called the (Lagrangian) Pl\"ucker space (cf. \eqref{eqINS}).
\end{enumerate}
\end{definition}

For $i=1$ we understand the above map to be zero, so that
$\Lambda^i_0\g_{-1} = \Lambda^i\g_{-1}$ by definition.

\begin{lemma}\label{lem:sp-fundamental}\
\begin{enumerate}
\item The spaces $\Lambda^i_0 \g_{-1}$, $1 \le i \le n$ are
precisely the fundamental irreducible representations of $\Sp (\g_{-1})$.
\item
For each $d\ge0$ the symmetric power $S^d \Lambda^n_0 \g_{-1}$ contains an
irreducible subrepresentation $S^d_0 \Lambda^n_0 \g_{-1}$ spanned by
rank one subspaces of the form $(\det L)^d \subset S^d \Lambda^n_0\g_{-1}$
for $L \subset \g_{-1}$ a Lagrangian subspace.
\end{enumerate}
\end{lemma}
\begin{proof}
A proof of part (1) may be found in~\cite[Ch. VIII, 3]{MR2109105};
we quote some of its points.
If we fix a bi-Lagrangian decomposition $\fg_{-1} = L \oplus L^*$
inducing an embedding $\End L \subset \sp( \g_{-1})$, every Cartan subalgebra
of $\End L$ is also a Cartan subalgebra of $\sp (\g_{-1})$. Fixing a basis $e_1,\dots, e_n \in L$,
let $\fh \subset \End L$ be the corresponding diagonal subalgebra, with a basis $H_1,\dots, H_n$
such that $H_i(e_j)=\delta_{ij} e_j$. Letting $\eta_1,\dots,\eta_n \in \fh^*$ be the dual
basis, it turns out that we can write a system of simple roots as
$$ \alpha_i = \eta_i - \eta_{i+1},\quad 1 \le i \le n-1,\quad \alpha_n = 2\eta_n $$
with corresponding fundamental weights
$$ \lambda_i = \eta_1 + \dots + \eta_i,\quad 1 \le i \le n. $$
Then $e_1 \wedge \dots \wedge e_i$ is the highest weight vector
in $\Lambda^i_0 \g_{-1}$, with weight $\lambda_i$. In particular $\det L$
is the highest weight vector in $\Lambda_0^n \g_{-1}$, so that $(\det L)^d$
is the highest weight vector of some irreducible subrepresentation
$V \subset S^d \Lambda^n_0 \g_{-1}$.
Since $\Sp (\g_{-1})$ acts transitively on $\LG(\g_{-1})$,
it follows that $(\det L')^d \subset V$ as well for every other Lagrangian $n$--plane
$L' \subset \g_{-1}$. Finally, since the span of all such elements is a sub-representation,
it must coincide with $V$ by irreducibility, thus proving (2).
\end{proof}

\begin{lemma}\label{lem7}
Sending a Lagrangian subspace $L \subset \g_{-1}$ to
its determinant $\det L \subset \Lambda^n_0 \g_{-1}$ defines
an $\Sp (\g_{-1})$-equivariant projective embedding
\begin{equation}\label{eq:plucker}
 \iota : \LG (\g_{-1} )\to \PP \Lambda^n_0 \g_{-1}
\end{equation}
onto the unique closed $\Sp (\g_{-1})$-orbit.
Furthermore, for each $d \ge 0$ the natural map
$$ S^d \Lambda^n_0 \g_{-1}^* \to \Gamma(\LG (\g_{-1}), \iota^*\OO(d)) $$
is $\Sp (\g_{-1})$-equivariant and restricts to an
isomorphism on $S^d_0 \Lambda^n_0 \g_{-1}^* \subset S^d\Lambda^n_0\g_{-1}^*$.
\end{lemma}

Let us first explain that given a point of $\LG(\g_{-1})$
corresponding to a Lagrangian $L \subset \g_{-1}$,
the fibre of $\iota^* \OO(d)$ consists of degree $d$
homogeneous polynomials on the rank one vector space $\det L \subset \Lambda^n_0\g_{-1}$,
and thus identifies with the \emph{dual} of $(\det L)^d \subset S^d \Lambda^n_0\g_{-1}$.
Hence, every element of $S^d\Lambda^n_0\g_{-1}^*$ defines by restriction an element
in the fibre of $\iota^*\OO(d)$ over any point of $\LG (\g_{-1})$,
thus giving rise to a global section. Factorisation is then immediate,
by point (2) of Lemma \ref{lem:sp-fundamental}. The non-trivial statement
is that we do get an isomorphism.
\begin{proof}
We use the setup introduced in the proof of Lemma \ref{lem:sp-fundamental}.
We know that $\Lambda^n_0\g_{-1}$ is the irreducible
representation with highest weight $\lambda_n$,
and $\det L$ is the highest weight line, where $L = \langle e_1,\dots,e_n\rangle$.
The stabiliser of $\det L$ is a fundamental parabolic subgroup $Q \subset \Sp (\g_{-1})$
with Lie algebra $\fq$ whose Levi factor $\fq_0$ is a reductive Lie algebra with
$\fh$ as a Cartan subalgebra, and a root system generated by the simple roots
$\alpha_1,\dots,\alpha_{n-1}$. In particular, the simple highest weight $U(\fq_0)$-module
with highest weight $d\lambda_n$, $d \ge0$,
is one-dimensional, and we denote it by $\C_{d\lambda_n}$. Letting the unipotent
radical act trivially, we inflate it to a $U(\fq)$-module, and furthermore view it
as a representation of $Q$ (note that
the maximal torus in the Levi factor $Q_0$ corresponding
to $\fh$ is the same as in the simply-connected group $\Sp(\g_{-1})$, so its character
lattice coincides with the full integral weight lattice, in particular containing $d\lambda_n$).

Now, since $\iota$ is well-defined and $\Sp(\g_{-1})$-equivariant,
it maps $\LG (\g_{-1})$ onto the highest weight orbit $\Sp(\g_{-1})/Q$ in $\PP\Lambda^n_0 \g_{-1}$.
Furthermore, since $L$ is precisely the space
of vectors $v \in \g_{-1}$ such that $v \wedge \det L = 0$,
the map $\iota$ is injective.
We may identify $\iota^*\OO(d)$ with the associated bundle $\Sp(\g_{-1})\times^Q \C_{d\lambda_n}$,
and then it follows from the Borel--Weil theorem that
the space of its global sections is isomorphic, as a representation of $\Sp(\g_{-1})$,
to the irreducible representation with highest weight $d\lambda_n$, i.e.,
$S^d_0 \Lambda^n\g_{-1}^*$.
\end{proof}

In particular, for $d=1$ we find that the natural map discussed above
gives a bijection between the dual of the embedding space $\Lambda^n_0\g^{-1}$
and the space of global sections of $\iota^* \OO(1)$.
The standard terminology is: \emph{linear non-degeneracy} for injectivity at $d=1$,
and \emph{projective normality} for surjectivity at $d > 0$. %Thus we have:
\begin{corollary}
The Pl\"ucker embedding $\iota$ is linearly non-degenerate and projectively normal.
\end{corollary}

Let us also observe that the kernel of the natural map $S^d \Lambda^n_0\g_{-1}
\to \Gamma(\LG (\g_{-1}),\iota^*\OO(d))$ consists of
homogeneous degree $d$ polynomials on the Pl\"ucker space which
vanish on $\iota(\LG (\g_{-1}))$ (or, more precisely, on its affine cone).
We denote this kernel by $I_d$. The direct sum
\begin{equation}\label{eqHEART}
I=\bigoplus_d I_d
\end{equation}
forms the homogeneous ideal of $\iota(\LG(\g_{-1}))$
in
$S^\bullet \Lambda_0^n \g_{-1}^*$.
\begin{corollary}\label{cor:sections-polys}
For each $d \ge 0$ we have natural $\Sp(\g_{-1})$-equivariant isomorphisms
$$ \Gamma(\LG (\g_{-1}), \iota^*\OO(d)) \simeq S^d_0 \Lambda^n_0\g_{-1}^*
\simeq S^d\Lambda^n_0\g_{-1}^* / I_d $$
where $I_d$ is the degree $d$ subspace of the
homogeneous ideal of $\iota(\LG(\g_{-1}))$.
\end{corollary}

As we have already remarked, none of the above requires us to work
with a bi-Lagrangian decomposition of $\g_{-1}$. Nevertheless, it
is sometimes convenient to fix one for computational purposes. The
additional structure it induces is summed up in the following Lemma.
We omit its proof, since it is straightforward and not essential for our purpose.
\begin{lemma}\label{lem:bilag-more}
The choice of a bi-Lagrangian decomposition as in
Lemma \ref{lem:bilag-stuff}
induces an identification
\begin{equation}\label{eqBOXBOX}
\Lambda^n_0 \g_{-1} \simeq \bigoplus_{i=0}^n S^2_0 \Lambda^i L^*
\end{equation}
equivariant under $\SL(L) \subset \Sp(\g_{-1})$. Furthermore,
the restriction of the Pl\"ucker embedding \eqref{eq:plucker} to the $\GL(L)$-invariant dense open
subset defined in Lemma \ref{lem:bilag-stuff} is
\begin{equation}\label{eqASTERIX}
 \iota : S^2 L^* \to \PP \Lambda^n_0 \g_{-1},\quad
\iota(\varphi)=[ 1 : \varphi : \Lambda^2\varphi : \cdots : \Lambda^{n-1}\varphi : \det\varphi ]
\end{equation}
where $\Lambda^i\varphi \in S^2_0 \Lambda^i L^*$ may be viewed
as the matrix of $i\Th$ minors of $\varphi\in S^2L^*$.
\end{lemma}

\subsection{Hypersurfaces and invariants}
We will now discuss invariant hypersurfaces in $\LG (\g_{-1})$
and their relation to $G_0^\ss$-invariant elements in $S^\bullet \Lambda^n_0 \g_{-1}^*$.
The passage to invariants is most natural if one
works with \emph{effective divisors} on $\LG (\g_{-1})$ instead
of one--codimensional submanifolds. As the term may not be completely
familiar to the broad differential-geometric audience, we shall instead
repurpose the term \emph{hypersurface}.
\begin{definition}\
\begin{enumerate}
\item
A \emph{hypersurface} in $\LG (\g_{-1})$ is a finite formal combination
$\sum m_i Z_i$ where $Z_i \subset \LG (\g_{-1})$ are
closed, irreducible, one--codimensional  subvarieties, and $m_i$ are positive integers.
\item
The \emph{hypersurface cut out by} a section $f \in \Gamma(\LG (\g_{-1}), \iota^*\OO(d))$,
$f\neq 0$, $d>0$ is $\sum m_i Z_i$ where the $Z_i$ are the irreducible components of
the zero locus of $f$, while $m_i$ is the order of vanishing of $f$ at a general point of $Z_i$.
\end{enumerate}
\end{definition}

The following fact justifies our choices. Being  entirely standard,   we just
  sketch   its proof to the reader convenience.
\begin{lemma}\label{lem:divisors}
The set of hypersurfaces in $\LG (\g_{-1})$ is in one-to-one correspondence with
the disjoint union of $\PP \Gamma(\LG(\g_{-1}), \iota^*\OO(d))$ for all $d > 0$.
\end{lemma}
\begin{proof}
Given a hypersurface $\sum m_i Z_i$, we can find
an open cover $\LG (\g_{-1}) = \bigcup U_\alpha$ and meromorphic functions
$f_\alpha$ such that $f_\alpha$ is analytic on $U_\alpha$,
vanishes precisely to order $m_i$ at a general point of $Z_i \cap U_\alpha$,
and the zeros of $f_\alpha|_{U_\alpha}$ are contained in $\bigcup Z_i$.
Then the transition functions $(f_\alpha / f_\beta)|_{U_\alpha\cap U_\beta}$ define a
\v{C}ech cocycle of invertible analytic functions, and thus a line bundle
$\mathcal{L}$. The collection $f_\alpha|_{U_\alpha}$ may be then viewed
as defining an element $f \in \Gamma(\LG (\g_{-1}), \mathcal{L})$. It is easy to see
that any other choice of a covering
and transition functions would lead to the same class in $H^1(\LG( \g_{-1}), \OO^\times)$,
and thus to an isomorphic line bundle; furthermore, two sections $f, g$ of $\mathcal{L}$
cutting out the same hypersurface give rise to a global analytic function $f/g$,
necessarily constant.

It remains to show that every line bundle over $\LG (\g_{-1})$ is isomorphic
to $\iota^*\OO(d)$ for some $d$. One first checks that for every $g \in \Sp(\g_{-1})$
and every line bundle $\mathcal{L}$ over $\LG(\g_{-1})$ there is an isomorphism
$\phi_g:g^*\mathcal{L}\simeq\mathcal{L}$: since $\Sp(\g_{-1})$ is connected,
this follows from discreteness of
the Picard group of $\LG(\g_{-1})$,
a consequence of $H^1(\LG (\g_{-1}),\OO)=0$ as given by the Bott--Borel--Weil
Theorem. Then, for each $\mathcal{L}$ one considers the Lie group
$H_{\mathcal{L}}$ consisting of pairs $(g,\phi_g)$ as above,
with the obvious multiplication and a
forgetful homomorphism $H_{\mathcal{L}} \to \Sp(\g_{-1})$. This group
acts on $\LG (\g_{-1})$ as well as on $\mathcal{L}$ in the natural way.
It is a central extension of $\Sp(\g_{-1})$ by $\C^\times$, and corresponds infinitesimally
to a central extension of Lie algebras. But since $\sp(\g_{-1})$ is simple,
the latter extension is necessarily split. By simply-connectedness of $\Sp(\g_{-1})$
the splitting homomorphism may be then integrated to $\Sp(\g_{-1}) \to H_{\mathcal{L}}$,
providing an action of $\Sp(\g_{-1})$ on $\mathcal{L}$. This proves that
every line bundle over $\LG(\g_{-1})$ is equivariant, i.e., admits a compatible $\Sp(\g_{-1})$-action.
Finally, via the associated bundle construction,
equivariant line bundles are classified up to isomorphism by one-dimensional
representations of the parabolic subgroup $Q \subset \Sp(\g_{-1})$
stabilising a point in $\LG (\g_{-1})$ (as in the proof of Lemma \ref{lem7}),
and thus by the characters of the central torus of the Levi factor of $Q$.
Since $Q$ is the stabiliser of the highest weight line in a \emph{fundamental}
representation, it follows that the central torus of its Levi factor has rank one.
Furthermore, $\iota^*\OO(1)$ corresponds to
the tautological representation of $\C^\times$, whence every other equivariant line bundle is
its power.
\end{proof}
Now we formalise properly the notion of \emph{degree}, already discussed in Subsection \ref{secBasics}.
\begin{definition}\label{defDegree}
A hypersurface in $\LG(\g_{-1})$ has degree $d > 0$ if it is cut out by
a global section of $\iota^* \OO(d)$.
\end{definition}
\begin{remark}\label{RemAsterix}
Using the description of the Pl\"ucker embedding given in Lemma \ref{lem:bilag-more},
one may verify that   Definition \ref{defDegree} above is compatible with the one given
informally in Subsection \ref{ss:from-pdes}.
\end{remark}

In order to handle the action of $G_0\subset\CSp(\g_{-1})$ rather than
just $G_0^\ss\subset\Sp(\g_{-1})$,
we need the following facts. Recall first that we may view $\CSp(\g_{-1})$ as a
central extension of $\Sp(\g_{-1})$ by $\C^\times$.
\begin{lemma}\label{Lem10}\
\begin{enumerate}
\item The $\Sp(\g_{-1})$-action on $\LG(\g_{-1})$ extends trivially to a
$\CSp(\g_{-1})$-action
(i.e., the centre $\C^\times$ acts trivially).
\item The $\CSp(\g_{-1})$-action on $\LG(\g_{-1})$ lifts naturally to an action on $\iota^*\OO(d)$
for all $d$.
\item The isomorphisms $\Gamma(\LG(\g_{-1}),\iota^*\OO(d))\simeq S^d_0 \Lambda^n_0 \g_{-1}^*
\simeq S^d \Lambda^n_0 \g_{-1}^*/I_d$, $d>0$, are $\CSp(\g_{-1})$-equivariant.
\end{enumerate}
\end{lemma}
\begin{proof}Straightforward, given
the induced natural action of $\CSp(\g_{-1})$ on the Pl\"ucker space $\Lambda^n_0\g_{-1}$.
\end{proof}

We are now ready to spell out the invariance condition for a hypersurface
in terms of the corresponding section.
\begin{lemma}\label{lem:inv-hypers}
Let $f \in \Gamma(\LG(\g_{-1}), \iota^* \OO(d))$, $f\neq 0$. Then
the hypersurface cut out by $f$ is $G_0$-invariant if and only if
there exists a homomorphism $\xi : G_0 \to \C^\times$ such that
$g^* f = \xi(g) f$ for all $g \in G_0$.
\end{lemma}
\begin{proof}
By Lemma \ref{lem:divisors} we find that $G_0$-invariance
of the hypersurface cut out by $f$ is equivalent to
the existence, for each $g \in G_0$, of a scaling factor $c_{g} \in \C^\times$
such that $g^*f = c_{g}f$. Furthermore, $c_g$ is uniquely determined by $g$,
and we have $c_{gh}f = h^* g^* f = h^* (c_g f) = c_g c_h f$ whence $g\mapsto c_g  $
 is  a character $\xi$.
\end{proof}

Recall now the decomposition $G_0 = G_0^\ss \cdot \tilde T$, where $G_0^\ss$ is semi-simple,
and $\tilde T$ is a torus (cf. \ref{eqDecGOGOss}).
It follows that characters of $G_0$ factor through $T = \tilde T / (G_0^\ss \cap \tilde T)$,
so that $f$ cuts out a $G_0$-invariant hypersurface if and only if it is $G_0^\ss$-invariant,
and transforms under the action of $T$ via some character $\xi \in \widehat T$.
We now apply the identification spelt out in Corollary \ref{cor:sections-polys}:
\begin{lemma}
\label{lem:ring-invs}
Let $R = (S^\bullet \Lambda^n_0 \g_{-1}^* / I)^{G_0^\ss}$
be the ring of $G_0^\ss$-invariants. Then:
\begin{enumerate}
\item $G_0$ acts on $R$, and the action factors through $T$,
\item $R = \bigoplus_{\xi \in \widehat T} R^\xi$ where $T$ acts on $R^\xi$ via $\xi : T \to \C^\times$,
\item the set of $G_0$-invariant hypersurfaces in $\LG(\g_{-1})$ is in one-to-one correspondence
with the disjoint union of $\PP R^\xi$ for all $\xi \in \widehat T \setminus \{0\}$.
\end{enumerate}
\end{lemma}
\begin{proof}
Part (1) follows from centrality of $\tilde T$ in $G_0$, and part (2) from the finite-dimensionality
of the graded summands of $R$. Part (3) is then a consequence of Corollary \ref{cor:sections-polys}
and Lemma \ref{lem:inv-hypers}.
\end{proof}
Let us point out that, since $G_0^\ss$ is semi-simple, hence reductive,
$I_d \subset S^d\Lambda^n_0\g_{-1}$ admits a $G_0^\ss$-invariant complement,
thus allowing us to identify
\begin{equation}\label{eq:quotients}
(S^\bullet \Lambda^n_0\g_{-1}^*/I)^{G_0^\ss}
= (S^\bullet \Lambda^n_0\g_{-1}^*)^{G_0^\ss} / I^{G_0^\ss}.
\end{equation}
That is, we may represent a $G_0$-invariant hypersurface by
a $G_0^\ss$-invariant homogeneous polynomial on the Pl\"ucker space.
Finally we remark that the `standard' grading on $R$,
with $(S^d \Lambda^n_0\g_{-1}^* / I_d)^{G_0^\ss}$ placed in
degree $d$, may be recovered from the rescaled $T$-weights $\xi\in\widehat T$.
\begin{lemma}\label{lem:j}
Let $j:\C^\times \hookrightarrow \CSp(\g_{-1})$ be the one--parameter subgroup
acting by scaling on $\g_{-1}$. Then
$j$ factors through $\tilde T$,
and its composite $\bar j : \C^\times \to T$ with the projection $\tilde T \to T$
induces a
homomorphism $\bar j^* : \widehat T \to \ZZ$
such that $R^\xi \subset S^d \Lambda_0^n\g_{-1}^*$
if and only if $\bar j^*\xi = -nd$.
\end{lemma}
\begin{proof}
Since $G_0$ contains the full maximal torus of $G$ corresponding to the Cartan
subalgebra we used to define the contact grading \eqref{eqContGrad} on $\g$, it follows that in particular
it contains a one-parameter subgroup acting by scaling on $\g_{-1}$, necessarily
coinciding with $j$. Being central, it factors through $\tilde T$.
As it acts on $S^d\Lambda_0^n\g_{-1}^*$
with weight $-nd$, the claim follows.
\end{proof}

\subsection{Back to the adjoint variety}
Let us now return to
the notion of the adjoint variety
$X$ of $\g$ (recall Definition \ref{def:adjoint}).
We have already introduced the bundle
$$ X^{(1)} \to X $$
whose fibre at $x \in X$ is the set of Lagrangian subspaces in $\CC_x$
(see Subsection \ref{secBasics}, \eqref{eqSTARSTAR}), together with
the natural
$G$-invariant identification $X^{(1)} \simeq G \times^P \LG(\g_{-1})$
(see Subsection \ref{ss:x1-ass-bun}, \eqref{eqSTAR})
induced by the identification $\CC \simeq G\times^P \g_{-1}$.
We extend our notion of a hypersurface from the fibre $\LG(\g_{-1})\simeq X^{(1)}_o$
to the entire bundle $X^{(1)}$. (It might be useful to remind that
$X^{(1)}$ is a projective manifold, so Chow's theorem still
lets us treat closed analytic subsets as algebraic subvarieties).
\begin{definition}\
\begin{enumerate}
\item A hypersurface in $X^{(1)}$ is a finite formal combination $\sum m_i Y_i$
where $Y_i \subset X^{(1)}$ are closed, irreducible, codimension $1$ subvarieties,
and $m_i$ are positive integers.
\item $\Inv(X,G)$ is the set of $G$-invariant hypersurfaces in $X^{(1)}$.
\end{enumerate}
\end{definition}

A general hypersurface in $X^{(1)}$ may have components projecting onto a
codimension $1$ subvariety of $X$. Clearly, this cannot occur in the $G$-invariant
case, where we do obtain a family of hypersurfaces in the fibres, all conjugate
to a single $G_0$-invariant hypersurface in $\LG(\g_{-1})$.
\begin{lemma}\label{lem:fibre-at-origin}
Let $\sum m_i Y_i$ be a $G$-invariant hypersurface in $X^{(1)}$.
Then $\sum m_i (Y_i \cap X^{(1)}_o)$ is a $G_0$-invariant hypersurface in $X^{(1)}_o = \LG(\g_{-1})$.
\end{lemma}
\begin{proof}
Let $Y$ be a $G$-invariant
one--codimensional subvariety in $X^{(1)}$. We need to check that $Y \cap X^{(1)}_o$
is a $G_0$-invariant codimension $1$ subvariety in $X^{(1)}_o$.
By $G$-invariance, the projection $Y \to X$ is surjective,
so that $Y \cap X^{(1)}_x$ is codimension $1$ in $X^{(1)}_x$ for general $x \in X$.
Again by $G$-invariance, $Y \cap X^{(1)}_x$ is codimension $1$ in $X^{(1)}_x$
for every $x \in X$, in particular for $x=o$. Finally, $G_0$-invariance of $Y \cap X^{(1)}_o$
is immediate from $G$-invariance of $Y$ and $P$-invariance of $o$.
\end{proof}

\begin{definition}\label{defDegreeGlob}\
\begin{enumerate}
\item
The \emph{fibre at the origin} of a $G$-invariant hypersurface in $X^{(1)}$
is the $G_0$-invariant hypersurface in $\LG(\g_{-1})$ arising as in Lemma
\ref{lem:fibre-at-origin}.
\item
The \emph{degree} of a $G$-invariant hypersurface in $X^{(1)}$ is
the degree of its fibre at the origin.
\end{enumerate}
\end{definition}
We are now able to exhibit  the main result of this long pedagogical section, that is, Proposition \ref{pro:invxg} below. It finally provides  the necessary  interpretation of $G$-invariant hypersurfaces in $X^{(1)}$
in terms of  projectivised $T$-weight subspaces in the ring of $G_0^\ss$-invariants
in the homogeneous coordinate ring of the Pl\"ucker-embedded Lagrangian
Grassmannian.
\begin{proposition}\label{pro:invxg}
There is a natural bijection
$$ \Inv(X,G) \simeq \coprod_{\xi \in \widehat T \setminus \{0\}} \PP R^\xi. $$
It identifies invariant hypersurfaces of degree $d>0$ in $X^{(1)}$
with points of the disjoint union of $\PP R^\xi$ such that $\bar j^*\xi=-nd$,
where $\bar j^* : \widehat T \to \ZZ$ is the homomorphism of Lemma \ref{lem:j}.
\end{proposition}
\begin{proof}
This follows directly from Lemmas \ref{lem:ring-invs}, \ref{lem:j}
and \ref{lem:fibre-at-origin}.
\end{proof}

\section{Proof of Theorem 1}\label{sec:proof}

\subsection{Reformulation}

Let $R_d \subset R$ % = (S^d \Lambda^n_0\g_{-1}^*/I)^{G_0^\ss}$
be the degree $d$ component of the ring of $G_0^\ss$-invariants in the
homogeneous coordinate ring of $\LG(\g_{-1})$. We then have, according to
Lemma \ref{lem:j}, a decomposition
$$ R_d = \bigoplus_{\substack{\xi \in \widehat T \\ \bar j^*\xi = -nd}} R^\xi $$
into weight subspaces for the quotient torus $T = G_0 / G_0^\ss$.
Then, by Proposition \ref{pro:invxg},
Theorem \ref{thm:invs} is equivalent to Theorem \ref{thm:invs2} below.
\begin{theorem}\label{thm:invs2}
The minimal degree $d$ such that
$R_d \neq 0$ is given by the first row of the following table,
while the second row gives the dimensions of its $T$-homogeneous summands:
(entries marked with an asterisk are conjectural).
\begin{center}\fbox{\begin{tabular}{r|cccccccc}
\normalfont{type} & $\Asf_.$ & $\Bsf_.$ & $\Dsf_.$ & $\Esf_6$ &$\Esf_7$ &$\Esf_8$ & $\Fsf_4$ & $\Gsf_2$ \\
\hline
  & 1 & 4 & 2 &2 &2 & 2 & 4 & 3\\
  & \{1,1\} & \{\textrm{unknown}\}  & \{1*\} &\{1\} &\{1\} & \{1\} & \{1\} & \{1\}\\
\end{tabular}}\end{center}
\end{theorem}

We also extract additional information about $R$ in types $\Asf$ and $\Gsf$.
\begin{proposition}\label{pro:AG}\
\begin{enumerate}
\item In type $\Asf$ the ring of invariants $R$ is generated by a pair
of elements of degree $1$ with distinct $T$-weights.
\item In type $\Gsf_2$ the ring of invariants $R$ is generated by a single
element of degree $3$.
\end{enumerate}
\end{proposition}

We thus proceed to prove Theorem \ref{thm:invs2} and Proposition \ref{pro:AG},
first outlining the general strategy.

\subsection{Strategy}\label{ss:strategy}
The approach will differ depending on the degree
as stated in Theorem \ref{thm:invs2} and on the Cartan type of $\g$.
The outliers, $\Asf_\ell$ and $\Gsf_2$, will be treated
separately at the very beginning (these are the easy ones).
Now, the remaining types are organised according to two binary criteria:
\begin{enumerate}
\item degree: quadric ($\Dsf_\ell$, $\Esf_6$, $\Esf_7$, $\Esf_8$) or quartic ($\Bsf_\ell, \Fsf_4$),
and
\item type: classical ($\Bsf_\ell$, $\Dsf_\ell$) or exceptional ($\Esf_6$, $\Esf_7$, $\Esf_8$, $\Fsf_4$).
\end{enumerate}
For classical algebras, we give a constructive proof involving an explicit invariant
in $R_d$, $d = 2$ or $4$. For the exceptional ones, we argue non-constructively by computing the
dimension of $R_d$ using representation--theoretic methods. In type $\Fsf_4$ we use
the standard way of branching a representation of $\spp_n$ to $\g_0^\ss$ in terms of
formal characters. Since the complexity of this approach is roughly controlled by the
size of the Weyl group of $\g$, it is practically impossible to apply to
$\Esf_8$ (with its Weyl group of size nearly 7 million, compared to 1152 for $\Fsf_4$).
Fortunately, a more refined method may be applied to find \emph{quadric} invariants,
involving only the quotient of the Weyl group of $\g$ by that of $\g_0$ (for $\Esf_8$
there are only 240 cosets).

In any case, it is not difficult
to construct \emph{candidates} for a nontrivial element of $R_d$. Indeed,
due to the isomorphism \eqref{eq:quotients}, to give a nontrivial element of $R_d$ is the
same as to give a $G_0^\ss$-invariant in $S^d \Lambda^n_0 \g_{-1}^*$
that is not contained in the ideal $I$
(i.e., does not vanish on the Pl\"ucker-embedded $\LG(\g_{-1}) \subset \p \Lambda^n_0\g_{-1}$).
We will now describe a way to produce $G_0^\ss$-invariants in $S^d \Lambda^n_0\g_{-1}^*$.
In the case of classical algebras we will be able to show explicitly that they do not belong to $I$.

In the \emph{quadric} case we are dealing with algebras of type $\Dsf$ and $\Esf$. It is an important
observation that in all these cases $n$ is \emph{even}. As a consequence,
the wedge product map $\Lambda^n_0 \g_{-1} \otimes \Lambda^n_0 \g_{-1} \to \det \g_{-1} \simeq \C$
defines an $\spp(\g_{-1})$-invariant quadratic form $b \in S^2\g_{-1}^*$. Of course, $b$
vanishes on $\LG(\g_{-1})$ (the symplectic group acts transitively on the Lagrangian Grassmannian, see Subsection \ref{ss:x1-ass-bun}).
Now, $\Lambda^n_0\g_{-1}$
decomposes under $\g_0^\ss$ into a direct sum of invariant subspaces: for each of these
the restriction of $b$ is either non-degenerate or zero. The isotropic summands
come in dual pairs, so that adding them produces a $G_0^\ss$-invariant \emph{orthogonal}
decomposition of $\Lambda^n_0\g_{-1}$
(we will see that the Pl\"ucker space is not a sum of two irreducible isotropic subspaces).
Now, the restriction of $b$ to each orthogonal summand may be again extended trivially to
all of $\Lambda^n_0\g_{-1}$
producing a $G_0^\ss$-invariant quadric. Their sum yields $b$ and thus vanishes
on $\LG(\g_{-1})$, but one may expect that \emph{not all} such restrictions vanish on their own.

In the \emph{quartic} case we may exploit a universal construction. It is well known (see, for instance,
\cite{MR2905245}) that the $G_0^\ss$-action on $\g_{-1}$ has quartic invariant, which may be
defined for $x \in \g_{-1}$ as $q(x) = \ad^4_x \in \Hom(\g_2,\g_{-2}) \simeq \C$. Now,
a $G_0^\ss$-invariant quartic on $\Lambda^n_0 \g_{-1}$ is given by $q^n$ under the natural
map
$$ S^n S^4 \g_{-1}^* \to S^4 \Lambda^n \g_{-1}^*. $$
In fact, we will show that $q^n$ does not vanish on $\LG(\g_{-1})$ for both
$\Bsf_\ell$ (where it is the lowest degree invariant) and $\Dsf_\ell$ (where
it defines a quartic invariant \emph{independent} from the square of the quadric).

%The next subsections are organised as follows:\begin{itemize}
%\item $\Asf_\ell$ and $\Gsf_\ell$,
%\item constructive proof for $\Dsf_\ell$,
%\item constructive proof for $\Bsf_\ell$,
%\item representation-theoretic setup for the exceptional algebras,
%\item computation of $\dim R_2$ for $\Esf_6$, $\Esf_7$ and $\Esf_8$,
%\item computation of $\dim R_4$ for $\Fsf_4$.
%\end{itemize}
%

\subsection{Proof of Theorem \ref{thm:invs} in type $\Asf$}
In this section we work out step--by--step the case $\Asf_{n+1}$, i.e., when $G=\PGL_{n+2}$, which is perhaps the simplest one. This `toy model' will help the reader to better understand how the program sketched in the introductory Subsection \ref{subFirstLook} can be applied in practice. This is also the only case when the central torus $\tilde T$ has rank two, which is another good reason to develop it in details.\par
Let $\PGL_{n+2}$ act naturally on the projective space $\p^{n+1}$.
\begin{lemma}\label{lemTypeA}
The $(2n+1)$--dimensional contact manifold $\p T^*\p^{n+1}$
is  precisely the adjoint contact variety $X$ of $\PGL_{n+2}$.
\end{lemma}
\begin{proof}
A point $L\in \p^{n+1}$ is a line in $\C^{n+2}$, and an element $H\in\p T^*_L\p^{n+1}$ is a tangent hyperplane to  $\p^{n+1}$ at $L$. As such, $H$ is an hyperplane in $\C^{n+2}$ containing the line $L$. In other words,   $\p T^*\p^{n+1}$ can be identified with the space
\begin{equation*}%\label{}
 \{  (L,H)\mid L\subset H\}\subset \p^{n+1}\times\p^{n+1\,\ast}\,
\end{equation*}
of $(1,n+1)$--flags in $\C^{n+2}$.\par
Fix standard coordinates $e_{1},\ldots, e_{n+2}$ on $\C^{n+2}$, together with their duals $\epsilon^i$, and set the point
\begin{equation*}%\label{}
o:=(L_0=\Span{e_1}, H_0=\Span{e_1,\ldots,e_{n+1}}=\ker\epsilon^{n+2})
\end{equation*}
as the origin of  $\p T^*\p^{n+1}$.  The Lie algebra of the stabiliser   of $o$ is
\begin{equation*}%\label{}
\gp= \left\{  \left(\begin{array}{ccc}\lambda & v & \mu \\0 &A & w \\0 & 0 & \nu\end{array}\right) \mid A\in  \gl_n , v,w\in\C^n,  \lambda+\nu+\tr A=0 \right\}\, .
\end{equation*}
Since $\dim\sll_{n+2}-\dim\gp=((n+2)^2-1)-(n^2+2n+2)=2n+1$, the  $\PGL_{n+2}$--orbit through $o$ is open in $\p T^*\p^{n+1}$.\par
Now we show that the   map
\begin{eqnarray*}
\p T^*\p^{n+1} &\stackrel{F}{\longrightarrow} & \p(\sll_{n+2})\subset\p(\C^{n+2}\otimes\C^{n+2\ast})\, ,\\
(L=\Span{v},H=\ker\varphi) &\longmapsto & [v\otimes\varphi]\,
\end{eqnarray*}
is well--defined, injective, $\PGL_{n+2}$--equivariant, and it defines a contactomorphism on its image, which is precisely $X$.\par
The class $[v\otimes\varphi]$  is well--defined because both $v$ and $\varphi$ are defined up to a  projective factor, and  $v\otimes\varphi$ indeed belongs to $\sll_{n+2}$, since from $L\subset H$ it follows that   $\tr (v\otimes\varphi) =\varphi(v)= 0$.\par
By construction,
\begin{equation*}
F(o)=[e_1\otimes\epsilon^{n+2}]\, ,
\end{equation*}
where
  $e_1\otimes\epsilon^{n+2}$ is the highest weight vector of $\sll_{n+2}$, whence
\begin{equation*}%\label{}
\PGL_{n+2}\cdot [e_1\otimes\epsilon^{n+2}]= X\, ,
\end{equation*}
by the very definition of adjoint variety (see Definition \ref{def:adjoint}).\par
The $\PGL_{n+2}$--equivariancy of $F$ is obvious, since
\begin{equation*}%\label{}
g\cdot (L,H)=(g(L),g(H))=(\Span{g(v)},\ker g^*(\varphi))\longmapsto [g(v)\otimes g^*(\varphi)]=g\cdot [v\otimes\varphi]\, ,\quad\forall g\in \PGL_{n+2}\, .
\end{equation*}
Moreover, $F$ is injective being the restriction of the Segre embedding  $\p^{n+1}\times\p^{n+1\,\ast}\subset \p(\C^{n+2}\otimes\C^{n+2\ast})$, and its image coincides with $X$. Indeed, if $[h]\in \p(\sll_{n+2})$, where $h$ is a rank--one homomorphism, then $(\textrm{im}\, h, \ker h)\in \p T^*\p^{n+1}$ and
\begin{equation*}%\label{}
[h]=F((\textrm{im}\, h, \ker h))\, .
\end{equation*}
It remains to prove that $F$ realises a contactomorphism between the contact structures on $\p T^*\p^{n+1}$ and $X$. By homogeneity, we can simply show that $T_oF$ sends the contact hyperplane
\begin{equation*}%\label{}
H_0\oplus T_{H_0}(\p T_{L_0}^*\p^{n+1})\subset T_{L_0}\p^{n+1}\oplus T_{H_0}(\p T_{L_0}^*\p^{n+1})=T_o(\p T^*\p^{n+1})
\end{equation*}
to the contact hyperplane of $T_{F(o)}X$. The latter is better described in terms of the cone $\widehat{X}$ over $X$, inside $\sll_{n+2}$: it is the subspace
\begin{equation*}%\label{}
 {\ker\ad_{e_1\otimes\epsilon^{n+2}}}{ }\subset {[\sll_{n+2},e_1\otimes\epsilon^{n+2}]}{ }=T_{e_1\otimes\epsilon^{n+2}}\widehat{X}\, .
\end{equation*}
A curve $\gamma(t):=(\Span{v_t},\ker\phi_t)$ belongs to the contact plane at $o$ if and only if $\epsilon^{n+2}(v_0')=0$, that is, if the horizontal projection of $\gamma(t)$ keeps, to first order, the line $\Span{v_t}$ inside the hyperplane $\ker\phi_t$. Observe that, since $\phi'_0$ is the velocity of a curve of hyperplanes containing $L_0$, we have also  $ \phi'_0(e_1)=0$. Consider now the curve
\begin{equation}\label{eqSprojecCurve}
 t\longmapsto e_1\otimes\epsilon^{n+2} + (v'_0\otimes\epsilon^{n+2}+e_1\otimes\phi'_0)t+o(t^2)\,
\end{equation}
in $\sll_{n+2}$, whose projectivisation is precisely $F_*\gamma$.
Since
\begin{equation*}%\label{}
\ad_{e_1\otimes\epsilon^{n+2}}(  v'_0\otimes\epsilon^{n+2}+e_1\otimes\phi'_0 ) =(\epsilon^{n+2}(v_0')+ \phi'_0(e_1))e_1\otimes\epsilon^{n+2}=0\, ,
\end{equation*}
the velocity at $0$ of \eqref{eqSprojecCurve}  belongs to the contact hyperplane of $\widehat{X}$ at $e_1\otimes\epsilon^{n+2}$, whence the velocity at $0$ of    $F_*\gamma$ at $F(o)$ belongs to the contact hyperplane of $X$ at $F(o)$.
\end{proof}
Recall (see   Subsection \ref{subContGrad}) the notion of contact grading.
\begin{corollary}
 The contact grading \eqref{eqContGrad} of $\sll_{n+2}$ is
 \begin{equation}\label{eq:gradingTypeA}
\sll_{n+2}= \underset{ {  \g_{-2}  }}{\underbrace{ \C}} \oplus\underset{ {   \g_{-1} }} {\underbrace{ (\C^n\oplus \C^{n\ast})}} \oplus\underset{ {  \g_0 }}{\underbrace{  (\sll_n\oplus\C^2)}} \oplus\underset{ {   \g_{1} }}{\underbrace{  (\C^{n\ast}\oplus\C^n) }} \oplus\underset{ {   \g_{2} }}{\underbrace{  \C^*}}\, .
\end{equation}
\end{corollary}
\begin{proof}
 The stabiliser   of $o$ is
\begin{equation*}%\label{}
\gp= \left\{  \left(\begin{array}{ccc}\lambda & v & \mu \\0 &A & w \\0 & 0 & \nu\end{array}\right) \mid A\in  \gl_n , v,w\in\C^n,  \lambda+\nu+\tr A=0 \right\}\, .
\end{equation*}
Then
\begin{equation*}%\label{}
 \left(\begin{array}{ccc}\lambda & v & \mu \\0 &A & w \\0 & 0 & \nu\end{array}\right) \longmapsto \left( \underset{ {  \g_0 }}{\underbrace{  A-\frac{1}{n}\tr A\id_n, \tr A, \lambda+\nu}} ,\underset{ {   \g_{1} }}{\underbrace{  \left(\begin{array}{ccc}0 & v & 0 \\0 &0 & 0 \\0 & 0 & 0\end{array}\right),\left(\begin{array}{ccc}0 & 0 & 0 \\0 &0 & w \\0 & 0 & 0\end{array}\right) }} ,\underset{ {   \g_{2} }}{\underbrace{  \mu}} \right)
\end{equation*}
defines an isomorphisms $\gp\cong\g_{0}\oplus\g_1\oplus\g_2$.
\end{proof}
\begin{corollary}
 The twisted symplectic form on  $\g_{-1}$ is the unique one extending the standard pairing between $\C^n$ and $\C^{n\ast}$.
\end{corollary}
\begin{proof}
 Just observe that
 \begin{equation*}%\label{}
\left[ \left(\begin{array}{ccc}0 & 0 & 0 \\v &0 & 0 \\0 & 0 & 0\end{array}\right)    ,   \left(\begin{array}{ccc}0 & 0 & 0 \\0 &0 & 0 \\0 & w & 0\end{array}\right)      \right]=-v\cdot w
\end{equation*}
is the matrix counterpart of the pairing
\begin{equation*}%\label{}
\C^{n}\times \C^{n\ast}\ni(x,\xi)\longmapsto \xi(x)\in\C \qedhere
\end{equation*}
\end{proof}
Let us stress here that
the identification of the particular summands
in \eqref{eq:gradingTypeA} is a representation of the semi-simple part $G_0^\ss \simeq \SL_n$.
Formula \eqref{eqDecGOGOss} reads now
%\begin{equation*}%\label{eqDecGOGOssAnpiu1}
$G_0=\SL_n\cdot \tilde T$,
%\end{equation*}
with a central torus $\tilde T$ of rank $2$,
and a quotient $T = \tilde T / \mu_n$. We describe
the characters of $T$ corresponding to the determinant
representations.
\begin{lemma}\label{lemTorusActionANP1}
There is a lattice isomorphism $\hat T \simeq \ZZ^2$
such that the characters of $\det \C^n$ and $\det \C^{n*}$
are $(1,-1)$ and $(-1,n+1)$ respectively.
\end{lemma}
\begin{proof}
Let us first consider the cover
$\SL_{n+2} \to \PGL_{n+2}$ with kernel $\mu_{n+2}$
(the group of $(n+2)$-nd roots of unity). The
central torus $\tilde T \subset G_0 \subset \PGL_{n+2}$
is a quotient of a torus $\tilde T' \subset \SL_{n+2}$:
$$
\tilde T = \tilde T' / \mu_{n+2},\quad \tilde T' = \{ \mathrm{diag}(\lambda,\mu,\cdots,\mu,\lambda^{-1}\mu^{-n}) \}
$$
and thus its character lattice is of index $n+2$
in the character lattice of $\tilde T'$.
We identify $\Hom(\tilde T',\C^\times)$ with $\ZZ^2$
so that $(a,b) \in \ZZ^2$ sends an element of $\tilde T'$ as above
to $\lambda^a \mu^b$. Then $\tilde T'$
acts on $\C^n$, $\C^{n*}$ with characters
$\eta_1=(1,-1)$, $\eta_2=(1,n+1)$ respectively. Since
$$ \det \left(\begin{matrix} 1 & -1 \\ 1 & n+1 \end{matrix}\right) = n+2, $$
it follows that $\eta_1, \eta_2$
generate a sub-lattice of index $n+2$ in $\Hom(\tilde T', \C^\times)$,
necessarily coinciding with $\Hom(\tilde T, \C^\times)$.
We have $$ T = \tilde T / \mu_n = \tilde T' / (\mu_{n+2} \times \mu_n) $$
where $\mu_n \subset \tilde T'$ consists of elements with $\lambda=1$, $\mu^n=1$.
It follows that $\hat T \subset \Hom(\tilde T,\C^\times)$ is the sub-lattice
consisting of $(a,b) \in \ZZ\langle \eta_1,\eta_2\rangle \subset \ZZ^2$
with $b = 0$ mod $n$, expressed in terms of the basis $\eta_1$, $\eta_2$ as
$$ \hat T = \{ c_1 \eta_1 + c_2 \eta_2\ |\ c_1 = c_2 \mod n \}. $$
Hence $\hat T$ is of index $n$ in $\Hom(\tilde T, \C^\times)$ and may be generated
by $\zeta_1 = \eta_1+(n+1)\eta_2$ and $\zeta_2 = \eta_1+\eta_2$. Now, the
$T$-characters of $\det\C^n$, $\det\C^{n*}$ are
$$ ne_1 = -\zeta_1+(n+1)\zeta_2,\quad ne_2 = \zeta_1 - \zeta_2 $$
respectively.
\end{proof}

Let us now decompose the dual Pl\"ucker space $\Lambda_0^n\g_{-1}^*$
into $G_0^\ss$-irreducible components
(see Subsection \ref{subPluck}). Since $G_0$ preserves the Lagrangian decomposition
$\g_{-1} = \C^n \oplus \C^{n*}$, the decomposition
\eqref{eqBOXBOX} takes a particularly simple form:
\begin{align}
\Lambda^n_0 \g_{-1}^*
 &= \Lambda^n_0 (\C^{n\ast}\oplus \C^{n})\nonumber\\
 &= \bigoplus_{i=0}^nS^2_0(\Lambda^i\C^{n\ast})\subset \bigoplus_{i=0}^n (\Lambda^i\C^{n\ast})^{\otimes 2}= \bigoplus_{i=0}^n (\Lambda^i\C^{n\ast})\otimes (\Lambda^{n-1}\C^{n })\label{eqMysterEmbed},\\
\Lambda^n_0\g_{-1}^* &=\C\oplus S_0^2\C^{n\ast}\oplus\cdots \oplus S_0^2 (\Lambda^{n-1}\C^{n\ast})\oplus S_0^2(\Lambda^n\C^{n\ast})\, .\label{eqDecPluckLinearAnp1}
\end{align}
The step \eqref{eqMysterEmbed} may use some extra comment. First, we decomposed $n$--forms on $\C^{n\ast}\oplus \C^{n}$ as products of forms on each summand, and then we used Poincar\'e duality. Finally,   one checks that  a bilinear form on $\Lambda^i\C^{n}$ belongs to the kernel of \eqref{eqINS} if and only if it is symmetric and trace--free.
So,  \eqref{eqDecPluckLinearAnp1} represents the decomposition of the space of linear functions on the Pl\"ucker embedding space of $X^{(1)}_o$ into $G_0^\ss$--irreducible submodules.
Clearly, there are only two one--dimensional summands in \eqref{eqDecPluckLinearAnp1}:
the first and the last. Using the identification $\hat T = \ZZ^2$ of Lemma
\ref{lemTorusActionANP1} we thus have
\begin{equation*}%\label{eqPezziInTyopeA}
\C=R^{(1,-1)}\, ,\quad S^2(\Lambda^n\C^{n\ast})=R^{(-1,n+1)}.
\end{equation*}
We conclude that $\Inv(X,\Asf_{n+1})$ contains exactly two elements of degree 1, i.e., the lowest--degree invariants we were looking for.\footnote{
We stress that in \eqref{eqDecPluckLinearAnp1} the submodules $S_0^2\C^{n\ast}$ and $S_0^2 (\Lambda^{n-1}\C^{n\ast})$ coincide in fact with  $S^2\C^{n\ast}$ and $S^2 (\Lambda^{n-1}\C^{n\ast})$, respectively, since the entries of a symmetric matrix, as well as of its cofactor matrix are independent. Only for $n\geq 4$ the trace--free submodules can be proper ones. See, e.g., the discussion     \cite{209058}, where the whole decomposition  \eqref{eqDecPluckLinearAnp1}  is obtained.
}
The reader may find their interpretation in terms of PDEs
in Subsection \ref{ss:pdes-a}.

For completeness, we state the following result.
\begin{lemma}\label{lemGeomLagrGrass}
$R^{(1,-1)}$ and $R^{(-1,n+1)}$ generate $R$.
\end{lemma}
\begin{proof}
We work geometrically on the Lagrangian Grassmannian $\LG(\g_{-1})
\subset \PP \Lambda^n_0\g_{-1}$. As above, we use a bi-Lagrangian splitting
$\fg_{-1} = \C^n \oplus \C^{n*}$ as in Lemmas \ref{lem:bilag-stuff} and \ref{lem:bilag-more}.
Recall that we have dense open subsets $U \simeq S^2 \C^{n*}$
and $U' \simeq S^2 \C^n$ in $\LG(\g_{-1})$,
consisting of Lagrangian subspaces $L \subset \g_{-1}$ with non-degenerate
projections onto $\C^n$, respectively $\C^{n*}$.
Let $H = \LG(\g_{-1}) \setminus U$
and $H' = \LG(\g_{-1}) \setminus U'$.
All these subsets are $G_0^\ss$--invariant, and
the action of the latter on $U$ and $U'$ coincides
with its natural linear action on $S^2 \C^{n*}$ and $S^2 \C^{n}$.
We may view $H' \cap U$ as the locus of those symmetric
maps $\C^n \to \C^{n*}$ that are not invertible, and analogously for $H \subset U'$.
It follows that $H' \cap U$ is the hypersurface in $S^2 \C^{n*}$ cut out by
the determinant $\det : S^2 \C^{n*} \to (\det \C^n)^{2} \simeq \C$. Thus
the hyperplane section $H'$ is defined by an element of $R^{(1,-1)} \subset
\Lambda^n_0 \g_{-1}^*$ and, dually, $H$ is defined by an element of $R^{(-1,n+1)}$.

Suppose now $r \in R_d$ is a nonzero invariant and let
$D$ be the associated hypersurface (recall that to us this means an effective divisor)
on $\LG(\g_{-1})$. We may
write $D = aH + \sum b_i Z_i$ where the $Z_i$
are codimension one subvarieties of $\LG(\g_{-1})$ such that $Z_i = \overline{Z_i \cap U}$.
Since $D_U = \sum b_i Z_i \cap U$ is necessarily $G_0^\ss$-invariant, it will
be enough to check that the ring of $\SL_n$--invariants in $\C[S^2\C^{n*}]$ is generated
by $\det$, for then it follows that $D_U$ is a multiple of $H' \cap U$.
That is a classical result (see, e.g., \cite{olver1999classical}).
\end{proof}

\subsection{Proof of Theorem \ref{thm:invs} in type $\Gsf$}
We move now to $\fg$ of type $\Gsf$, that is, to the case when $G$ is the 14--dimensional Lie group $\Gsf_{2}$. Being the smallest amongst the exceptional Lie groups, $\Gsf_2$ is perhaps the best understood one, and this section does not add anything new (see \cite{MR2441524} for a thorough review of $\Gsf_{2}$, and also \cite{MR2322419}).\par
The contact grading \eqref{eqContGrad} reads now
 \begin{equation*}%\label{}
\Lie(\Gsf_2)= \underset{ {  \g_{-2}  }}{\underbrace{ \C}} \oplus\underset{ {   \g_{-1} }} {\underbrace{ S^3\C^2}} \oplus\underset{ {  \g_0 }}{\underbrace{  (\sll_2\oplus\C)}} \oplus\underset{ {   \g_{1} }}{\underbrace{  S^3\C^{2\ast} }} \oplus\underset{ {   \g_{2} }}{\underbrace{  \C^*}}  \, ,
\end{equation*}
where $\fg_{-1}$ is the 4--dimensional irreducible representation of
$G_0^\ss \simeq \SL_2$.\par
Observe that $X^{(1)}_o$ is the 3--dimensional  Grassmannian of Lagrangian 2--planes in a 4--dimensional symplectic space, which is a quadric in $\p^4$ (see, e.g.,  \cite{MR2876965}). Accordingly, $\Lambda_0^2\g_{-1}$ must be an irreducible 5--dimensional $\SL_2$--module, i.e., $\Lambda_0^2\g_{-1}\cong S^4\C^2$, and   \eqref{eqDefR} becomes
\begin{equation*}%\label{}
R = (S^\bullet   (S^4\C^{2*} )/ I)^{\SL_2}  = \C[S^4\C^2]^{\SL_2} / I^{\SL_2}\, .
\end{equation*}
It is a classical result that
\begin{equation*}%\label{}
\C[S^4\C^2]^{\SL_2} =\C[q,c]\, ,
\end{equation*}
where $q$ is a quadric and $c$ a cubic
  (see, e.g., \cite{olver1999classical}). Since   $I=\Span{q}$, the lowest--degree constituent of  $\Inv(X,\Gsf_2)$ is $\p R_3=\{ [c]\}$.
Observe that   the grading  of $R$ induced by the central torus in $G_0$   coincides with the standard
grading.

\subsection{Proof of Theorem \ref{thm:invs} in types $\Bsf$ and $\Dsf$}\label{subProofTypeBD}
In this section we deal with both the cases when $\g$ is of type $\Bsf$ and $\Dsf$, since they both  correspond to the special orthogonal Lie group. We adopt a common approach,  by  stressing the differences  due to the parity of $n$, where $2n$ is the dimension of $\g_0$ (cf. \eqref{eqDefEnne}).\par
Accordingly, the      contact grading \eqref{eqContGrad}   becomes
 \begin{equation*}%\label{}
\so_{n+4}= \underset{ {  \g_{-2}  }}{\underbrace{ \C}} \oplus\underset{ {   \g_{-1} }} {\underbrace{  \C^2\otimes\C^n}} \oplus\underset{ {  \g_0 }}{\underbrace{  (\gl_2\oplus\so_n)}} \oplus\underset{ {   \g_{1} }}{\underbrace{  ( \C^2\otimes\C^n)^* }} \oplus\underset{ {   \g_{2} }}{\underbrace{  \C^*}}  \, .
\end{equation*}
We have $G_0^\ss \simeq \SL_2 \cdot \SO_n$, a quotient of
$\SL_2 \times \SO_n$ by a finite subgroup. The decomposition \eqref{eqBOXBOX} of the Pl\"ucker embedding space into $\g_0^\ss$--irreducible submodules reads
\begin{equation}\label{eqDecPluckBD}
\Lambda^n_0\fg_{-1} \simeq \left\{   \begin{array}{ll}\bigoplus_{a+b = \frac{n-1}{2}} S^{2a+1}\C^2 \otimes S^2_0 \Lambda^b \C^n & n\textrm{ odd,} \\\bigoplus_{a+b = \frac{n}{2}} S^{2a}\C^2 \otimes S^2_0 \Lambda^b \C^n & n \textrm{ even,}\end{array}   \right.
\end{equation}
where $S^2_0 \Lambda^b \C^n$ is: an
irreducible
representation of $\SO_n$ with highest weight
being twice the highest weight of $\Lambda^b \C^n$ for $b<n/2$; or
the direct sum of the irreducible representations whose highest
weights are twice the highest weights of the two summands of $\Lambda^{n/2}\C^n$
for $b=n/2$.
This follows from the decomposition of the $(\SL_2\times\SL_n)$-representation
$$ \Lambda^n (\C^2 \otimes \C^n) = \bigoplus_{|\lambda|=n} \Sigma_\lambda(\C^2) \otimes
\Sigma_{\lambda^*}(\C^n) $$
where the sum is over Young diagrams of size $n$, $\lambda^*$ denotes the transpose
of $\lambda$, and $\Sigma_\lambda$ is the Schur functor associated with $\lambda$ (see~\cite[Exercise 6.11 b]{MR1153249}). Indeed, one sees that the only diagrams entering
the sum are $\alpha = [n-b,b]$ in row-length notation, $0 \le b \le n/2$, with
$\Sigma_{[n-b,b]}(\C^2) \simeq S^{n-2b}\C^2$ and
$\Sigma_{[n-b,b]^*}(\C^n) \simeq (\End \Lambda^b \C^n)_0$.
The latter denotes the unique irreducible $\SL_n$-subrepresentation
in $\End \Lambda^b \C^n$ containing the image of $\SL_n$
under the representation map $\SL_n \to \GL (\Lambda^b \C^n)$. Then,
reducing to $\SL_2 \times \SO_n \subset \SL_2 \times\SL_n$ and
taking a quotient by $\omega\wedge\Lambda^{n-2}(\C^2\otimes\C^n)$,
one arrives at \eqref{eqDecPluckBD}.
As an immediate consequence of \eqref{eqDecPluckBD}
we have $\PP R_1=\emptyset$, that is, there are no $G_0^\ss$-invariant
hyperplanes.

\begin{remark}
Observe that  the Pl\"ucker embedding space $\Lambda^n_0\fg_{-1}$ is equipped with a  nondegenerate  pairing
\begin{eqnarray*}
\Lambda^n_0\fg_{-1}\times \Lambda^n_0\fg_{-1} &\longrightarrow & \Lambda^{2n} \fg_{-1} \simeq\C\, \\
( \phi,\psi )   &\longmapsto &  \phi \wedge \psi\, ,
\end{eqnarray*}
which is a quadratic form for $n$ even (i.e., $\g$ of type  $\Dsf$), and a symplectic form for $n$ odd (i.e., $\g$ of type  $\Bsf$).
%In both cases
%   $X^{(1)}_o$ is isotropic {\color{red} In what sense? I'd just say
%that
In the even case the corresponding null quadric contains $X^{(1)}_o$, and thus does not produce  a nontrivial invariant of degree $2$. However, we may
use the restriction of the corresponding bilinear form to a $G_0$-invariant
subspace of the Pl\"ucker space.
\end{remark}
\begin{pro}\label{pro:invD}
For $\g$ of type $\Dsf$, there is a nontrivial element   $[B]$ in  $\p R_2$.
\end{pro}
\begin{proof}
Observe that    the map
\begin{eqnarray*}
  \Lambda^n_0 \fg_{-1} &\stackrel{\pi}{\longrightarrow }& S^n \C^2\otimes\Lambda^n\C^n\equiv S^n \C^2\, \\
\xi_1\otimes v_1 \wedge \dots \wedge \xi_n\otimes v_n &\longmapsto &
(\xi_1\odot\cdots\odot\xi_n )\otimes  (v_1\wedge\cdots\wedge v_n)
\end{eqnarray*}
is surjective and  $G_0^\ss$--equivariant, and that   the $\SL_2$--invariant projection
\begin{equation*}
S^2 S^n\C^2 \stackrel{q}{\longrightarrow } S^n\Lambda^2\C^2   \simeq \C
\end{equation*}
defines an $\SL_2$--invariant
quadratic form $q$ on $ S^n\C^{2*}$.
Therefore, $B:=\pi^*(q)$  is a quadratic form on
$  \Lambda^n_0\fg_{-1}$. \par
It remains to show that $B$ does belong to the ideal $I$ (cf. \eqref{eqHEART}), that is, that $B$ does not vanish on $X_o^{(1)}$.
To this end, we to show that $B(\phi)\neq 0$ for some
$\phi \in \Lambda^n_0\fg_{-1}$ such that $[\phi] \in X^{(1)}_o$.
Fix an orthonormal basis $e_1,\dots,e_n$ of $\C^n$.
Let $\eta \in S^n\C^2$ be such that $q(\eta)\neq0$,
and fix a factorisation $\eta = \xi_1\cdots \xi_n$
with $\xi_i \in \C^2$, $1\le i\le n$. Now,
consider the linear subspace
$L = \langle \xi_1\otimes e_1,\dots,\xi_n\otimes e_n \rangle \subset \fg_{-1}$.
It is by construction Lagrangian, and furthermore its representing $n$-form
$\phi = (\xi_1\otimes e_1) \wedge\cdots\wedge (\xi_n\otimes e_n) $
satisfies
$B(\phi) = q(\eta) \neq 0$.
\end{proof}
This proves the part of Theorem \ref{thm:invs2} referring to type $\Dsf$.
In type $\Bsf$, Proposition \ref{lem:invBd2} below shows that   no such invariant quadric exists.
\begin{pro}\label{lem:invBd2}
For $\g$ of type $\Bsf$, we have $\p R_2=\emptyset$.
\end{pro}
\begin{proof}
From the decomposition formula \eqref{eqDecPluckBD} of the Pl\"ucker embedding space $ \Lambda^n_0 \fg_{-1}$ it follows that  there are no non--trivial $G_0^\ss$-invariants in $S^2 (\Lambda^n_0 \fg_{-1})^*$. Indeed, each summand is the tensor product of a symplectic module and an orthogonal module,
hence symplectic. Furthermore, no two summands are mutually dual.
\end{proof}
The next step in type $\Bsf$ is to look for cubic invariants. Along the same line as Proposition \ref{lem:invBd2}, we have:
\begin{pro}\label{lem:invBd3}
For $\g$ of type $\Bsf$, we have $\p R_3=\emptyset$.
\end{pro}
\begin{proof}
Identifying integral weights of $\SL_2$ with $\mathbb{Z}$,
and letting $S_a = S^a\C^2$, we have that
the set of weights of $S_{a_1}\otimes\cdots\otimes S_{a_r}$
is contained in $a_1+\dots+a_r + 2\mathbb{Z}$. In particular,
the tensor product of an odd number of even-dimensional representations
of $\SL_2$ cannot contain the trivial representation. Since
a $G_0^\ss$-invariant in a tensor power of $\Lambda_0^n\g_{-1}^*$
necessarily decomposes into summands which are products of
an $\SL_2$-invariant and an $\SO_n$-invariant, it follows that
there are no odd-degree $G_0^\ss$-invariants on $\Lambda_0^n\g_{-1}$
whatsoever.
\end{proof}
The final result of this section is common to both types $\Dsf$ and $\Bsf$. In type $\Bsf$ it provides the sought--for lowest--degree invariant, thus proving the corresponding part of Theorem
\ref{thm:invs2}, whereas in type $\Dsf$ it exhibits another interesting element of $\Inv(X,G)$.\par
We begin by observing that, by the classical invariant theory
of $\SL_2$ and $\SO_n$, there
is a unique one-dimensional subspace in the $G_0^\ss$-irreducible
decomposition of $S^4\g_{-1}^*$,
given dually by the projection
$$ S^4 (\C^2 \otimes \C^n) \to S^2 \Lambda^2 \C^2 \otimes \ker \left[
   S^2 \Lambda^2 \C^n \xrightarrow{\wedge} \Lambda^4\C^n\right]
\xrightarrow{\id \otimes \langle,\rangle} (\det\C^2)^2 \otimes \C \simeq \C. $$
The corresponding quartic
$q : S^4 (\C^2\otimes\C^n) \to \C$
is defined as
$$
q(\xi_1 \otimes e_1, \xi_2\otimes e_2, \xi_3\otimes e_3, \xi_4\otimes e_4)
= \epsilon(\xi_1,\xi_2)\epsilon(\xi_3, \xi_4) \left[
\langle e_1,e_3\rangle \langle e_2,e_4\rangle -
\langle e_2,e_3\rangle \langle e_1,e_4\rangle
\right]
$$
where $\epsilon \in \Lambda^2\C^{2*}$ is a volume form and $\langle\cdot,\cdot\rangle$
the $\SO_n$-invariant inner product on $\Lambda^2\C^n$. It is necessarily
proportional to the `canonical' $G_0^\ss$-invariant quartic
described in Subsection \ref{ss:strategy}.

\begin{pro}\label{pro:invB}
For $\fg$ of type $\Bsf$ or $\Dsf$, we have that $[q^n]\in\p R_4$ is well-defined.
\end{pro}
\begin{proof}
Consider as before a Lagrangian subspace of the form
$L = \langle \xi_1\otimes e_1,\dots,\xi_n\otimes e_n \rangle \subset \fg_{-1}$,
where $e_1,\dots,e_n$ is an orthonormal basis in $\C^n$, while $\xi_1,\dots,\xi_n \in \C^2$
is a general $n$--tuple. We compute
$ q^n(\phi) $ for $\phi = \bigwedge_{i=1}^n (\xi_i \otimes e_i)$:
\begin{eqnarray*}
q^n(\phi) =&
\sum_{\sigma^{(1)},\sigma^{(2)},\sigma^{(3)}}&
\sgn(\sigma^{(1)}\sigma^{(2)}\sigma^{(3)})
\\ & \times\prod_{i=1}^n &
\epsilon(\xi_i, \xi_{\sigma^{(1)}i}) \epsilon(\xi_{\sigma^{(2)}i}, \xi_{\sigma^{(3)}i})
\\ & &
\left[
\langle e_i, e_{\sigma^{(2)}i} \rangle
\langle e_{\sigma^{(1)}i} , e_{\sigma^{(3)}i} \rangle
- \langle e_i, e_{\sigma^{(3)}i} \rangle
\langle e_{\sigma^{(1)}i} , e_{\sigma^{(2)}i} \rangle
\right]
\\ = &
\sum_{\sigma^{(1)},\sigma^{(2)},\sigma^{(3)}}&
\sgn(\sigma^{(1)}\sigma^{(2)}\sigma^{(3)})
\prod_{i=1}^n
\left[\delta^i_{\sigma^{(2)}i}
\delta^{\sigma^{(1)}i}_{\sigma^{(3)}i}
- \delta^i_{\sigma^{(3)}i}
\delta^{\sigma^{(1)}i}_{\sigma^{(2)}i}
\right]
\\ & \times \prod_{i=1}^n &
\epsilon(\xi_i, \xi_{\sigma^{(1)}i}) \epsilon(\xi_{\sigma^{(2)}i}, \xi_{\sigma^{(3)}i}).
%\\ = & Q(\xi_1\cdots\xi_n) &
\end{eqnarray*}
Note that this way we have
$$ q^n(\phi) = Q(\xi_1\cdots\xi_n) $$
for certain $\SL_2$--invariant
quartic $Q$ on $S^n\C^2$. We only need to check that $Q\neq0$.
We further rewrite:
$$
Q(\xi_1\cdots\xi_n) = (-1)^n\sum_\sigma \left(\prod_i\epsilon(\xi_i, \xi_{\sigma i})^2\right)
\times \left(\sgn(\sigma) \sum_{J \subset \{1,\dots,n\}}
C_{\sigma, J} \right)\, ,
$$
where
$$
C_{\sigma, J} = \sgn(\sigma^J) \sgn(\sigma^{J^c}),\quad \sigma^J(i) = \begin{cases} i & i \in J\, , \\ \sigma(i) & i \notin J \, ,\end{cases}
$$
and $\sgn(\sigma^J)=0$ if $\sigma^J$ is not a permutation.
The sum over $J$ may be restricted to $\sigma$-invariant sets,
in which case $C_{\sigma,J} = \sgn(\sigma)$ and we obtain
$$
Q(\xi_1\cdots\xi_n) = (-1)^n\sum_\sigma \left(\prod_i\epsilon(\xi_i, \xi_{\sigma i})^2\right)
\cdot {\#\{ J \subset \{1,\dots,n\}\ |\ \sigma J = J \}}\, .
$$
Choosing $\xi_1,\dots,\xi_n$ in $\mathbb{R}^n$ such that
$\epsilon(\xi_i,\xi_j)=0$ if and only if $i=j$,
we find that $(-1)^n Q(\xi_1,\dots,\xi_n)$ is
a sum of non-negative reals, with positive
terms corresponding to fix-point free $\sigma$. Hence $Q(\xi_1,\dots,\xi_n)\neq0$ and
thus $q^n(\phi)\neq0$.
\end{proof}

\subsection{Representation-theoretic setup}\label{SubRepTheoSetUp}
Having dealt with types $\Asf$, $\Bsf$, $\Dsf$ and $\Gsf$ in a rather
direct manner, we shall need to resort to more abstract methods in order
to handle the remaining types $\Esf$ and $\Fsf$. This subsection
introduces some representation-theoretic tools that are valid
in greater generality,\footnote{Only the type $\Asf$ is excluded in what follows}  by picking up where we left off   Subsection \ref{subContGrad}.
We will use the language of modules over the universal enveloping $U(\g_0)$
rather than representations of $G_0$. Since the latter is connected, this does
not change the notion of invariance.

\par
Fix a Cartan subalgebra $\fh \subset \fg$ and sets
$\Delta \subset \Phi^+ \subset \Phi$ of simple and positive roots within
the root system of $\fg$ with respect to $U(\fh)$, compatible with the
 grading. In particular,
$\fg_2$ is the root space of the longest positive root.
Given a $U(\fg)$-module $M$, let $C^\bullet(\fg_-,M)$ denote
the cochain complex computing the Lie algebra cohomology of $\fg_-$
with values in $M$. If $M$ is graded compatibly with $\fg$, let
$C^\bullet_i(\fg_-,M)$ denote the homogeneous degree $i$ subcomplex.
Use $Z^\bullet_i, B^\bullet_i, H^\bullet_i$ to denote the spaces
of cocycles, coboundaries and the cohomology, respectively.\par
Let us identify $\sp(\fg_{-1})$ with $\sp_n$ together with a choice of
a Cartan and Borel subalgebra. Let $\lambda_1,\dots,\lambda_n$ be the fundamental weights
of $\sp_n$ and $V_\lambda$ the simple $U(\sp_n)$-module of highest weight $\lambda$.
Let $W$ be the Weyl group of $\fg$, with $W^\fp$ the subset consisting
of words $w \in W$ such that $w\rho$ is $\fg_0$-dominant with $\rho$ being the sum of all
fundamental weights of $\fg$.
Let $W^\fp_i$ denote the subset of $W^\fp$ consisting of
words $w$ of length $i$.
\begin{lemma}\label{lem:kost}
For each $1 \le i \le n$ we have $U(\fg_0)$-module isomorphisms
$$V_{\lambda_i} \simeq \Lambda^i_0 \fg_{-1} \simeq
H^i(\fg_-,\C)^* \simeq
\bigoplus_{w \in W^\fp_i} L(w\cdot 0)
$$
where $L(\lambda)$ denotes the simple $U(\fg_0)$-module
with highest weight $\lambda$.
\end{lemma}

\begin{proof}
Note first that each $C^i(\g_-,\C)$
is the direct sum of subspaces of degrees $i$ and $i+1$.
Considering the part of the complex computing the degree $i$ subspace
$H^i_i(\g_-,\C)$, we have
the following identifications:
$$\begin{CD}
 C^{i-1}_i(\g_-,\C) @>{\partial}>> C^i_i(\g_-,\C) @>{\partial}>> C^{i+1}_i(\g_-,\C) \\
 @| @| @| \\
 \g_{-2}^* \otimes \Lambda^{i-2}\g_{-1}^*  @>{\omega\wedge}>> \Lambda^i \g_{-1}^*
@>>> 0\, ,
\end{CD}$$
where $\omega $ is the twisted symplectic form \eqref{eqSTAR}. We thus
have a $U(\g_0)$-module isomorphism
$$
 H^i_i(\g_-,\C) \simeq \Lambda^i \g_{-1}^* / (\omega \wedge (\g_{-2}^* \otimes \Lambda^{i-2} \g_{-1}^* ))
$$
and its dual
$$
 H^i_i(\g_-,\C)^* \simeq \Lambda^i_0 \g_{-1}.
$$
On the other hand, along similar lines we get $H^i_{i+1}(\g_-,\C)=0$
by injectivity of the wedge map
$\omega\wedge: \Lambda^i \g_{-1}^* \to \Lambda^{i+2} \g_{-1}^* \otimes
\g_{-2}$ for $i < n$ (this is a standard fact on symplectic vector spaces).
Hence $H^i(\g_-,\C) = H^i_i(\g_-,\C)$ and thus
$$ \Lambda^i_0\g_{-1} \simeq H^i(\g_-,\C)^*. $$
The remaining two isomorphisms are given by Kostant's theorem (the decomposition
of cohomology into simple modules~\cite{MR0142696}) and standard representation theory of $\sp_n$
(identification of fundamental modules with Lagrangian exterior powers).
\end{proof}

\subsection{Computing the space of quadric invariants
for even $n$}
We will now introduce a representation-theoretic method to
compute the dimension of the space of invariant quadrics
valid whenever $n$ is even (i.e., when the Pl\"ucker space
has an $\Sp_n$-invariant symmetric bilinear form). That is,
it can be applied in type $\Esf$ (which is our main point
of interest here) as well as $\Dsf$.

We denote by $I$ the ideal of
$\LG(\g_{-1})$ in $S^\bullet \Lambda^n_0 \g_{-1}^*$.
By Corollary \ref{cor:sections-polys}, its degree $2$ part $I_2$
may be identified with the complement of $V_{2\lambda_n}$ in $S^2 V_{\lambda_n}$.
Accordingly, $R_2$ becomes identified with the space of $\g_0^\ss$-invariants
in the simple $U(\spp_n)$-module $V_{2\lambda_n}$.
We should thus decompose the latter into simple $U(\g_0)$-submodules
and look for rank--one summands.
In order to use the
Weyl-group description of Lemma \ref{lem:kost}, we
need to express modules of the form $V_{2\lambda_i}$
in terms of tensor products of the fundamental modules.

\begin{lemma}\label{lem:bases}Assume $n$ even.
Let $\lambda_i$ denote the $i\Th$ fundamental
weight of $\sp_n$, and set for convenience $\lambda_0=0$.
We then have:
$$
S^2 V_{\lambda_i} \simeq \sum_{0 \le j \le i/2} V_{2\lambda_{i - 2j}}
\oplus \sum_{\substack{\frac{i-n}{2} \le j < k \le \frac{i}{2} \\ j+k \ge 0 \\ k-j \le n-i}} V_{\lambda_{i - 2j} + \lambda_{i - 2k}}
$$
for all $i$ and
$$
V_{\lambda_i}\otimes V_{\lambda_j} \simeq \sum_{\substack{k,l \ge 0 \\ i-k-l\ge 0 \\ j+k \le n}} V_{\lambda_{i-k-l} + \lambda_{j+k-l}}
$$
for $i<j$ with $j-i$ even, as $U(\sp_n)$-modules.
\end{lemma}

\begin{proof}
We invoke the rules for computing tensor products
of representations of the symplectic group in terms of Young diagrams
(these can be derived from~\cite[Sec. 2.5]{MR885807}).
We will write
a Young diagram as a nonincreasing sequence where the entries give
the height of the subsequent \emph{columns}. In particular $\C = [\ ]$, $V_{\lambda_i} = [i]$
and $V_{\lambda_i + \lambda_j} = [j,i]$ if $i \le j$.
For convenience, we allow ourselves to write $[i,0]$ for $[i]$
and $[0]$ for $[\ ]$ (this mirrors the convention $\lambda_0=0$).
Furthermore, a column of length $n+i$ is replaced by
one of length $n-i$.
If the Young diagrams were used to denote representations of $\mathfrak{sl}_{2n}$ rather than $\sp_n$,
the rule for decomposing a tensor product $[j] \otimes [i]$ with $j \ge i$ would be simply:
$$
[j] \otimes [i] = [j+i,0] + [j+i-1,1] + \dots + [j,i],\quad \textrm{(for $\mathfrak{sl}_{2n}$)}
$$
i.e.,  we put the columns $[j]$ (first, `red') and $[i]$ (second, `black') next to each other,
and move a number black boxes underneath the red ones.
Since in the case of $\sp_{n}$ we additionally have the invariant symplectic form on $[1]$,
the rule should be modified so that when moving a red box, we can
either `add' it, appending to the first column, or
`subtract' it, annihilating a red box. We may assume
we first add a number of black boxes, and then subtract a number of them.
Furthermore, self-duality of $[1]$ implies that we ought to remove a red-black
pair from the first column as soon as it becomes taller than $n$:
in other words, we may add a black box only as long as
the height of the first column is at most $n$.
Thus we obtain
\begin{equation}\label{eqGroeth}
[j]\otimes [i] = \sum_{\substack{k,l \ge 0 \\ i-k-l\ge 0 \\ j+k \le n }} [j+k-l, i-k-l]
= \sum_{\substack{j-n\le p \le q \le i,\\ p+q\ge 0 \\ q - p \le 2(n - j)}} [j-p,i-q]
% l+k = q       2l = q+p
% l-k = p       2k = q-p
\end{equation}
where the first sum clearly coincides with the original expression
for $V_{\lambda_i}\otimes V_{\lambda_j}$.
The same formula specialises to the symmetric square of $[i]$,
where the terms of the above sum contained in $S^2 [i] \subset [i]\otimes [i]$
are those of the form $[i-p, i-q]$ with $p$, $q$ even. These are easily seen
to give the original expression for $S^2 V_{\lambda_i}$ (with $j=p$ and $k=q$).
\end{proof}
\begin{remark}\label{remGroeth}
 It is convenient to view   equations \eqref{eqGroeth} in the
Grothendieck group $K$ of the category of finite-dimensional $U(\sp_n)$-modules
(this is simply the free abelian group generated by classes of finite-dimensional
irreducible representations of $\sp_n$).
The relations may be then inverted so that, in particular, the class
$[V_{2\lambda_n}]$ may be expressed as a linear combination of
$[S^2 V_{\lambda_i}]$ and $[V_{\lambda_i} \otimes V_{\lambda_j}]$ with $0 \le i<j\le n$.
However, since we are only interested in the dimension of the space of $\fg_0^\ss$-invariants,
it is easier to first apply the corresponding homomorphism $K \to \ZZ$ to both sides of the
above equations (viewed in $K$), and then solve for $\dim (V_{2\lambda_n})^{\fg_0^\ss}$ in terms of
modules whose invariants we know.
\end{remark}

\begin{lemma}\label{lem:invs-weyl}Let $\bar w_\circ$ be the longest element in the Weyl group of $\fg_0^\ss$, and $\bar h^\circ \in \fh \cap \fg_0^\ss$ the sum of the coroots associated to positive
roots of $\fg_0^\ss$.
Let $\iota : \fh \cap \fg_0^\ss \to \fh$ denote the inclusion of Cartan subalgebras,
and $\iota^*$ its transpose acting as restriction on weights.
Then
\begin{eqnarray*}
2\dim (S^2 V_{\lambda_i})^{\fg_0^\ss} &=&
\#\{
 (w, w') \in W^\fp_i\times W^\fp_i\ |\  -\bar w_\circ\iota^*(w\cdot 0) = \iota^*(w'\cdot 0)
\}
\\&+&
\#\{
 w \in W^\fp_i\ |\ -\bar w_\circ\iota^*(w\cdot 0) = \iota^*(w\cdot 0),\quad
\langle \iota^*(w\cdot 0), \bar h^\circ \rangle = 0 \mod 2)
\}
\\&-&
\#\{
 w \in W^\fp_i\ |\ -\bar w_\circ\iota^*(w\cdot 0) = \iota^*(w\cdot 0),\quad
\langle \iota^*(w\cdot 0), \bar h^\circ \rangle = 1 \mod 2)
\},
\\
\dim (V_{\lambda_i}\otimes V_{\lambda_j})^{\fg_0^\ss} &=&
\#\{
 (w, w')\in W^\fp_i\times W^\fp_j\ |\ -\bar w_\circ\iota^*(w\cdot 0) = \iota^*(w'\cdot 0)
\}\, ,
\end{eqnarray*}
for all $1 \le i,j \le n$.
\end{lemma}

\begin{proof}
Recall that $-\bar w_\circ$ sends the highest weight of a finite-dimensional
simple $U(\g_0^\ss)$-module to the highest weight of its dual.
In particular a finite-dimensional simple $U(\g_0^\ss)$-module of highest weight $\bar\lambda$
is self-dual if and only if $-\bar w_\circ \bar\lambda = \bar\lambda$. The self-duality is
implemented either by a symmetric bilinear invariant
or an alternating one. By \cite[Thm. 3.2.17]{MR2522486}, the parity of the invariant
coincides with the parity of $\langle \bar\lambda, \bar h^\circ\rangle$.
The formulae then follow
from Lemma \ref{lem:kost}.
\end{proof}
We interpret the equations of Lemma \ref{lem:bases}
as relations
\begin{eqnarray}\label{eq:rels-in-k}
[S^2 V_{\lambda_i}] - \sum_{0 \le j \le i/2} [V_{2\lambda_{i - 2j}}]
- \sum_{\substack{\frac{i-n}{2} \le j < k \le \frac{i}{2} \\ j+k \ge 0 \\ k-j \le n-i}} [V_{\lambda_{i - 2j} + \lambda_{i - 2k}}] &=& 0\quad (1 \le i \le n) \\ {}
[V_{\lambda_i}\otimes V_{\lambda_j}]- \sum_{\substack{k,l \ge 0 \\ i-k-l\ge 0 \\ j+k \le n}} [V_{\lambda_{i-k-l} + \lambda_{j+k-l}}] &=& 0\quad (1 \le i < j \le n,\quad j-i\in 2\ZZ)\nonumber
\end{eqnarray}
 in $K$, with $\lambda_0=0$ by definition (as explained in the above Remark \ref{remGroeth}).
Applying the homomorphism $K \to \ZZ$,  $[M] \mapsto \dim M^{\fg_0^\ss}$,
and substituting
the expressions given in Lemma \ref{lem:invs-weyl},
we obtain a determined linear system
for the unknowns
$$ d_i = \dim(V_{2\lambda_i})^{\fg_0^\ss},\ 1 \le i\le n, \quad d_{ij} = \dim (V_{\lambda_i+\lambda_j})^{\fg_0^\ss},\ 1\le i < j\le n,\ j-i\in 2\ZZ. $$
In particular, we may solve for $d_n$.

The only non-trivial task is the Weyl-group computation in Lemma \ref{lem:invs-weyl}. Let us recall that $W^\fp$ is the set of minimal length coset representatives
for the quotient $W / W_\fp$, where $W_\fp \subset W$ denotes the parabolic Weyl subgroup generated
by simple roots in $\Phi_0$. Equivalently, $W_\fp$ is the stabiliser of the highest root
$\gamma \in \Phi^+$, so that $W^\fp$ may be naturally identified with the orbit $W \gamma$
(note that $\gamma$, being the highest weight of the adjoint representation, is a
fundamental weight since    type $\Asf$ has been excluded). This gives rise to an algorithm for
generating $W^\fp$ described, e.g., in~\cite[Prop. 3.2.16 and the following paragraph]{MR2532439}. Putting these together,
we have the following algorithm to compute $\dim(V_{2\lambda_n})^{\g_0^\ss} = \dim R_2$
for $\g$ of type $\Dsf$, $\Esf$ or $\Gsf$ and rank $\ell$.
\begin{enumerate}
\item Obtain from a database:
\begin{center}\begin{tabular}{lcl}
the Cartan matrix of $\Phi$ & as & a list of $\ell$ elements of $\ZZ^\ell$, \\
the highest weight of $\g$ & as & an integer $1 \le a \le \ell$, \\
the involution $-\bar w_\circ$ & as & a permutation of $\{1,\dots,\ell\}$ fixing $a$, \\
the coroot $\bar h^\circ$ & as & an element of $\ZZ^\ell$ with trivial $a$-th entry, \\
the integer $n$,
\end{tabular}\end{center}
where we use $\ZZ^\ell$ to represent weights (in the basis of fundamental weights)
and coroots (in the basis of simple coroots).
\item Generate $W^\fp_i$, $1 \le i \le n$, as lists of words over $\{1,\dots,\ell\}$.
\item Compute $\dim(S^2 V_{\lambda_i})^{\g_0^\ss}$, $1 \le i\le n$
and $\dim (V_{\lambda_i}\otimes V_{\lambda_j})^{\g_0^\ss}$, $1 \le i<j\le n$ as in
Lemma \ref{lem:invs-weyl}.
\item Set up the formal linear system \eqref{eq:rels-in-k} and substitute
$$
        [V_{\lambda_0}] \mapsto 1,\ [V_{\lambda_i}] \mapsto 1,\
        [S^2 V_{\lambda_i}] \mapsto \dim (S^2 V_{\lambda_i})^{\g_0^\ss}\ (1\le i\le n),\quad
        [V_{\lambda_i}\otimes V_{\lambda_j}] \mapsto \dim (V_{\lambda_i}\otimes V_{\lambda_j})^{\g_0^\ss}\ (1\le i<j\le n)\, ,
$$
and
$$
        [V_{2\lambda_i}] \mapsto d_i\ (1 \le i \le n),\quad [V_{\lambda_i+\lambda_j}] \mapsto d_{ij}\
        (1\le i<j\le n).
$$
\item Solve the resulting linear system on $d_i, d_{ij}$ over $\ZZ$.
\item Return $d_n$.
\end{enumerate}
The algorithm is straightforward to implement (see~\cite{MR1743970} for a
comprehensive discussion of computational methods in Lie theory).
Note that since $\g_0^\ss$ is simply-laced,
the coefficients of $\bar h^\circ$ in the basis of simple coroots coincide with
those of the sum of all positive roots of $\g_0^\ss$ in the basis of simple roots.

\subsection{Types $\Esf_6$ $\Esf_7$, $\Esf_8$}
We list the database entries required for the computation,
and the final answer.
The code used for this computation is available as~\cite{jan-quads}.
The expressions for $-\bar w_\circ$ and $\bar h^\circ$
can be found in Bourbaki \cite[\S 4, tables, entries VII and XI]{MR1890629},
up to the necessary relabeling
the of Dynkin sub-diagram corresponding to
$\fg_0^\ss \subset \g$ (the Bourbaki labeling of the diagram
for $\g$ induces a labeling on the diagram of $\g_0^\ss$
that has to be mapped to its own Bourbaki labeling).
We conclude that $\dim R_2 = 1$ in all three cases.
Since the grading induced by the central torus of $G_0$
is a rescaling of the standard one (see Remark \ref{remTorus}), we have that
there exists a unique degree $2$ element in $\Inv(X,G)$.

\begin{center}\fbox{\begin{tabular}{c|c|c}
$\Esf_6$ & $\Esf_7$ & $\Esf_8$ \\
\hline
$\left(\begin{matrix}
                                 2& 0&-1& 0& 0& 0\\
                                 0& 2& 0&-1& 0& 0\\
                                -1& 0& 2&-1& 0& 0\\
                                 0&-1&-1& 2&-1& 0\\
                                 0& 0& 0&-1& 2&-1\\
                                 0& 0& 0& 0&-1& 2
\end{matrix}\right)$

&

$\left(\begin{matrix}
                                  2& 0&-1& 0& 0& 0& 0\\
                                  0& 2& 0&-1& 0& 0& 0\\
                                 -1& 0& 2&-1& 0& 0& 0\\
                                  0&-1&-1& 2&-1& 0& 0\\
                                  0& 0& 0&-1& 2&-1& 0\\
                                  0& 0& 0& 0&-1& 2&-1\\
                                  0& 0& 0& 0& 0&-1& 2
\end{matrix}\right)$

&

$\left(\begin{matrix}
                                  2& 0&-1& 0& 0& 0& 0& 0\\
                                  0& 2& 0&-1& 0& 0& 0& 0\\
                                 -1& 0& 2&-1& 0& 0& 0& 0\\
                                  0&-1&-1& 2&-1& 0& 0& 0\\
                                  0& 0& 0&-1& 2&-1& 0& 0\\
                                  0& 0& 0& 0&-1& 2&-1& 0\\
                                  0& 0& 0& 0& 0&-1& 2&-1\\
                                  0& 0& 0& 0& 0& 0&-1& 2
\end{matrix}\right)$
\\
$a = 2$ & $a = 1$ & $a=8$ \\
$-\bar w_\circ =\left( \begin{matrix}
1 & 2 & 3 & 4 & 5 & 6 \\
6 & 2 & 5 & 4 & 3 & 1
\end{matrix} \right)$ &
$-\bar w_\circ = \id$ & $-\bar w_\circ = \id$ \\
$\bar h^\circ = (5,0,8,9,8,5)$ & $\bar h^\circ = (0,15,15,28,24,18,10)$ &
        $\bar h^\circ = (34,49,66,96,75,52,27,0)$ \\
$n=10$ & $n=16$ & $n=28$ \\
\hline
$d_n = 1$ & $d_n = 1$ & $d_n=1$
\end{tabular}}
\end{center}

\subsection{Type $\Fsf_4$}
In order to deal with the remaining type $\Fsf_4$,
we invoke the brute-force branching method relying
on the computation of formal characters. We shall take
for granted that a procedure for computing the formal character
of a given finite-dimensional highest-weight module of a given
semi-simple Lie algebra is at our disposal (these are typically
refinements of the Freudenthal multiplicity formula, see~\cite[Sec. 8.9]{MR1743970}).
The algorithm to compute $\dim (V_{d \lambda_n})^{\g_0^\ss}$ is then as follows.
\begin{enumerate}
\item Identify the weight lattice of $\sp_n$, resp. $\g_0^{ss}$, with $\ZZ^n$,
resp. $\ZZ^3$, using the bases of fundamental weights.
\item Choose an element $w_i \in W^\fp_i$ for each $1 \le i \le n$.
\item Construct the $\ZZ$-module homomorphism
$\rho : \ZZ^n \to \ZZ^3$ representing the map sending $\lambda_i$
to $\iota^*(w_i\cdot 0)$.
\item Compute the formal character $\ch V_{d \lambda_n}$ as an element
of the group ring of $\ZZ^n$.
\item Compute $\rho_*\ch V_{d\lambda_n}$, an element in the group ring of $\ZZ^3$.
\item Decompose
\begin{equation}\label{eq:decomp-char}
 \rho_* \ch V_{d\lambda_n} = \sum c_\mu \ch L_{\mu}
\end{equation}
into a combination formal characters of finite-dimensional simple $U(\g_0^\ss)$-modules.
\item Return $c_0$.
\end{enumerate}
In our case we have $\g_0^\ss$ of type $\Csf_3$
and, using Bourbaki's labelling of the fundamental weights,
the matrix of $\rho$ reads
$$
\rho = \left(\begin{matrix}
 0 & 0 & 1 & 3 & 5 & 4 & 4 \\
 0 & 2 & 2 & 0 & 0 & 1 & 0 \\
 1 & 0 & 0 & 1 & 0 & 0 & 1
\end{matrix}\right).
$$
The formal characters are too complex to be included here. A simple program computing the
decomposition \eqref{eq:decomp-char} in \texttt{LiE} is available as~\cite{jan-sympluck}.
We conclude that $\dim (V_{d \lambda_n})^{\g_0^\ss}$ is zero for $1 \le d \le 3$ and one for $d=4$, thus completing the proof of Theorem \ref{thm:invs2} in type $\Fsf$.

\section{Further discussion}
\label{sec:further-discussion}
\subsection{Maximality}

One interesting feature of $G$-invariant PDEs over $X$ is that
even their local infinitesimal symmetry algebras are precisely
isomorphic to $\fg$. This had already been observed in~\cite{2016arXiv160308251T}. We
shall sketch the argument.

\begin{lemma}\label{lem:maximality}
Assume $\g$ is not of type $\Asf$. Then
the map $\fg_0 \to \csp(\fg_{-1})$ induced by the adjoint action is
an embedding onto a maximal subalgebra.
\end{lemma}

Inspecting the list of embeddings $\g_0^\ss \subset \sp_n$,
this may be extracted from Dynkin's classification
of maximal subgroups in the simple Lie groups. Since
we have already used Kostant's theorem in  Subsection \ref{SubRepTheoSetUp} above,
we will  use it here to provide a self-contained proof.
\begin{proof}
Following the notation introduced in the
proof of the main Theorem \ref{thm:invs},
we work in $C^\bullet_0(\fg_-,\fg)$ and suppress the $(\fg_-,\fg)$ part from the notation.
We have
$
C^0_0 = \fg_0$, $ C^1_0 = \End_0 \fg_-$, $ Z^1_0 \simeq \csp(\fg_{-1})$.
Since $H^0(\fg_-,\fg)$ is the annihilator of $\fg_-$ in $\fg$, namely $\fg_{-2}$, we have
$H^0_0=0$ whence $\fg_0 \to \csp(\fg_{-1})$ is injective.
Now, by Kostant's theorem,
$ H^1 \simeq L(s_\alpha \cdot \gamma) $,
where $\gamma \in \Phi^+$ is the highest root, $\alpha$ is the unique simple root
non-orthogonal to $\gamma$, and $L(\lambda)$ denotes the simple $U(\fg_0)$-module with highest
weight $\lambda$. In particular, we have just demonstrated that
$ H^1_0 \simeq \csp(\fg_{-1}) / \fg_0 $
is a simple $U(\fg_0)$-module (nil in type $\Csf$), thus proving the claim.
\end{proof}
Since we wish to speak of local symmetries of a PDE,
let us first point out that the construction of the
bundle of Lagrangian Grassmannians (as in \eqref{eq:m1}) is functorial with respect
to local contactomorphisms. More precisely, if $\phi : U \to V$ is
a map of contact complex manifolds inducing isomorphisms
on fibres of contact distributions, then $\phi$ lifts naturally to a
map $\phi^{(1)} : U^{(1)} \to V^{(1)}$ inducing isomorphisms on fibres.
It will be convenient to work with germs of infinitesimal
symmetries at a point. Let us quickly define the necessary terms.
\begin{defn}
Let $\E \subset X^{(1)}$ be a hypersurface, $x \in X$ a point
and $[v]$ a germ at $x$ of a vector field on $X$.
We say that $[v]$ is a germ of an  \emph{infinitesimal symmetry}
of $\E$ if there exist:
open neighborhoods $V \subset U$ of $x$, an open disc $\Delta \subset \C$
and a representative $v$ of $[v]$ on $U$,  such that
\begin{enumerate}
\item $v$ is contact, i.e. $[v,\CC] \subset \CC$ over $U$,
\item the local flow $\phi_t : V \to U$, $t\in\Delta$,  of $v$ is well-defined,
\item the lifts $\phi^{(1)}_t : V^{(1)} \to U^{(1)}$, $t\in\Delta$,
preserve $\E$.
\end{enumerate}
\end{defn}
\begin{proposition}
Assume  $\g$ is not of type $\Asf$.\footnote{Nor of type $\Csf$,   see Remark \ref{remStupido}.}
Let $\E \in \Inv(X,G)$ be a $G$-invariant hypersurface
in $X^{(1)}$. Fix a point $x \in X$ and let $\fs$ be
the Lie algebra of germs at $x$ of infinitesimal symmetries
of $\E$. Then the local action map $\g \to \fs$ is an isomorphism.
\end{proposition}
The proof is a standard application of Tanaka theory~\cite{MR0266258}
(see also~\cite{MR2765511}).
\begin{proof}
By $G$-invariance we may take $x = o$ and use the identification $T_oX \simeq \g/\g^0$.
Let $\ev : \fs \to \g/\g^0$ be the evaluation map at $o$.
There is a natural filtration on $\fs$ defined by setting $\fs^i = \ev^{-1} \g^i/\g^0$
for $i \le 0$, and extending inductively for $i>0$ so that $v \in \fs^{i+1}$ if and only if
$[v,\fs^{-1}] \subset \fs^i$. It is straightforward to check that $[\fs^i,\fs^j] \subset \fs^{i+j}$
and thus we have the associated graded Lie algebra $\gr \fs$. Furthermore,
the natural embedding $\g \to \fs$ is a homomorphism of filtered Lie algebras.
Since $\g \subset \fs$ acts infinitesimally transitively, we have that
$\ev:\fs/\fs^0 \to \g/\g^0$ is an isomorphism of vector spaces. We now observe that:
\begin{enumerate}
\item the induced map $\gr_-\fs \to \g_-$ is an isomorphism of graded nilpotent Lie algebras,
\item the maps $\gr_i \fs \to \Hom(\gr_{-1}\fs, \gr_{i-1}\fs)$ induced by the adjoint action
are injective.
\end{enumerate}
Claim (1) is most easily seen by restricting to
$\gr_- \g \subset \gr_-\fs$. Claim (2) follows from the very definition of the filtration.
We will show that $\g_i \to \gr_i \fs$ is an isomorphism. This is clearly true for $i<0$,
and also for $i=0$ by Lemma \ref{lem:maximality}.
By induction we may assume it had been shown for all $i < k$, $k > 0$. We then have
$\gr_k \fs \subset \Hom(\g_{-1}, \g_{k-1})$ by (2) above and by the inductive hypothesis.
In fact, using the Lie algebra structure we may extend this to an embedding of $\gr_k\fs$ into the space $\Der_k(\g_-,\g)$
of degree $k$ derivations of $\g_-$ into the $U(\g_-)$-module $\g$. This space clearly
contains $\g_k$ and we have
$$ \Der_k(\g_-,\g) / \g_k \simeq H^1_k(\g_-,\g). $$
It will thus be enough to check that the cohomology space on the right hand side vanishes
for $k > 0$. This had been done by Yamaguchi in~\cite[Prop. 5.1 (2)]{MR1274961}.
\end{proof}

\subsection{The Lagrangian Chow transform and invariant hypersurfaces of geometric origin}
\label{secRumianek}
As we have already remarked, one may produce
certain invariant PDEs over adjoint varieties by means of a
straightforward geometric construction. This
observation is a key idea in~\cite{2016arXiv160308251T}.
We   review it here for purpose of comparison.
 As before, $\fg$ is simple
not of type $\Csf$, with adjoint variety $X \subset \PP\fg$. The notation
is as introduced in Sections \ref{SecDescMainRes} and \ref{sec:prerequisites}. In particular, the contact hyperplane
at the origin $o \in X$ is identified with $\g_{-1}$, and the action
of the isotropy subgroup $P \subset G$ of $o$ restricted to $\g_{-1}$
factors through the reductive group $G_0$ (see Subsection \ref{subAdjContMan}).
\begin{definition}
The sub-adjoint variety $Y \subset \PP \g_{-1}$ of $\g$ is
the union of the closed orbits of $G_0$ in the projectivised
irreducible summands of $\g_{-1}$.
\end{definition}

Of course, as indicated in the table in Lemma \ref{lem:properties-grading},
the only case with decomposable $\g_{-1}$ is type $\Asf$: there
$G_0 \simeq \GL_n \times \C^\times$,  $\g_{-1} \simeq \C^n \oplus \C^{n*}$
and $Y \simeq \PP^{n-1} \cup \PP^{(n-1)*}$. In the remaining cases $\g_{-1}$
is irreducible, and so is $Y$. As we will soon explain, it is interesting
to compute the \emph{degree} of $Y$ as a subvariety of $\PP\g_{-1}$.
\begin{lemma}
The following table, supplementing that of Lemma \ref{lem:properties-grading},
gives the sub-adjoint variety $Y \subset \PP\g_{-1}$ and its degree.
The sub-adjoint variety is always Legendrian and, in particular, of codimension $n$.
\begin{center}
\fbox{\begin{tabular}{l|l|l|l|l|l|l}
restriction & type of $\fg$   &  $\g_0^\ss$ & $\fg_{-1}$ & $Y$ & embedding & $\deg Y$\\ \hline
$n\ge 1$  & $\Asf_{n+1}$& $\sll_n $ & $\C^n \oplus \C^{n*}$ & $\PP^{n-1} \sqcup \PP^{(n-1)\ast}$  &
linear & $2$ \\
$n\ge 3$ & $\Bsf_{(n+3)/2}$
or $\Dsf_{(n+4)/2}$
&  $\sll_2 \oplus \so_n$ & $\C^2\otimes\C^n$ & $\PP^1 \times Q^{n-2}$ & Segre & $2(n-1)$ \\
&$\Esf_6$          & $\sll_6$  & $\Lambda^3\C^6$ & $\Gr(3,6)$ & Pl\"ucker & $42$ \\
&$\Esf_7$          & $\mathfrak{spin}_{12}$ & spinor 32-dim & 15-dim spinor & & $286$ \\
&$\Esf_8$          & $\Esf_7$  & fundamental 56-dim & 27-dim Freudenthal & & $13188$ \\
&$\Fsf_4$          & $\sp_3$  & $\Lambda^3_0\C^6$ & $\LG(3,6)$ & Pl\"ucker &  $16$ \\
&$\Gsf_2$          & $\sll_2$  & $S^3\C^2$ & $\PP^1$ & Veronese & $3$
\end{tabular}}
\end{center}
Here $Q^{n-2} \subset \PP^{n-1}$ denotes a non-singular quadric hypersurface.
\end{lemma}
\begin{proof}
The description of $Y$, its embedding in $\PP\g_{-1}$ and
the Legendrian property
may be found in~\cite{MR2584515}.
The degrees in types $\Asf$, $\Bsf$, $\Dsf$
and $\Gsf_2$ are easy to compute. For $\Asf$ we have
the union of a pair of linear varieties of equal dimension,
hence of degree $1 + 1 = 2$. For $\Gsf_2$ we have
the twisted cubic $\nu_3(\PP^1) \subset \PP^3$,
hence of degree $3$ ($\nu_d^* \OO(1) = \OO(d)$ for
the $d\Th$ Veronese embedding). Finally for
$\Bsf$ and $\Dsf$ we have the Segre-embedded
product of a line times a quadric. Computing
in the Chow ring, we have
$A(Y) = A(\PP^1)\otimes A(Q^{n-2})$
and in particular $A^{n-1}(Y) = A^1(\PP^1) \otimes A^{n-2}(Q^{n-2})$
generated by $[\pt]\otimes[\pt]$;
we denote by $h \in A^1(Q^{n-2})$ the class of
the hyperplane section so that $h^{n-2} = 2[\pt]$.
Now, the pullback of the hyperplane class by
the Segre embedding is $[\pt]\otimes 1 + 1 \otimes h$,
and we compute the degree as
$$([\pt]\otimes 1 + 1 \otimes h)^{n-1} = (n-1)([\pt]\otimes 1)(1\otimes h)^{n-2}
= 2(n-1)[\pt]\otimes[\pt]. $$
The degree of the Lagrangian Grassmannian in its Pl\"ucker
embedding is given as the
very first formula in~\cite{MR1937794},
leading in our case to $2^3 6! 2! / 3! 5! = 16$.
It remains to find the degrees in types $\Esf$ and $\Fsf_4$. The necessary
information can be extracted from the paper~\cite{MR672845},
to which we keep referring in what follows. More precisely, we are interested
in the answer to Problem 2.3 on p. 47 for the following pairs $(G, P_\alpha)$:
$$ (\Asf_5, P_3),\ (\Dsf_6, P_6),\ (\Esf_7,\ P_7) $$
with Bourbaki labelling as usual.
The value $42 = 9! / 2^2 3^3 4^2 5$ for the Grassmannian $\SL_6/P_3$ is given by the Example on p. 46 ($n=k=3$).
As explained on p. 51, the
pair $(\Dsf_6, P_6)$ may be replaced by $(\Bsf_5, P_5)$ and then
the value $286 = {15! 2! 4!}/{5! 7! 8! 9!}$ is given by Corollary 4.9 on p. 54 ($n=5$, $d=15$).
The value $13188$ for the Freudenthal variety $\Esf_7/P_7$ appears in the `Remarks'
on p. 57.
\end{proof}

The idea is now to produce a hypersurface $\E_Y \subset \LG(\g_{-1})$ from the sub-adjoint
variety $Y \subset \PP\g_{-1}$. This is a Lagrangian version of the usual Chow form,
assigning to a projective variety of codimension $k$ in $\PP^{N-1}$
a hypersurface in the Grassmannian $\Gr(k, N)$. Let us state its properties.
\begin{lemma}\label{lemChow}
Fix a standard symplectic form on $\C^{2n}$ and consider
the canonical $\Sp_n$-equivariant double fibration
\begin{equation}%\label{}
\xymatrix{
& \Fl^\iso(1,n,2n)\ar[dl]_p\ar[dr]^q & \\
 \PP^{2n-1} && \LG(n,2n)\, ,
}
\end{equation}
where $\Fl^\iso(1,n,2n)$ is the isotropic partial flag variety
embedded into $\PP^{2n-1}\times\LG(n,2n)$ as an incidence correspondence.
Let $Z \subset \PP^{2n-1}$ be
an irreducible subvariety of pure codimension $n$ and degree $d$. Then
$\E_Z = q(p^{-1}Z)$ is an irreducible hypersurface of degree $d$ in $\LG(n,2n)$.
\end{lemma}
\begin{definition}\label{defLagChow}
 We call $\E_Z$ the \emph{Lagrangian Chow transform} of $Z$.
More generally, if $Z_1,\dots,Z_r$ are several irreducible components,
all of codimension $n$, we set $\E_{\bigcup Z_i} = q(p^{-1}\bigcup Z_i) = \bigcup \E_{Z_i}$.
\end{definition}
\begin{proof}
The projection $p$ is a Zariski-locally
trivial bundle with fibres isomorphic to $\LG(n-1,2(n-1))$.
It then follows immediately that $\E_Z$ is irreducible. Furthermore,
$\dim \E_Z \le \dim p^{-1}Z = \dim Z + (n-1)n/2 = (n^2 + n - 2)/2$
so that $\codim \E_Z \ge 1$. Now, consider a
general line $\ell \subset \LG(n,2n)$: there is a canonical identification
$\ell \simeq \PP (K/K^\perp)$ where $K \subset \C^{2n}$ is a general $(n+1)$-dimensional
subspace such that the symplectic form restricted to $K$ has rank $1$.
Since $\PP K$ has complementary dimension to $Z$,
the intersection $\PP K \cap Z$ is nonempty.
So is then
$\ell \cap \E_Z$, proving $\codim \E_Z = 1$.
Now, $\ell$ being general, its intersection with $\E_Z$
is transverse and consists of $\deg\E_Z$ points.
That is, $\PP K^\perp \cap Z = \emptyset$ and we may identify
$\PP K \cap Z$ with $\ell \cap \E_Z$ scheme-theoretically: in particular,
$\PP K \cap Z$ is reduced and thus consists of $d$ points, proving $\deg\E_Z=d$.
\end{proof}

This way the sub-adjoint variety gives rise to a hypersurface in the Lagrangian
Grassmannian. Furthermore, since the construction is manifestly $\Sp_n$-invariant,
it follows that $\E_Y \subset \LG(\g_{-1})$ is $G_0$-invariant. Its degree
is equal to that of $Y$ (in the only case where $Y$ is not irreducible,
the transform of each component is a distinct hyperplane section of $\LG(\g_{-1})$,
so that their union has degree $2$), thus explaining the last row in
the table we have included in Theorem \ref{thm:invs}. In particular,
it is remarkable that, except for type $\Asf$ and $\Gsf_2$, the
`natural', geometric $\E_Y$ is of very high degree, as  compared to the
minimal degree hypersurfaces we have produced. On the other hand,
in types $\Asf$ and $\Gsf_2$, the components of the hypersurface $\E_Y$
are the unique invariant hypersurfaces and thus coincide with ours.

\section{Explicit invariant PDEs}\label{sec:pdes}

The key to recast our main result, that is Theorem \ref{thm:invs}, in the context of nonlinear PDEs is to choose suitable Darboux coordinates on the adjoint variety $X=G/P$. Recall also Lemma \ref{lem:bilag-stuff}.
\begin{proposition}\label{propDarboux}
 For any bi--Lagrangian decomposition
%\begin{equation*}%\label{eqBiLagGminus1}
 $L\oplus L^* = \g_{-1}$
%\end{equation*}
there exist complex Darboux coordinates
 \begin{equation}\label{eqDarbSuGiMenoUno}
  x^1,\ldots, x^n, u, u_1,\ldots, u_n,
\end{equation}
in a neighborhood of $o$, such that
%\begin{eqnarray*}
%L&=&g^{-1}\cdot\Span{\left.\frac{\partial}{\partial x^1}+u_1\frac{\partial}{\partial u}\right|_{gP},\ldots, \left.\frac{\partial}{\partial x^n}+u_n\frac{\partial}{\partial u}\right|_{gP}}\, ,\\
%\g_{-2} &=&g^{-1}\cdot\Span{\left.\frac{\partial}{\partial u}\right|_{gP} }\, \\
%L^*&=&g^{-1}\cdot\Span{\left.\frac{\partial}{\partial u_1}\right|_{gP},\ldots, \left.\frac{\partial}{\partial u_n}\right|_{gP}}\, ,
%\end{eqnarray*}
%for all $g\in G$.%, having identified $\g_-$ with $T_oM$.
$$
L=g^{-1}\cdot\Span{\left.D_{x^1}\right|_{gP},\ldots, \left.D_{x^n}\right|_{gP}} ,\,\,\,
\g_{-2} =g^{-1}\cdot\Span{\left.\frac{\partial}{\partial u}\right|_{gP}} ,\,\,\,
L^*=g^{-1}\cdot\Span{\left.\frac{\partial}{\partial u_1}\right|_{gP},\ldots, \left.\frac{\partial}{\partial u_n}\right|_{gP}}\, ,
$$
for all $g\in G$, where $D_{x^i}$ are the total derivatives (cf.  \eqref{eq.total}).
\end{proposition}
\begin{proof}
By restricting the exponential map
$ \g\ni g \mapsto \exp (g)\in G$
to $\g_-$, one obtains an (algebraic) isomorphism
$%\begin{equation*}%\label{}
\Psi:\g_-\longrightarrow U\subseteq X
$, %\end{equation*}
between the linear space $\g_-$ and an open neighborhood $U$ of the origin.
For $v \in\g_-$ denote by $\widehat{v}$ the  vector field on $U $ induced by $v$. Then we have that
\begin{equation}\label{eqFormulaMagicaJan}
\widehat{v }_{\Psi (w )}=T_w \Psi  \left(v -\frac{1}{2}[w ,v ]\right)\, ,
\end{equation}
for all $w \in\g_-$.
Formula \eqref{eqFormulaMagicaJan} follows directly from the Baker--Campbell--Hausdorff formula
\begin{equation*}
 e^{t \widehat{v }}\Psi (w ) =  e^{tv }e^w P = e^{tv +w +\frac{1}{2}[tv ,w ]}P=\Psi \left(w +tv +\frac{1}{2}t[v ,w ]\right)\, .
\end{equation*}
We can  now pull--back the contact distribution $\CC$ on $X $  to a contact distribution (denoted by the same symbol $\CC$) on $\g_{-}$, by setting
$%\begin{equation*}%\label{}
\CC_w :=(T_w \Psi )^{-1}(\CC_{\Psi (w )})
$. %\end{equation*}
Then   \eqref{eqFormulaMagicaJan} implies
\begin{equation}\label{eqFormulaMagicaJan2}
\CC_w =\left(\id -\frac{1}{2}\ad_w \right)\g_{-1}\, .
\end{equation}
%
%
%
%
%
%By definition
%\begin{equation}%\label{}
%d_{gP}F:T_{gP}\g_-\equiv\g_- \longrightarrow T_{gP}U\equiv T_{gP}M
%\end{equation}
%is a diffeomorphism. It is easy to verify that the contact distribution $\CC$ on $M$, pulled back via  $F$ onto $\g_-$ is given by
%\begin{equation}%\label{}
%\CC_v:=(\id+\ad_v)\g_{-1}\subset\g_{-}\equiv  T_{v}\g_{-}\, ,\quad \forall v\in\g_-\, .
%\end{equation}
%%
Fix now vectors $l^1,\ldots, l^n,r,l_1,\ldots, l_n$ such that
$
L=\Span{l^1,\ldots, l^n}$, $ \g_{-2}=\Span{r}$, $ L^*=\Span{l_1,\ldots, l_n}$.
Then from \eqref{eqFormulaMagicaJan2}  it follows that the vectors fields $D_i$ and $V^i$ on $\g_{-1}$ defined by
$$
\left.D_i\right|_w := \left(\id -\frac{1}{2}\ad_w \right)l_i\, ,\quad
\left.V^i\right|_w := \left(\id -\frac{1}{2}\ad_w \right)l^i\, ,
$$
for all $w \in\g_{-}$, form a basis of $\CC$.
The last step is to show that there are coordinates
 \begin{equation}\label{eqDarbSuGiMenoUnoPRE}
  \underline{x}^1,\ldots, \underline{x}^n, \underline{u}, \underline{u}_1,\ldots, \underline{u}_n,
\end{equation}
on  $\g_-$  such that
$$
 D_i =  \frac{\partial}{\partial \underline{x}^i}+\underline{u}_i\frac{\partial}{\partial \underline{u}} \, ,\quad
V^i =  \frac{\partial}{\partial \underline{u}^i}  \, .
$$
But this is true, if one sets
$
  \underline{x}^i:= l^i$, $
    \underline{u}:= r^\vee+l^il_i$, $
      \underline{u}_i:= l_i
$.
Then the desired coordinates \eqref{eqDarbSuGiMenoUno} are just the pull--backs via $\Psi $ of \eqref{eqDarbSuGiMenoUnoPRE}.
%\begin{eqnarray}
% \Span{\left.\frac{\partial}{\partial \underline{z}^1}+\underline{p}_1\frac{\partial}{\partial \underline{u}}\right|_{v},\ldots, \left.\frac{\partial}{\partial \underline{z}^n}+\underline{p}_n\frac{\partial}{\partial \underline{u}}\right|_{v}}&=& (\id+\ad_v)L\, ,\\
%  (\id+\ad_v)^{-1} \Span{\left.\frac{\partial}{\partial \underline{u}}\right|_{v} }&= & (\id+\ad_v)\g_{-2}\, ,\\
% \Span{\left.\frac{\partial}{\partial \underline{p}_1}\right|_{v},\ldots, \left.\frac{\partial}{\partial \underline{p}_n}\right|_{v}}&=& (\id+\ad_v)L^*\, ,
%\end{eqnarray}

%
%
%The claim then follows from the $G$--equivariance of $F$: indeed, $F$ can be used to `transfer' the above coordinates \eqref{eqDarbSuGiMenoUnoPRE} on $U$ and obtain  \eqref{eqDarbSuGiMenoUno}.
%%
\end{proof}

\subsection{The case $\Asf$}\label{ss:pdes-a}
Besides being  technically the simplest, this case is made interesting by the fact that the torus  $T$ has rank $2$ (see Remark \ref{remTorus}). The adjoint  manifold $X$ is the projectivised cotangent bundle $\p T^*\p^{n+1}$,
which is a  $\PGL(n+2)$--homogeneous contact manifold of dimension $2n+1$ (see Lemma \ref{lemTypeA}). The contact plane $\CC_o$ at the origin is the direct sum $\C^n\oplus\C^{n\ast}$, which happens to be bi--Lagrangian. Hence, we choose Darboux coordinates as in Proposition \ref{propDarboux}, in such a way that $\C^n$ is spanned by the `total derivatives' $D_{x^i}$ and $\C^{n\ast}$ is spanned with the `vertical vectors' $\partial_{u_i}$.\par
The fact that $\rk T=2$ is mirrored by the fact that the sub-adjoint
variety $Y$ has two irreducible components: $\p^{n-1}$ and $\p^{(n-1)\ast}$ (see Lemma \ref{lem:properties-grading}).
Accordingly, the Lagrangian Chow transform  $\E_Y$ of the \emph{whole} $Y$ is a hypersurface of degree 2, made by two irreducible components of degree 1, namely the Lagrangian Chow transforms $\E_{\p^{n-1}}$ and $\E_{\p^{(n-1)\ast}}$ of the corresponding irreducible pieces of $Y$ (see Definition \ref{defLagChow}). Hence,    $\E_{\p^{n-1}}$ and $\E_{\p^{(n-1)\ast}}$ are \emph{both} homogeneous $2\Nd$ order PDEs in type $\Asf_{n+1}$, of minimal degree.\par
In the Darboux coordinates provided by Proposition \ref{propDarboux}, these are precisely the parabolic   Monge--Amp\`ere equations, already discussed in \cite{MR2985508}. For instance, in order to compute $\E_{\p^{n-1}}$, we just observe that a Lagrangian $n$--plane
\begin{equation}\label{eqLagPlan}
L(u_{ij}):=\Span{D_{x^i}+u_{ij} \partial_{u_j}\mid i=1,\ldots n}
\end{equation}
intersects nontrivially $\C^n=\Span{D_{x^i} \mid i=1,\ldots n}$ if and only if $ \det (u_{ij})=0$. Hence, $\E_{\p^{n-1}}=\{\det (u_{ij})=0\}$. The other component $\E_{\p^{(n-1)\ast}}$ of $\E_Y$ cannot be written explicitly as a $2\Nd$ order PDE, since no Lagrangian $n$--plane nontrivially intersecting $\C^{n\ast}$  can be written in the form \eqref{eqLagPlan}.
Indeed, these Darboux coordinates
are adapted to the structure given by the bi-Lagrangian
splitting of the contact distribution (there is a class of
adapted coordinates, stable under $\PGL_{n+1}$ acting
locally as point transformations). However, one could consider
more `generic' Darboux coordinates and simultaneously
 express both $\E_{\p^{n-1}}$ and $\E_{\p^{(n-1)*}}$
as explicit PDEs.
%There is no Darboux coordinate system which allows to describe simultaneously both the components of $\E_Y$. In Mathematical Physics, the equation $\E_{\p^{(n-1)\ast}}$ is usually discarded as deemed unphysical (it would entail infinite--valued derivatives).\par
%
%Nevertheless,  after a total Legendre transform of Darboux coordinates (that is, basically interchanging $D_i$ with $V^i$), one gets again $\E_{\p^{(n-1)\ast}}=\{\det \|u_{ij}\|=0\}$.

\subsection{The case $\Bsf_3$}
We denote by  $Y\subset\p(\g_{-1})$  the sub--adjoint variety of $X$. In this section we compute  the Lagrangian Chow transform $\E_Y$ of $Y$ in the case   $G=\Bsf_3$, because  $\E_Y$ is precisely the minimal--degree homogeneous equation on $X$. Observe that, as a second--order PDEs, $\E_Y$ has $3$ independent variables ($n=3$), and as an algebraic hypersurface is of degree 4 in the minors of $(u_{ij})$ (recall Definition \ref{defDegreeGlob} (2)).

\smallskip
Before focusing on $G=\Bsf_3$, let us examine the general case when $G$ is of type $\Bsf$ or $\Dsf$.
Let
$%\begin{equation*}%\label{}
\CC=\Span{D_{x^1},\ldots, D_{x^n},\partial_{u_1},\ldots, \partial_{u_n}}
$ %\end{equation*}
be the contact distribution on $X$, and identify
$%\begin{equation*}%\label{}
\CC_o\equiv\g_{-1}\equiv \C^2\otimes\C^n=\Span{A\otimes e_1,\ldots, A\otimes e_n,B\otimes e_1,\ldots, B\otimes e_n}
$, %\end{equation*}
where
$
\C^2 = \Span{A,B}$ and $\C^n = \Span{e_1,\ldots,e_n}
$
Now recall that $\C^n $ is equipped with a metric $g$, and that the sub--adjoint variety $Y$ is the image of $\p^1\times \mathcal{N}_g$ in the Segre embedding
\begin{equation*}%\label{}
\p(\C^2)\times\p(\C^{n})\equiv\p^1\times \p^{n-1}\longrightarrow \p^{2n-1}=\p(\C^2\otimes\C^n)\, ,
\end{equation*}
where
$
 \mathcal{N}_g\subset\p(\C^n)
$
is the null variety of $g$ (see Subsection \ref{subProofTypeBD}).
We assume that $g$ is diagonal, i.e.,
$% \begin{equation*}%\label{}
g=\mathrm{diag}(\lambda_1,\ldots,\lambda_n)
$. %end{equation*}
Then $Y$ has dimension $n-1$ in $\p(\g_{-1})=\p^{2n-1}$ (see Subsection \ref{secRumianek}), and the $n$ quadratic equations
\begin{equation}\label{eq.var.D}
\rank
\left(
\begin{array}{cccc}
dx^1 & dx^2 & \ldots & dx^n
\\
du_1 & du_2 & \ldots & du_n
\end{array}
\right) \leq 1\,,\,\,
\sum_{i=1}^n\lambda_i(dx^i)^2 =0\,
\end{equation}
vanish on $Y$.
Indeed, the rank of the above $2\times n$ matrix is $\leq 1$ if and only if $n-1$ minors of rank two vanishes. Then, if an element $[v]$ of $\p^{2n-1}$ is in the image of the Segre variety, it will be also in $Y$ if and only if $v$ is null with respect to the metric
\begin{equation*}%\label{}
\left(\begin{array}{cc}0 & \mathrm{diag}(\lambda_1,\ldots,\lambda_n) \\\mathrm{diag}(\lambda_1,\ldots,\lambda_n) & 0\end{array}\right)
\end{equation*}
on $\g_{-1}$, induced from  the metric $g$.
Let us now specialise to the case $n=3$.
If  $\lambda_i=1$, $\forall\,i$, the system \eqref{eq.var.D} is
\begin{equation}\label{eq.var.D.3}
\left\{
\begin{array}{l}
x^1u_2-u_1x^2=0
\\
x^1u_3-u_1x^3=0
\\
x^1u_1+x^2u_2+x^3u_3=0\, .
\end{array}
\right.
\end{equation}
The variety $\widetilde{Y}\subset\p^5$, cut out by \eqref{eq.var.D.3},    contains by construction  the subajoint variety $Y$, but also other `parasitic components', namely
\begin{equation}\label{eq.var.D.3.plus}
\left\{
\begin{array}{l}
u_1=0
\\
u_2=0
\\
u_3=0
\end{array}
\right.
\,\,,\qquad
\left\{
\begin{array}{l}
x^1=0
\\
u_1=0
\\
x^2u_2+x^3u_3=0\, .
\end{array}
\right.
\end{equation}
The first one is a $\p^2$ inside $\p^5$, whereas the second is a quadric inside a $\p^3$ inside $\p^5$.

\smallskip
By computing the  equation   $\E_{\widetilde Y}$ associated with the variety   ${\widetilde Y}$ described by \eqref{eq.var.D.3}, we obtain
\begin{equation*}%\label{eq.E.X.D3}
\det(u_{ij})\cdot (u_{13}^2u_{22}+u_{12}^2u_{33}-2u_{12}u_{13}u_{23})
%\cdot(u_{12}u_{23}u_{11}-u_{13}u_{12}^2-u_{33}u_{23}u_{12}+u_{13}u_{23}^2)^4
\cdot F=0\, ,
\end{equation*}
where $F$ is a long expression of degree $6$ in second derivatives. According to the general theory of Lagrangian Chow transforms, $\E_{\widetilde Y}$ is composed of various irreducible components, only one of whose is the desired equation $\E_{  Y}$ (see Definition \ref{defLagChow}). It remains to establish which is which.\par
First, $\det(u_{ij})=0$ is the (Monge-Amp\`ere) equation associated to the first variety of \eqref{eq.var.D.3.plus} in the spirit of \cite{MR2985508}, and it cannot be $\E_{  Y}$ since it is of degree 1 in the Pl\"ucker coordinates. Second,   $u_{13}^2u_{22}+u_{12}^2u_{33}-2u_{12}u_{13}u_{23}=0$ is the equation associated to the second variety of \eqref{eq.var.D.3.plus}.
%Finally,   the equation $u_{12}u_{23}u_{11}-u_{13}u_{12}^2-u_{33}u_{23}u_{12}+u_{13}u_{23}^2=0$ is associated to another variety contained into $Y$, \textbf{which is needed computing [AGGIUSTARE UN PO']}. \par
%
%
%
Now we  obtain $\E_{  Y}$ in a less direct way, by computing the full ideal $\mathcal{I}(Y)$ of $Y$. As it turns out, $\mathcal{I}(Y)$ is generated by 6 elements, as opposed to the 3 equations appearing in \eqref{eq.var.D.3}. To this end, recall that, by the Cauchy decomposition formula,
\begin{eqnarray}
 S^2(\C^2\otimes\C^3)&=&  \ydiagram{2} (\C^2)\otimes  \ydiagram{2}  (\C^3)\nonumber\\
&&\oplus  \ydiagram{1,1} (\C^2)\otimes  \ydiagram{1,1}  (\C^3)\nonumber\\
&=& \left(S^2(\C^2)\otimes  S^2  (\C^3)\right)\oplus\C^3\, ,\label{eqSegrC2C3}
\end{eqnarray}
where the boxes denote the appropriate Schur functors.\par
Dual to the Segre embedding
\begin{eqnarray*}
\p(\C^2)\times\p(\C^3)\equiv\p^1\times\p^2 &\longrightarrow &\p(\C^2\otimes\C^3)\equiv\p^5\, ,\\
([v],[w]) &\longmapsto& [v\otimes w]\, ,
\end{eqnarray*}
there is the projection
\begin{equation}\label{eqProjDualeSegreC2C3}
 S^2(\C^2\otimes\C^3)\longrightarrow S^2(\C^2)\otimes  S^2  (\C^3)
\end{equation}
from the $21$-dimensional space $ S^2(\C^2\otimes\C^3)$ of quadratic forms on $\p^5$, to the $18$-dimensional space $S^2(\C^2)\otimes  S^2  (\C^3)$ of bi--homogeneous  forms on the product space $\p^1\times\p^2 $, of bi--degree $(2,2)$. The kernel of \eqref{eqProjDualeSegreC2C3} is precisely the space $\C^3$ appearing in \eqref{eqSegrC2C3}, which consists of the three quadrics cutting out the Segre in $\p^5$.

Fix coordinates $[A:B]$ on $\p^1$, and coordinates $[z^1:z^2:z^3]$ on $\p^2$.
Suppose that $\C^3$ is equipped with a non--degenerate quadratic form
\begin{equation}\label{eqMetrC3}
g=(z^1)^2+(z^2)^2+(z^3)^2\, ,
\end{equation}
and let
$% \begin{equation*}%\label{}
Q:=\{g=0\}\subset \p^2
$. %\end{equation*}
Accordingly, we can single out the trace--free part of the quadratic forms on $\C^3$, that is
$% \begin{equation*}%\label{}
S^2(\C^3)=S_0^2(\C^3)\oplus\Span{g}
$, %\end{equation*}
and thus \eqref{eqSegrC2C3} can be further split as
\begin{equation*}%\label{}
 S^2(\C^2\otimes\C^3)=   \left(\left(S^2(\C^2)\otimes  S_0^2  (\C^3)\right)\oplus\C^3\right)\oplus \left(  \left(S^2(\C^2)\otimes  \Span{g}\right)  \oplus\C^3 \right)\, .
\end{equation*}
The canonical projection
$%\begin{equation*}%\label{}
 S^2(\C^2\otimes\C^3)\longrightarrow  \left(S^2(\C^2)\otimes  S_0^2  (\C^3)\right)\oplus\C^3
$ %\end{equation*}
 is precisely the dual to the embedding
$%  \begin{equation*}%\label{}
Y:=\p^1\times Q\longrightarrow \p^5
$, % \end{equation*}
and its $6$-dimensional kernel, i.e.,
\begin{equation}\label{eqKer6D}
\left(S^2(\C^2)\otimes  \Span{g}\right)  \oplus\C^3\, ,
\end{equation}
is the space of quadrics cutting out the 2--fold $Y$ in $\p^5$.
From \eqref{eqMetrC3} it follows immediately that
\begin{align*}%\label{}
S^2(\C^2)\otimes  \Span{g}&= \Span{A^2, AB, B^2}\otimes\Span{g}
= \Span{A^2g, ABg, B^2g}\\
&= \Span{(Az^1)^2+(Az^2)^2+(Az^3)^2, Az^1Bz^1+Az^2Bz^2+Az^3Bz^3   ,(Bz^1)^2+(Bz^2)^2+(Bz^3)^2}\\
&= \Span{(x^1)^2+(x^2)^2+(x^3)^2,x^1u_1+x^2u_2+x^3u_3,u_1^2+u_2^2+u_3^2}\, ,
\end{align*}
having set
\begin{eqnarray*}
x^i := Az^i\, ,\quad u_i:= Bz^i\, .
\end{eqnarray*}
In these coordinates, the $\C^3$ appearing in \eqref{eqKer6D} is the $3$-dimensional space
$ %\begin{equation*}%\label{}
\Span{x^1u_2-x^2u_1,\, x^1u_3-x^3u_1,\, x^2u_3-x^3u_2}
$,  %\end{equation*}
so that the ideal
\begin{equation}\label{eqIdealXB3}
\mathcal{I}(Y)=\Span{ x^1u_2-x^2u_1,x^1u_3-x^3u_1,x^2u_3-x^3u_2,(x^1)^2+(x^2)^2+(x^3)^2,x^1u_1+x^2u_2+x^3u_3,u_1^2+u_2^2+u_3^2  }
\end{equation}
is generated by 6 elements, and this number cannot be reduced. Recall that $\deg Y=4$ in virtue of Lemma \ref{lemChow}.
Now we find a generator of the ideal of $ \E_Y:=\{  L\in \LL(\g_{-1})\mid L\cap Y\neq\emptyset\}$ (see Definition \ref{defLagChow}). \par
To this end we recall the standard coordinates $u_{ij}$ on   $ \LL(\g_{-1})$, introduced earlier in Subsection \ref{ss:from-pdes} (see also Lemma \ref{lem:bilag-stuff} (3)).
Indeed, a Lagrangian plane $L$ is locally given by the 3 equations
\begin{equation}\label{eqP1}
u_i = u_{ij} x^j\,, \quad i=1, 2, 3\, ,
\end{equation}
where we recall that $u_{ij}=u_{ji}$. Moreover, $L$ belongs to $\E_Y$ if and only if, by replacing $u_1$, $u_2$ and $u_3$ according to \eqref{eqP1}, in each of the six generators of the ideal  $\mathcal{I}(Y)$ given in \eqref{eqIdealXB3}, one obtains a compatible systems of polynomials in the three variables $x^1,x^2,x^3$. These are:
\begin{eqnarray*}
q_1&=&x^1 \left(u_{12} x^1+u_{22} x^2+u_{23} x^3\right)-x^2 \left(u_{11} x^1+u_{12} x^2+u_{13} x^3\right)\, ,\\ q_2&=& x^1
   \left(u_{13} x^1+u_{23} x^2+u_{33} x^3\right)-x^3 \left(u_{11} x^1+u_{12} x^2+u_{13} x^3\right)\, ,\\ q_3&=& x^2
   \left(u_{13} x^1+u_{23} x^2+u_{33} x^3\right)-x^3 \left(u_{12} x^1+u_{22} x^2+u_{23} x^3\right)\, ,\\ q_4&=& u_{11}
   \left(x^1\right)^2+2 u_{12} x^2 x^1+u_{22} \left(x^2\right)^2+u_{33} \left(x^3\right)^2+2 x^3 \left(u_{13}
   x^1+u_{23} x^2\right)\, ,\\ q_5&=&\left(x^1\right)^2+\left(x^2\right)^2+\left(x^3\right)^2\, ,\\ q_6&=&\left(u_{11} x^1+u_{12}
   x^2+u_{13} x^3\right){}^2+\left(u_{12} x^1+u_{22} x^2+u_{23} x^3\right){}^2+\left(u_{13} x^1+u_{23}
   x^2+u_{33} x^3\right){}^2\, .
\end{eqnarray*}
Here we identify polynomials in the $u_{ij}$ with
element of the affine coordinate ring of
the corresponding open subset of $\LG(\g_{-1})$ (see the proof of Lemma \ref{lemGeomLagrGrass}).
We furthermore invert $x^3$ to remove the
apex of the affine cone over $Y$, so
that we may now use elimination to intersect
the ideal generated by $q_1,\dots,q_6$ with the
ring $\C[x^3, 1 / x^3, u_{ij}]$.
Using  \texttt{Macaulay 2}, we obtain that the ideal of $\E_Y$
in $\C[u_{ij}]$ pulls back to
\begin{equation}%\label{}
 \Span{(x^3)^2\cdot F} \subset \C[x^3,1/x^3, u_{ij}]\, ,
\end{equation}
where
{\scriptsize
\begin{align*}
F=&4 u_{12}^6+u_{11}^2 u_{12}^4+12 u_{13}^2 u_{12}^4+u_{22}^2 u_{12}^4+12 u_{23}^2 u_{12}^4-8 u_{33}^2
   u_{12}^4-10 u_{11} u_{22} u_{12}^4+8 u_{11} u_{33} u_{12}^4+8 u_{22} u_{33} u_{12}^4-36 u_{11} u_{13}
   u_{23} u_{12}^3-36 u_{13} u_{22} u_{23} u_{12}^3
   \\&
   +72 u_{13} u_{23} u_{33} u_{12}^3+12 u_{13}^4 u_{12}^2+12
   u_{23}^4 u_{12}^2+4 u_{33}^4 u_{12}^2-2 u_{11} u_{22}^3 u_{12}^2-8 u_{11} u_{33}^3 u_{12}^2-8 u_{22}
   u_{33}^3 u_{12}^2+2 u_{11}^2 u_{13}^2 u_{12}^2+8 u_{11}^2 u_{22}^2 u_{12}^2+20 u_{13}^2 u_{22}^2
   u_{12}^2
   \\&
   +20 u_{11}^2 u_{23}^2 u_{12}^2-84 u_{13}^2 u_{23}^2 u_{12}^2+2 u_{22}^2 u_{23}^2 u_{12}^2-2 u_{11}
   u_{22} u_{23}^2 u_{12}^2+2 u_{11}^2 u_{33}^2 u_{12}^2+20 u_{13}^2 u_{33}^2 u_{12}^2+2 u_{22}^2 u_{33}^2
   u_{12}^2+20 u_{23}^2 u_{33}^2 u_{12}^2+20 u_{11} u_{22} u_{33}^2 u_{12}^2
   \\&
   -2 u_{11}^3 u_{22} u_{12}^2
   -2
   u_{11} u_{13}^2 u_{22} u_{12}^2+2 u_{11}^3 u_{33} u_{12}^2+2 u_{22}^3 u_{33} u_{12}^2-2 u_{11} u_{13}^2
   u_{33} u_{12}^2-10 u_{11} u_{22}^2 u_{33} u_{12}^2-38 u_{11} u_{23}^2 u_{33} u_{12}^2-2 u_{22} u_{23}^2
   u_{33} u_{12}^2
   \\&
   -10 u_{11}^2 u_{22} u_{33} u_{12}^2-38 u_{13}^2 u_{22} u_{33} u_{12}^2
   +72 u_{11} u_{13}
   u_{23}^3 u_{12}-36 u_{13} u_{22} u_{23}^3 u_{12}-8 u_{13} u_{23} u_{33}^3 u_{12}+12 u_{11} u_{13} u_{23}
   u_{33}^2 u_{12}+12 u_{13} u_{22} u_{23} u_{33}^2 u_{12}
   \\&
   -36 u_{11} u_{13}^3 u_{23} u_{12}-8 u_{13} u_{22}^3
   u_{23} u_{12}+12 u_{11} u_{13} u_{22}^2 u_{23} u_{12}
   -8 u_{11}^3 u_{13} u_{23} u_{12}+72 u_{13}^3 u_{22}
   u_{23} u_{12}+12 u_{11}^2 u_{13} u_{22} u_{23} u_{12}-36 u_{13} u_{23}^3 u_{33} u_{12}
   \\&
   -36 u_{13}^3 u_{23}
   u_{33} u_{12}+12 u_{13} u_{22}^2 u_{23} u_{33} u_{12}+12 u_{11}^2 u_{13} u_{23} u_{33} u_{12}-48 u_{11}
   u_{13} u_{22} u_{23} u_{33} u_{12}+4 u_{13}^6+4 u_{23}^6+u_{11}^2 u_{13}^4+u_{11}^2 u_{22}^4+4 u_{13}^2
   u_{22}^4
   \\&
   -8 u_{11}^2 u_{23}^4+12 u_{13}^2 u_{23}^4+u_{22}^2 u_{23}^4+8 u_{11} u_{22} u_{23}^4+u_{11}^2
   u_{33}^4+u_{22}^2 u_{33}^4-2 u_{11} u_{22} u_{33}^4-2 u_{11}^3 u_{22}^3-8 u_{11} u_{13}^2 u_{22}^3-2
   u_{11}^3 u_{33}^3-2 u_{22}^3 u_{33}^3-2 u_{11} u_{13}^2 u_{33}^3
   \\&
   +2 u_{11} u_{22}^2 u_{33}^3+2 u_{11}
   u_{23}^2 u_{33}^3-2 u_{22} u_{23}^2 u_{33}^3+2 u_{11}^2 u_{22} u_{33}^3+2 u_{13}^2 u_{22} u_{33}^3+u_{11}^4
   u_{22}^2-8 u_{13}^4 u_{22}^2+2 u_{11}^2 u_{13}^2 u_{22}^2+4 u_{11}^4 u_{23}^2+12 u_{13}^4 u_{23}^2+2 u_{11}
   u_{22}^3 u_{23}^2
   \\&
   +20 u_{11}^2 u_{13}^2 u_{23}^2+2 u_{11}^2 u_{22}^2 u_{23}^2+20 u_{13}^2 u_{22}^2
   u_{23}^2-8 u_{11}^3 u_{22} u_{23}^2-38 u_{11} u_{13}^2 u_{22} u_{23}^2+u_{11}^4 u_{33}^2+u_{13}^4
   u_{33}^2+u_{22}^4 u_{33}^2+u_{23}^4 u_{33}^2+2 u_{11} u_{22}^3 u_{33}^2+8 u_{11}^2 u_{13}^2 u_{33}^2
   \\&
   -6
   u_{11}^2 u_{22}^2 u_{33}^2+2 u_{13}^2 u_{22}^2 u_{33}^2+2 u_{11}^2 u_{23}^2 u_{33}^2+2 u_{13}^2 u_{23}^2
   u_{33}^2+8 u_{22}^2 u_{23}^2 u_{33}^2-10 u_{11} u_{22} u_{23}^2 u_{33}^2+2 u_{11}^3 u_{22} u_{33}^2-10
   u_{11} u_{13}^2 u_{22} u_{33}^2+8 u_{11} u_{13}^4 u_{22}
   \\&
   +2 u_{11}^3 u_{13}^2 u_{22}-10 u_{11} u_{13}^4
   u_{33}-2 u_{11} u_{22}^4 u_{33}+8 u_{11} u_{23}^4 u_{33}-10 u_{22} u_{23}^4 u_{33}+2 u_{11}^2 u_{22}^3
   u_{33}-8 u_{13}^2 u_{22}^3 u_{33}-2 u_{11}^3 u_{13}^2 u_{33}+2 u_{11}^3 u_{22}^2 u_{33}
   \\&
   +20 u_{11} u_{13}^2
   u_{22}^2 u_{33}-8 u_{11}^3 u_{23}^2 u_{33}-2 u_{22}^3 u_{23}^2 u_{33}-2 u_{11} u_{13}^2 u_{23}^2 u_{33}-10
   u_{11} u_{22}^2 u_{23}^2 u_{33}+20 u_{11}^2 u_{22} u_{23}^2 u_{33}-2 u_{13}^2 u_{22} u_{23}^2 u_{33}-2
   u_{11}^4 u_{22} u_{33}
   \\&
   +8 u_{13}^4 u_{22} u_{33}-10 u_{11}^2 u_{13}^2 u_{22} u_{33}.
\end{align*}
}
%It should be proved that $c$ is of degree 4 in the minors of the $3\times 3$ symmetric matrix $(u_{ij})$.

\subsection{The case $\Dsf_\bullet$}
The next case, in the  $\Bsf$--$\Dsf$ series, when the `exceptionally simple PDE' $\E_Y$ is not anymore the minimal--degree one, is $\Dsf_4$, which we describe here. To this end, we provide a construction of the minimal--degree equation for all groups of type $\Dsf$. We fix a basis
\begin{equation}\label{eqBasisC2}
\C^2=\Span{\alpha,\beta}\, ,
\end{equation}
from which it follows the bi--Lagrangian decomposition
\begin{equation}\label{eqBiLagDecC2Cn}
\C^2\otimes\C^n=(\Span{\alpha}\otimes\C^n)\oplus (\Span{\beta}\otimes\C^n)
\end{equation}
of the $2n$--dimensional symplectic space and irreducible $\SL_2\times\SO_n$--module $\C^2\otimes\C^n$.
Observe that there are obvious identifications
\begin{eqnarray*}
\Lambda^n(\C^2\otimes\C^n)&\equiv&\bigoplus_{i=0}^n\left(  \Lambda^i(\Span{\alpha}\otimes\C^n) \otimes \Lambda^{n-i}(\Span{\beta}\otimes\C^n)  \right)\equiv\bigoplus_{i=0}^n\Span{\alpha^i\beta^{n-i}}\otimes \Lambda^i(\C^n)\otimes  \Lambda^{n-i}(\C^n)\\
&\equiv&\bigoplus_{i=0}^n\Span{\alpha^i\beta^{n-i}}\otimes \Lambda^i(\C^n)^{\otimes\, 2}  \, ,
\end{eqnarray*}
such that the Pl\"ucker embedding space $\Lambda_0^n(\C^2\otimes\C^n)$ of $\C^2\otimes\C^n$ can be identified with
\begin{equation*}%\label{}
\Lambda_0^n(\C^2\otimes\C^n)=\bigoplus_{i=0}^n\Span{\alpha^i\beta^{n-i}}\otimes S^2_0 \Lambda^i(\C^n)\, .
\end{equation*}
The notation, as well as the decomposition itself, are \emph{different} than
the one given in \eqref{eqDecPluckBD}. In particular, here $S^2_0 \Lambda^i\C^n$
denotes the kernel of the natural $\SL_n$-invariant map
$ \Lambda^i \C^n \otimes \Lambda^i \C^n \to \Lambda^{i-1}\C^n \otimes \Lambda^{i+1}\C^n $
%induced by the $\GL_n$-equivariant inclusions into $(\C^n)^{\otimes 2i}$ and
%corresponding projections given by suitable symmetrisers.
given on decomposable elements by
$$
(v_1\wedge\dots\wedge v_i) \otimes (w_1\wedge\dots\wedge w_i) \mapsto
\sum_{j=1}^i (-1)^{i+j} (v_1 \wedge\dots \wedge\widehat{v_j}\wedge\dots\wedge v_i) \otimes
(v_j \wedge w_1\wedge\dots\wedge w_i).
$$
It is precisely the $\SL_n$-irreducible summand in
$S^2 \Lambda^i\C^n$ whose highest weight is twice the $i$-fundamental
weight (for $i<n$).
The Lagrangian $n$--planes in a favourable position with respect to the splitting \eqref{eqBiLagDecC2Cn} are labeled by symmetric $n\times n$ matrices $
U=\big(u_{ij}\big)
$.
Indeed, if $U$ is understood as a map from $\Span{\alpha}\otimes\C^n$ to $\Span{\beta}\otimes\C^n$, then its graph is a Lagrangian subspace $L(U)$ nondegenerately projecting over the first space (see  Lemma \ref{lem:bilag-stuff} (3)). It is convenient now to introduce the natural extension
$%  \begin{equation*}%\label{}
U^{\bullet}:\Lambda^\bullet(\Span{\alpha}\otimes\C^n)\longrightarrow \Lambda^\bullet(\Span{\beta}\otimes\C^n)
$ %\end{equation*}
of $U$ to the exterior algebra, and its restrictions
$% \begin{equation*}%\label{}
U^{(i)}:\Lambda^i(\Span{\alpha}\otimes\C^n)\longrightarrow \Lambda^i(\Span{\beta}\otimes\C^n)
$ %\end{equation*}
to the corresponding $i\Th$ degree pieces. By Poincar\'e duality, we also have
$% \begin{equation*}%\label{}
U^{(i)}\in  \Span{\alpha^i\beta^{n-i}}\otimes S^2\Lambda^i(\C^n) $, % \end{equation*}
and in fact $U^{(i)}$ always lies in the $\SL_n$-irreducible subspace
$% \begin{equation*}%\label{}
S^2_0\Lambda^i(\C^n)\subseteq S^2\Lambda^i(\C^n)
$. %\end{equation*}
Pl\"ucker--embedding $L(U)$ into $\p (\Lambda_0^n(\C^2\otimes\C^n)) $ means taking the `volume' of  $L(U)$ , viz.
\begin{equation}\label{eqVolLU}
\det(L(U))=\left[   \sum_{i=0}^n \alpha^i\beta^{n-i}  U^{(i)}  \right]\, .
\end{equation}
But now we can use the projection
$% \begin{equation*}%\label{}
\pi:\Lambda_0^n(\C^2\otimes\C^n)\longrightarrow S^n\C^2
 $ % \end{equation*}
to map the representative of \eqref{eqVolLU}
into
\begin{equation}\label{eqPiLU}
\pi\left(   \sum_{i=0}^n \alpha^i\beta^{n-i}  U^{(i)}  \right)=   \sum_{i=0}^n \alpha^i\beta^{n-i} \tr U^{(i)}  \in S^n\C^2\, ,
\end{equation}
where we use the $\SO_n$-invariant
quadratic form on $\C^n$. It remains to observe that $S^n\C^2$ is equipped with a quadratic form, whose matrix in the standard basis
$%\begin{equation*}%\label{}
\alpha^n,\alpha^{n-1}\beta,\alpha^{n-2}\beta^2\,\ldots, \alpha\beta^{n-1},\beta^n
$, %\end{equation*}
is the anti--diagonal one,  with entries
$ c_0, c_1, \dots, c_{n-1}, c_n, c_{n-1}, \dots, c_1, c_0$, where
$$c_k = (-1)^k {n \choose k}. $$
Thus, evaluating this form on \eqref{eqPiLU}, one gets a quadratic expression
\begin{equation*}%\label{}
F(U):=\sum_{i=0}^n (-1)^i {n \choose i} \tr U^{(i)} \tr U^{(n-i)}
\end{equation*}
in the minors of the $n\times n$ symmetric matrix $U$, i.e., a $2\Nd$ order nonlinear PDE of `hyperquadric section' type. For example, %\textbf{(I now that there are some coefficients missing/wrong)},
for $n=4$ we have
\begin{equation}\label{eqD4}
\frac{1}{2}F(U)= \det U - 4\tr (U)\tr (U^\#) + 3\tr(U^{(2)})^2\, ,
\end{equation}
having denoted by $U^\#$ the cofactor matrix of $U$. Observe that \eqref{eqD4} is indeed quadratic, because $\det U$ is multiplied by the `minor of order zero', i.e., by 1 (cf. \eqref{eqASTERIX}).
\subsection{The case $\Gsf_2$} %\textbf{[TO BE DONE -- GIOVANNI IS DOING IT]}
In the case of $G=\Gsf_2$, the contact grading \eqref{eqContGrad} reads
\begin{equation*}%\label{eqContGradG2}
 \g = \C \oplus \underset{\CC_o=\g{-1}}{\underbrace{S^3\C^2}} \oplus \underset{\g_0}{\underbrace{\sll(2)\oplus\C} }\oplus S^3\C^* \oplus \C^*\, ,
\end{equation*}
where the semi-simple part $\sll(2)$ of $\g_0=\gl(2)$ has been spelt out.  Recall that the standard $\sll(2)$--module structure on $S^3\C^2$ is precisely the one induced from the bracket with $\g_0$ (see Subsection \ref{subContGrad}). The sub--adjoint variety $Y\subset\p\CC_0$ coincides with the unique closed  $\sll(2)$--orbit $\p^1=\p\C^2\subset\p S^3\C^2$, which is made of rank--one elements (see Subsection \ref{secRumianek}). In other words, $Y$ is  the twisted cubic in $\p^3$ and hence the minimal--degree $2\Nd$ order PDE on the adjoint contact manifold $M$ of $\Gsf_2$ is the Lagrangian Chow transform $\E_Y$ of the field of twisted cubics on $M$.\par
In order to write down $\E_Y$ in Darboux coordinates, we choose a bi--Lagrangian decomposition of $S^3\C^2$ and then use Proposition \ref{propDarboux}. The (conformally) unique symplectic form on $S^3\C^2$ is the one induced by the (conformally) unique symplectic structure on $\C^2$, which in turn correspond to the choice of a volume form on $\C^2$. By using  the same basis \eqref{eqBasisC2} as before, we obtain a bi--Lagrangian decomposition
\begin{equation*}%\label{}
S^3\C^2=\Span{\alpha^3,\alpha^2\beta,\alpha\beta^2,\beta^3}=\underset{L}{\underbrace{\Span{ \alpha^3,\alpha^2\beta}}}\oplus \underset{L^*}{\underbrace{\Span{ \alpha\beta^2,\beta^3}}} =
\Span{x^1,x^2}\oplus\Span{-3u_2,u_1}\cong S^3\C^{2\ast}\, .
\end{equation*}
The ideal $\mathcal{I}_Y\subset S^\bullet (S^3\C^{2\ast})$  is generated by three elements of degree two. Indeed, there is an exact sequence
\begin{equation*}%\label{}
0\longrightarrow  \mathcal{I}_Y\cap S^2 (S^3\C^{2\ast})  \longrightarrow  S^2 (S^3\C^{2\ast})  \longrightarrow   S^6\C^{2\ast}  \longrightarrow 0
\end{equation*}
of $\sll(2)$--modules, decomposing the 10--dimensional  $S^2 (S^3\C^{2\ast})$ into    irreducible representations. Parametrically,    $Y=\{[ (t\alpha+s\beta)^3 \mid (t:s)\in\p^1]$, that is,
$% \begin{equation*}%\label{}
Y=\{[ t^3 : t^2s : s^3 : -3ts^2   ]\mid (t:s)\in\p^1\}
$, %\end{equation*}
in the coordinates $[x^1:x^2:u_1:u_2]$.
Then  it is easy to check that $\mathcal{I}_Y$ is generated by
$3x^1u_1 + x^2u_2$, $x^1u_2 + 3(x^2)^2$, $9x^2u_1 - u_2^2$.
It remains to apply the Lagrangian Chow--form and discover that
\begin{equation*}%\label{}
\E_Y=\{27u_{11}^2 - u_{12}^2 u_{22}^2 + u_{11} u_{22}^3 + 16 u_{12}^3 - 18 u_{11}u_{12}u_{22}=0\}\, .
\end{equation*}

\section*{Acknowledgements}
The authors thank Dennis The and Boris Kruglikov for their
insightful remarks on a previous version of this work.
The research of all four authors has been partially supported by the project
 ``FIR (Futuro in Ricerca) 2013 -- Geometria delle equazioni differenziali''. Dmitri Alekseevsky thanks the Dipartimento di Scienze Matematiche ``G.L. Lagrange'' of the Politecnico of Torino for hospitality.
Dmitri Alekseevsky carried this work at IITP and is supported by an RNF grant (project n.14-50-00150). He also thanks the Dipartimento di Scienze Matematiche ``G.L. Lagrange'' of the Politecnico of Torino for hospitality.
Giovanni Moreno and Gianni Manno have been also partially supported by
the Marie Sk\l odowska--Curie fellowship SEP--210182301 ``GEOGRAL". Giovanni Moreno has been also partially founded by the  Polish National Science Centre grant
under the contract number 2016/22/M/ST1/00542,  by the University of Salerno, and by the project P201/12/G028 of the Czech Republic Grant Agency (GA \v{C}R).  Gianni Manno and Giovanni Moreno are members of G.N.S.A.G.A of I.N.d.A.M.

\bibliographystyle{plainnat}
\bibliography{BibUniver}
\bigskip

\end{document}